\documentclass{scrartcl}
\usepackage[top=2cm, bottom=2.5cm, left=2.5cm, right=2.5cm]{geometry}

\usepackage[T1]{fontenc}
\usepackage[utf8]{inputenc}
\usepackage{amsmath, amsthm, amssymb}
\usepackage{mathtools}
\usepackage{dsfont}
\usepackage{graphicx}
\usepackage[format=plain]{caption}
\usepackage{capt-of}
\usepackage{subcaption}
\usepackage{bm}
\usepackage{csquotes}
\usepackage{xcolor}
\usepackage{dlfltxbcodetips} 
\usepackage[hidelinks]{hyperref}
\usepackage[section]{algorithm}
\usepackage{algpseudocode}
\usepackage{nicefrac}
\usepackage{array, makecell} 
\usepackage{array, booktabs, tabularx} 
\newcolumntype{C}{>{\centering\arraybackslash}X}
\usepackage{siunitx}
\DeclareSIUnit\calorie{cal}
\DeclareSIUnit\kcal{kcal}

\renewcommand{\d}{\,\mathrm{d}} 
\newcommand{\e}{\mathrm{e}} 
\newcommand{\N}{\mathbb{N}}  

\newcommand{\R}{\mathbb{R}}

\newcommand{\diag}{\operatorname{diag}}
\newcommand{\supp}{\operatorname{supp}}

\newcommand{\diam}{\operatorname{diam}}
\newcommand{\dom}{\operatorname{dom}}

\newcommand{\sgn}{\operatorname{sgn}}
\newcommand{\conv}{\operatorname{conv}}

\newcommand{\eps}{\varepsilon}
\renewcommand{\phi}{\varphi}
\newcommand{\ul}{\underline}
\newcommand{\ol}{\overline}
\newcommand{\tr}{\operatorname{tr}}
\newcommand{\X}{\mathcal{X}}

\newcommand{\Iineq}{I_{\mathrm{ineq}}}
\newcommand{\Ieq}{I_{\mathrm{eq}}}
\newcommand{\Iineqc}{I_{\mathrm{ineq,1}}}
\newcommand{\Iineqnc}{I_{\mathrm{ineq,2}}}
\newcommand{\Ic}{I^{\mathrm{sensi}}_{\mathrm{c}}}
\newcommand{\Inc}{I^{\mathrm{sensi}}_{\mathrm{nc}}}
\newcommand{\Xifin}{\Xi_{\mathrm{fin}}}
\newcommand{\Xifeas}{\Xi_{\mathrm{feas}}}

\newcommand{\Vc}{V_{\mathrm{c}}}
\newcommand{\Vnc}{V_{\mathrm{nc}}}
\newcommand{\fc}{f_{\mathrm{c}}}
\newcommand{\fnc}{f_{\mathrm{nc}}}
\newcommand{\Psiineq}{\Psi_{\mathrm{ineq}}}
\newcommand{\Psieq}{\Psi_{\mathrm{eq}}}
\newcommand{\gequ}{g_{\mathrm{eq}}}
\newcommand{\CD}{\mathrm{CD}}
\newcommand{\B}{\mathcal{B}}
\newcommand{\dimx}{d_\mathrm{x}}
\newcommand{\dimtheta}{d_\mathrm{\theta}}
\newcommand{\dimy}{d_\mathrm{y}}
\newcommand{\psd}{\mathrm{psd}}
\newcommand{\pd}{\mathrm{pd}}
\newcommand{\indfct}{\mathrm{I}}
\newcommand{\signedmsre}{\Xi(X,\R)}
\newcommand{\nonnegmsre}{\Xi(X,[0,\infty))}
\newcommand{\St}{\mathcal{S}}
\newcommand{\T}{\mathcal{T}}
\newcommand{\tfin}{t_{\mathrm{f}}}
\newcommand{\tm}{t_{\mathrm{m}}}
\newcommand{\constr}{\mathrm{constr}}
\newcommand{\ret}{\mathrm{roi}}

\newcommand{\footremember}[2]{%
   \footnote{#2}
    \newcounter{#1}
    \setcounter{#1}{\value{footnote}}%
}
\newcommand{\footrecall}[1]{%
    \footnotemark[\value{#1}]%
} 

\numberwithin{equation}{section}

\newtheorem{thm}{Theorem}[section]
\newtheorem{cor}[thm]{Corollary}

\newtheorem{lm}[thm]{Lemma}
\newtheorem{cond}[thm]{Condition}

\theoremstyle{definition} \newtheorem{algo}[thm]{Algorithm}
\theoremstyle{definition} \newtheorem{ex}[thm]{Example}
\theoremstyle{definition}

\title{An adaptive discretization algorithm for locally optimal experimental design with constraints}

\author{Jochen Schmid\footremember{contribution}{J.~Schmid and P.~Seufert contributed equally to this work.}, Philipp Seufert\footrecall{contribution} , Jan Schwientek, Tobias Seidel,\\ Karl-Heinz K\"ufer
\\\small Fraunhofer Institute for Industrial Mathematics (ITWM), 67663 Kaiserslautern, Germany\\
\small jochen.schmid@itwm.fraunhofer.de}  

\date{}

\begin{document}

\maketitle

\begin{abstract}
\small{ \noindent 
We develop a novel iterative algorithm for locally optimal experimental design under constraints, like budget or performance constraints. It is an adaptive discretization algorithm. In every iteration, a discretized version of the constrained-design problem is solved and then the discretization is adaptively refined by adding an approximate violator of a suitable sufficient $\eps$-optimality condition for the current design. 
We prove that 
with $\eps = 0$, our algorithm converges to an optimal design and that with $\eps > 0$, our algorithm  finitely terminates at an $\eps$-optimal design. 
Compared to the 
existing algorithms on constrained experimental design, our algorithm comes with considerably less computational effort 
because the nonlinear subproblems in our algorithm have a smaller dimension and have to be solved only approximately and only in selected iterations (typically the last few). 
Additionally, our algorithm covers a considerably larger class of constraints. 
We demonstrate the good convergence properties of the  algorithm on experimental design problems from chemical engineering that feature time and yield constraints. 
}
\end{abstract}


\section{Introduction}

A large variety of scientific and industrial applications nowadays rely heavily on 
mathematical models for prediction and optimization. In particular, this is true for chemical process engineering. It is crucial, of course, that the employed models are sufficiently 
precise. In order to obtain precise models, one typically performs experiments and collects data, to which the model is then fitted. Yet, performing experiments is usually time- and cost-intensive. It is therefore important to find experimental plans, or designs, which are maximally informative, that is, lead to the highest possible model precision. 
Coming up with such maximally informative experimental designs is the core task of unconstrained optimal experimental design, and there is a tremendous amount of literature on such unconstrained design problems. See the textbooks~\cite{Silvey, Pu, AtDoTo, FeLe, PrPa} and the survey paper~\cite{DuAt25} for a general overview of the theory, algorithms, and applications of unconstrained design and~\cite{Fe, Wy70, At73, SiTiTo78, Bo86, Yu10, Yu11, YaBiTa13, HaFiRi20, ScSeBo} for 
selected algorithms to solve unconstrained design problems. In mathematical terms, unconstrained locally optimal design problems can be cast in the form 
\begin{align} \label{eq:unconstrained-design-problem-intro}
	\min_{\xi\in\Xi(X)} \Psi_0(\xi).
\end{align}
In this optimization problem, $X$ is the so-called design space, that is, the set of possible individual experiments, $\Xi(X)$ is the set of 
approximate (as opposed to exact) experimental designs on $X$, that is, probability measures on the Borel sigma-algebra $\B_X$ of $X$, and $\Psi_0: \Xi(X) \to \R \cup \{\infty\}$ is the so-called design criterion, that is, a function that measures the inverse information content of the design $\xi$ about the model parameters~\cite{FeLe, PrPa, Pu, ScSeBo, Silvey}. 
Virtually all relevant design criteria are convex functions of the design and therefore the unconstrained design problems~\eqref{eq:unconstrained-design-problem-intro} are 
infinite-dimensional convex optimization problems. 
\smallskip

In many applications, however, finding experimental designs that are just maximally informative is not enough. Instead, one often has to additionally ensure that the proposed experimental design satisfies certain budget or performance constraints. Some typical examples are that the time or material costs required for carrying out the experiments in the proposed experimental design have to be less than some given thresholds or that the purity of a substance produced in the experiments of the proposed experimental design is high enough such that it can be re-used after the experiment. Coming up with experimental designs that are maximally informative and at the same time respect the relevant constraints is the task of constrained optimal experimental design. As is demonstrated in~\cite{CoFe95, FeLe}, the constraints in many applications are convex or even affine functions of the design and therefore a large variety of 
constrained locally optimal design problems can be cast in the form
\begin{align}
	\label{eq:constrained-design-problem-intro}
	\min_{\xi\in\Xi(X)} \Psi_0(\xi)
	\quad \text{s.t.} \quad 
	\Psi_i(\xi) \le 0 \text{ for all } i \in \Iineq \text{ and } \Psi_i(\xi) = 0 \text{ for all } i \in \Ieq
\end{align}
with convex inequality constraint functions $\Psi_i: \Xi(X) \to \R \cup \{\infty\}$ for $i \in \Iineq$ and affine equality constraint functions $\Psi_i: \Xi(X) \to \R$ for $i \in \Ieq$. 
Compared to unconstrained design, however, constrained design problems~\eqref{eq:constrained-design-problem-intro} -- and especially dedicated algorithms for their numerical solution -- are much less studied in the literature. 
Specifically, there are some few papers proposing algorithms for constrained design problems~\eqref{eq:constrained-design-problem-intro} with relatively general, finitely many constraints~\cite{CoFe95, Ga86, MoZu02}. 
Additionally, there are a few papers investigating constrained design problems~\eqref{eq:constrained-design-problem-intro} with so-called bound constraints~\cite{Fe89, MoZu04, Pr04, HeLe20} (that is, infinitely many affine inequality constraints of the specific form $\xi(E) \le \omega(E)$ for all $E \in \B_X$ with a reference design $\omega$) or with so-called marginal constraints~\cite{CoTh80, HuHs93, CoFe95}  (that is, infinitely many affine equality constraints of the specific form $\xi(X \times F) = \omega(F)$ for all $F \in \B_Y$, where $Y$ is the set of variables that are not controllable in experiments and $\omega$ is a reference design).
\smallskip

In the present paper, we develop a novel algorithm to solve constrained design problems of the form~\eqref{eq:constrained-design-problem-intro} with finitely many constraints. It comes with considerably less computational effort and, at the same time, covers a considerably broader setting than the algorithms from~\cite{CoFe95, Ga86, MoZu02}. 
%
Similarly to the 
unconstrained-design algorithms from~\cite{YaBiTa13, ScSeBo}, the algorithm proposed here is an adaptive discretization algorithm. In every iteration, it proceeds in the following two steps. 
In the first step, a saddle point of a discretized version of the constrained design problem~\eqref{eq:constrained-design-problem-intro} is computed, namely, a saddle point $(\xi^k,\lambda^k)$ of the finite-dimensional 
convex optimization problem 
\begin{align}
	\label{eq:discretized-constrained-design-problem-intro}
	\min_{\xi\in\Xi(X^k)} \Psi_0(\xi)
	\quad \text{s.t.} \quad 
	\Psi_i(\xi) \le 0 \text{ for all } i \in \Iineq \text{ and } \Psi_i(\xi) = 0 \text{ for all } i \in \Ieq
\end{align}
that arises  from~\eqref{eq:constrained-design-problem-intro} by replacing the infinite design space $X$ by a suitable discretization $X^k$. 
In the second step, the discretization $X^k$ 
is adaptively refined to 
\begin{align}
	X^{k+1} = X^k \cup \{x^k\}
\end{align}
by adding one point, namely an appropriate approximate violator $x^k$ of a suitable sufficient $\eps$-optimality condition at $(\xi^k,\lambda^k)$. In essence, 
this sufficient $\eps$-optimality condition consists in a lower-bound condition on the sensitivity function $\psi^L$ of the Lagrangian $L$  of~\eqref{eq:constrained-design-problem-intro}, namely
\begin{align}
	\label{eq:eps-optimality-condition-intro}
	\psi^L(\xi^k,\lambda^k,x) \ge -\eps \qquad (x \in X).
\end{align}
In analogy to unconstrained design, this sensitivity function, or more precisely, the function $x \mapsto \psi^L(\xi^k,\lambda^k,x)$, is the integral kernel of the directional derivatives of $L$ at $(\xi^k,\lambda^k)$ in the directions $\eta - \xi^k$. An approximate violator of~\eqref{eq:eps-optimality-condition-intro}, in turn, is a point $x^k \in X$ with
\begin{align}
	\label{eq:approximate-violator-intro}
	\psi^L(\xi^k,\lambda^k,x^k) < -\eps + \ul{\delta}_k 
\end{align} 
for some $\ul{\delta}_k \ge 0$. 
%
In the first version of our algorithm -- called the special algorithm in the following -- the point $x^k$ is computed 
as a $\ul{\delta}_k$-approximate solution of the sensitivity minimization problem
\begin{align}
	\label{eq:sensi-minimization-problem-intro}
	\min_{x \in X} \psi^L(\xi^k,\lambda^k,x)
\end{align}
in every iteration. It should be noticed that this typically 
requires global-optimization techniques because $x \mapsto \psi^L(\xi^k,\lambda^k,x)$ is non-convex and multimodal in most applications. 
%
In the second version of our algorithm -- called the general algorithm in the following -- we therefore restrict 
the computation of $\ul{\delta}_k$-approximate solutions
of the sensitivity minimization problems~\eqref{eq:sensi-minimization-problem-intro} to as few iterations as possible. 
Specifically, the general algorithm first performs a local search for an arbitrary violator $x^k$ of~\eqref{eq:eps-optimality-condition-intro} in every iteration, and 
only in iterations where the search finds no violator 
is 
the point $x^k$ computed as before as a $\ul{\delta}_k$-approximate solution of~\eqref{eq:sensi-minimization-problem-intro}. 
\smallskip

We will show that the special algorithm with optimality tolerance $\eps = 0$ in~\eqref{eq:eps-optimality-condition-intro} converges to an optimal design for~\eqref{eq:constrained-design-problem-intro}. We will further show that the general algorithm with optimality tolerance $\eps > 0$ in~\eqref{eq:eps-optimality-condition-intro} terminates after finitely many iterations at an $\eps$-optimal design for~\eqref{eq:constrained-design-problem-intro}.
%
All we need to establish these convergence and termination results are standard lower semicontinuity and standard directional differentiability assumptions on the objective and constraint functions, plus some mild constraint qualification in terms of the initial discretization $X^0$ of $X$. Specifically, we need that the initial discretized design problem, that is, problem~\eqref{eq:discretized-constrained-design-problem-intro} with $k = 0$, is strictly feasible and that the image of the initial design-measure space $\Xi(X^0)$ under the equality constraints is not one-sided w.r.t.~$0$ (in the sense that it does not lie completely on one side of a hyperplane in $\R^{\Ieq}$ through $0$).
\smallskip

Compared to the algorithms developed in~\cite{CoFe95, Ga86, MoZu02}, the algorithmic approach proposed here offers several clear 
advantages.
%
A first advantage of our algorithmic approach over~\cite{CoFe95, Ga86, MoZu02} is its generality.  Indeed, it covers general convex constraints and also allows for non-continuous sensitivity functions.
In contrast, the algorithms 
formulated in~\cite{CoFe95, MoZu02} cover only affine (inequality) constraints~\cite{CoFe95} or, respectively, only affine equality constraints~\cite{MoZu02}. Additionally, while the algorithm from~\cite{Ga86} in principle covers also convex inequality constraints, the boundedness of the multiplier sequence that is necessary for the convergence of the algorithm is established, again, only in the special case of affine inequality constraints. And finally, the sensitivity functions of the affine constraints in~\cite{CoFe95, Ga86, MoZu02} are always assumed to be continuous, while in our approach the affine inequality constraints 
are allowed to be non-continuous. 
%
A second important advantage of our algorithms over~\cite{CoFe95, Ga86, MoZu02} is that they require substantially less computational effort in connection with their nonlinear subproblems.
Indeed, the nonlinear sensitivity minimization problems~\eqref{eq:sensi-minimization-problem-intro} of our algorithms have a considerably smaller dimension  than the nonlinear subproblems of~\cite{CoFe95} and~\cite{MoZu02} (namely, the dimension $\dimx$ of $X$ versus $(m+1) (\dimx + 1)$ and $(m+1) \dimx$, respectively, with $m$ being the number of the considered affine constraints in~\cite{CoFe95} or \cite{MoZu02}, respectively).
Additionally, the algorithms from~\cite{CoFe95, Ga86, MoZu02} all require exact solutions of their nonlinear subproblems in every iteration, whereas our algorithms require only approximate solutions of their nonlinear subproblems~\eqref{eq:sensi-minimization-problem-intro} and, in the case of our general algorithm, only in iterations where the local search for violators of~\eqref{eq:eps-optimality-condition-intro} fails. 
%
A third advantage of our algorithms over~\cite{CoFe95, Ga86, MoZu02} are their good convergence properties.
Indeed, our algorithms typically progress quickly towards the optimal criterion value (in the case of the special algorithm) or an $\eps$-optimal criterion value (in the case of the general algorithm). 
In essence, this can be explained by the fact that in our algorithms the design is re-optimized in every iteration. In fact, this is in complete analogy to the unconstrained-design algorithms from~\cite{YaBiTa13, ScSeBo} which converge at least linearly by virtue of~\cite{ScSeBo}. In contrast, the algorithms from~\cite{CoFe95, Ga86, MoZu02} are generally prone to slow convergence because in the unconstrained special case they all reduce to the classical vertex-direction algorithm~\cite{Fe, Wy70}. And the latter algorithm, in turn, generally converges slowly~\cite{PrPa} -- in fact, at most sublinearly~\cite{CaCu68, Ja13} in general -- and is prone to zig-zagging vertex-direction 
changes near the optimum~\cite{LaJa15}. 
\smallskip


We illustrate the good convergence properties and the wide applicability of our algorithm on  constrained design problems with two kinds of models. In the first set of design problems, the underlying model is an explicitly defined stationary model, namely an exponential growth model, while in the second set of design problems, the underlying model is an implicitly defined dynamic model from chemical engineering, namely a reaction kinetics model. As we will see, our  algorithm converges fast in all these examples, requiring only few iterations until termination at an $\eps$-optimal design with $\eps = 10^{-3}$. 
Also, we will see that the concrete verification of the regularity conditions of our convergence and termination results is straightforward.

\section{Setting and preliminaries}

In this section, we formally introduce the constrained  design problems considered in this paper, establish their solvability as well as sufficient and necessary optimality conditions for their solutions.
%
%
In this paper, we consider constrained locally optimal experimental design problems of the form
\begin{align} \label{eq:COED(X)}
	\min_{\xi \in \Xi(X)} \Psi_0(\xi) \quad \text{s.t.} \quad \Psi_i(\xi) \le 0 \text{ for all } i \in \Iineq \text{ and } \Psi_i(\xi) = 0 \text{ for all } i \in \Ieq,
\end{align}
where $\Psi_0: \Xi(X) \to \R \cup \{\infty\}$ is an appropriate design criterion and $\Psi_i: \Xi(X) \to \R \cup \{\infty\}$ for $i \in \Iineq$ and $\Psi_i: \Xi(X) \to \R$ for $i \in \Ieq$ are finitely many inequality or equality constraint functions. See~\cite{BuSc24, FeLe, PrPa, ScSeBo} for detailed motivations of this formulation of the experimental-design objective and constraints. 
With the abbreviation
\begin{align} \label{eq:Xifeas(X)-definition}
	\Xifeas(X) := 
	\{\xi \in \Xi(X): \Psi_i(\xi) \le 0 \text{ for all } i \in \Iineq \text{ and } \Psi_i(\xi) = 0 \text{ for all } i \in \Ieq \},
\end{align}
the considered constrained experimental design problem~\eqref{eq:COED(X)} can be recast in the following more concise form:
\begin{align} \label{eq:COED(X)-short}
\min_{\xi \in \Xifeas (X)} \Psi_0(\xi).
\end{align}
In the following, we will sometimes refer to this design problem on $X$ by the symbol $\CD(X)$ for the sake of brevity.
A design $\xi^* \in \Xi(X)$ is called an \emph{$\eps$-optimal solution of~\eqref{eq:COED(X)}} or an \emph{$\eps$-optimal design for~\eqref{eq:COED(X)}}  for some given $\eps \in [0,\infty)$ iff 
\begin{align}
	\xi^* \in \Xifeas(X)
	\qquad \text{and} \qquad
	\Psi_0(\xi^*) \le \inf_{\xi\in\Xifeas(X)} \Psi_0(\xi) + \eps.
\end{align} 
Clearly, a design $\xi^* \in \Xi(X)$ is an \emph{optimal solution of~\eqref{eq:COED(X)}} iff it is an $\eps$-optimal solution with $\eps := 0$. 
As usual, the domain of an extended real-valued function $\Psi: \Xi(X) \to \R \cup \{\infty\}$ is defined to be the 
set 
\begin{align}
\dom \Psi := \{\xi \in \Xi(X): \Psi(\xi) < \infty\}
\end{align}
of input points for which $\Psi$ takes a finite value. In the entire paper, the sets of positive semidefinite or, respectively, positive definite $d \times d$ matrices will be denoted by
\begin{align}
	\R_\psd^{d \times d}
	\qquad \text{and} \qquad
	\R_\pd^{d \times d}.
\end{align}

\subsection{Solvability of the constrained design problem}

In this section, we establish the existence of optimal solutions for~\eqref{eq:COED(X)} and, in particular, of optimal solutions with finitely many support points. In order to do so, we need to impose some very basic continuity and semicontinuity assumptions. In the entire paper, (semi)continuity of a function $\Psi: \Xi(X) \to \R \cup \{\infty\}$ is always understood w.r.t.~the weak topology on $\Xi(X)$, that is, the topology induced by convergence in distribution. It is well-known by Prohorov's theorem (Theorem~II.6.4 in~\cite{Parthasarathy}) that for every compact metric space $X$, the space $\Xi(X)$ endowed with the weak topology is a compact metrizable space itself. In particular, $\Xi(X)$ is sequentially compact then. See, for instance, Lemma 2.10 of~\cite{ScSeBo} for details. 

\begin{cond} \label{cond:objective-and-constraint-fcts}
$X$ is a non-empty compact metric space, $\Iineq$, $\Ieq$ are disjoint finite sets, $I := \Iineq \cup \Ieq \not\ni 0$, and the design criterion and the constraint functions satisfy the following properties: 
\begin{itemize}
\item[(i)] $\Psi_0: \Xi(X) \to \R \cup \{\infty\}$ is a lower semicontinuous mapping 
\item[(ii)] $\Psi_i: \Xi(X) \to \R \cup \{\infty\}$ for every $i \in \Iineq$ is a lower semicontinuous mapping 
\item[(iii)] $\Psi_i: \Xi(X) \to \R$ for every $i \in \Ieq$ is a continuous mapping
\end{itemize}
\end{cond}

With these continuity and lower semicontinuity assumptions and the compactness of $\Xi(X)$, the existence of optimal solutions for~\eqref{eq:COED(X)} follows by standard arguments. 

\begin{thm} \label{thm:solvability}
If Condition~\ref{cond:objective-and-constraint-fcts} is satisfied and $\Xifeas(X) \ne \emptyset$, then $\Xifeas(X)$ is compact and the design problem~\eqref{eq:COED(X)} has a solution. 
\end{thm}

\begin{proof}
It immediately follows by the lower semicontinuity of the inequality constraint functions and the continuity of the equality constraint functions that $\Xifeas(X)$ is a closed subset of the compact metric space $\Xi(X)$ and hence $\Xifeas(X)$ is compact as well. It further follows by the lower semicontinuity of the objective function $\Psi_0$ and by the compactness of $\Xifeas(X)$ that every minimizing sequence $(\xi^n)$ for~\eqref{eq:COED(X)} accumulates at a solution $\xi^*$ of~\eqref{eq:COED(X)}. In particular, \eqref{eq:COED(X)} does have a solution, as desired.
\end{proof}

Apart from the continuity and lower semicontinuity assumptions above, we will also need some convexity, affinity and  directional differentiability assumptions. In the spirit of~\cite{FeLe}, 
we call the restriction $\Psi|_{\dom\Psi}$ of a function $\Psi: \Xi(X) \to \R \cup \{\infty\}$ \emph{directionally differentiable w.r.t.~$\Xi(X)$} iff there exists a function $\psi: \dom \Psi \times X \to \R$ such that for every $\xi \in \dom \Psi$ the function $\psi(\xi,\cdot)$ is $\eta$-integrable for every $\eta \in \Xi(X)$ and 
\begin{align} \label{eq:directional-differentiability-def}
\frac{\Psi(\xi+\alpha(\eta-\xi)) - \Psi(\xi)}{\alpha} \longrightarrow \int_X \psi(\xi,x)\d\eta(x) \qquad (\alpha \searrow 0).
\end{align} 
In other words, $\Psi|_{\dom\Psi}$ is directionally differentiable w.r.t.~$\Xi(X)$ iff for every point $\xi \in \dom \Psi$ and every $\eta \in \Xi(X)$, the directional derivative $\partial_{\eta-\xi}\Psi(\xi)$ of $\Psi$ at $\xi$ in the direction $\eta-\xi$ exists in $\R$ in the usual sense~\cite{Jahn} and is of the specific integral form
\begin{align} \label{eq:directional-derivative-integral-form}
\partial_{\eta-\xi}\Psi(\xi) = \int_X \psi(\xi,x)\d\eta(x).
\end{align}
It should be noted that if $\Psi|_{\dom\Psi}$ is directionally differentiable w.r.t.~$\Xi(X)$, then there is only one function~$\psi$ satisfying the requirements from the above definition. (In order to see this, just consider \eqref{eq:directional-differentiability-def} with the point measures $\eta := \delta_x$ for $x \in X$.) Similarly to~\cite{FeLe}, we then call this function $\psi$ the \emph{sensitivity function of $\Psi|_{\dom\Psi}$ w.r.t.~$\Xi(X)$}. At first glance, the additional requirement that the directional derivatives $\Psi|_{\dom\Psi}$ be of the specific integral form~\eqref{eq:directional-derivative-integral-form} might seem restrictive. Yet, for almost all standard design criteria $\Psi_0$ and constraint functions $\Psi_i$, this additional requirement is satisfied. See Lemma~\ref{lm:sufficient-conditions-for-standard-assumptions} below, in conjunction with Proposition~2.12 and Example~2.23  of~\cite{ScSeBo}, for instance. 

\begin{cond} \label{cond:sensi-fcts}
\begin{itemize}
\item[(i)] $\Psi_0$ is convex and $\Psi_0|_{\dom \Psi_0}$ is directionally differentiable w.r.t.~$\Xi(X)$ with a continuous sensitivity function $\psi_0: \dom \Psi_0 \times X \to \R$
\item[(ii)] $\Psi_i$ is convex for every $i \in \Iineq$ and $\Psi_i|_{\dom \Psi_i}$ 
is directionally differentiable w.r.t.~$\Xi(X)$ with a sensitivity function $\psi_i: \dom \Psi_i \times X \to \R$ satisfying one of the following properties:
\begin{itemize}
\item[(a)] $\psi_i$ is continuous, or
\item[(b)] there exists a bounded measurable function $g_i: X \to \R$ such that $\psi_i(\xi,x) = g_i(x)-\Psi_i(\xi)$ for all $(\xi,x) \in \dom \Psi_i \times X$ 
\end{itemize}
\item[(iii)] $\Psi_i$ is affine for every $i \in \Ieq$.
\end{itemize}
\end{cond}

We now give simple sufficient conditions for the Conditions~\ref{cond:objective-and-constraint-fcts} and \ref{cond:sensi-fcts} to be satisfied. In all our applications, the design criterion $\Psi_0$ will be of the particular form $\Psi_0(\xi) = \Phi(M(\xi))$ with $M(\xi) = \int_X m(x)\d\xi(x)$ 
and with a suitable design criterion $\Phi$ on the set $\R^{d \times d}_{\mathrm{psd}}$ of positive semidefinite matrices. 

\begin{cond} \label{cond:sufficient-conditions-for-standard-assumptions}
$X$ is a compact metric space and $\Iineqc, \Iineqnc, \Ieq$ are pairwise disjoint finite sets and $I := \Iineqc \cup \Iineqnc \cup \Ieq\not\ni 0$. Also, $\Psi_i: \Xi(X) \to \R \cup \{\infty\}$ for $i \in \{0\} \cup I$ are mappings satisfying the following properties: 
\begin{itemize}
\item[(i)] $\Psi_i(\xi) = \Phi_i(M_i(\xi))$ for $i \in \{0\} \cup \Iineqc$ and $\xi \in \Xi(X)$, where $\Phi_i: \R^{d_i \times d_i}_{\mathrm{psd}} \to \R \cup \{\infty\}$ is a lower semicontinuous convex mapping and
\begin{align}
M_i(\xi) := \int_X m_i(x) \d\xi(x) \qquad (\xi \in \Xi(X))
\end{align}
with some $m_i \in C(X,\R^{d_i \times d_i}_{\mathrm{psd}})$ and $d_i \in \N$
\item[(ii)] $\Psi_i(\xi) = \int_X g_i(x)\d\xi(x)$ for $i \in \Iineqnc$ and $\xi \in \Xi(X)$, where $g_i:X \to \R$ is a bounded and lower semicontinuous function
\item[(iii)] $\Psi_i(\xi) = \int_X g_i(x)\d\xi(x)$ for $i \in \Ieq$ and $\xi \in \Xi(X)$, where $g_i:X \to \R$ is a continuous function.
\end{itemize}
\end{cond}

\begin{lm} \label{lm:sufficient-conditions-for-standard-assumptions}
If Condition~\ref{cond:sufficient-conditions-for-standard-assumptions} is satisfied, then Condition~\ref{cond:objective-and-constraint-fcts} is satisfied with $\Iineq := \Iineqc \cup \Iineqnc$. If, in addition, $\Phi_i|_{\dom\Phi_i}$ is continuously differentiable and $\dom\Phi_i$ is open in $\R^{d_i\times d_i}$ for every $i \in \{0\} \cup \Iineqc$, then Condition~\ref{cond:sensi-fcts} is satisfied as well and the sensitivity functions $\psi_i$ are given by
\begin{gather}
\psi_i(\xi,x) = D\Phi_i(M_i(\xi))(m_i(x) - M_i(\xi)) \qquad (i \in \{0\} \cup \Iineqc)
\label{eq:sensi-fct-formula-for-Iineqc}\\
\psi_i(\xi,x) = g_i(x) - \Psi_i(\xi) \qquad (i \in \Iineqnc \cup \Ieq)
\end{gather}
for all $(\xi,x) \in \dom\Psi_i \times X$, where $D\Phi_i(M)$ denotes the derivative of $\Phi_i|_{\dom\Phi_i}$ at $M$. Additionally, $\dom\Psi_i = \Xi(X)$ for all $i \in \Iineqnc \cup \Ieq$. 
\end{lm}

\begin{proof}
In order to see the directional differentiability of $\Psi_i$ for $i \in \Iineqc$ with a continuous sensitivity function of the form~\eqref{eq:sensi-fct-formula-for-Iineqc}, we can invoke Lemma~2.22 of~\cite{ScSeBo}. In order to see the lower semicontinuity of $\Psi_i$ for $i \in \Iineqnc$, we can use Alexandrov's portmanteau theorem 
in the version of~\cite{DupuisEllis} (Theorem~A.3.12), for instance. All other assertions are straightforward to verify.
\end{proof}

While the design criterion $\Psi_0$ and the inequality constraint functions $\Psi_i$ need not necessarily be of the form from Condition~\ref{cond:sufficient-conditions-for-standard-assumptions}~(i) and~(ii), the equality constraint functions $\Psi_i$ are necessarily of the form from Condition~\ref{cond:sufficient-conditions-for-standard-assumptions}~(iii). 

\begin{lm} \label{lm:equality-constr-fcts-have-continuous-sensi-fct}
If Conditions~\ref{cond:objective-and-constraint-fcts} and~\ref{cond:sensi-fcts} are satisfied, then for every $i \in \Ieq$ there exists a $g_i \in C(X,\R)$ such that
\begin{align} \label{eq:equality-constr-fcts-have-continuous-sensi-fct-1}
\Psi_i(\xi) = \int_X g_i(x) \d\xi(x) \qquad (\xi \in \Xi(X)).
\end{align}
In particular, $\Psi_i$ for every $i \in \Ieq$ is directionally differentiable w.r.t.~$\Xi(X)$ with a continuous sensitivity function $\psi_i: \Xi(X) \times X \to \R$ given by
\begin{align} \label{eq:equality-constr-fcts-have-continuous-sensi-fct-2}
\psi_i(\xi,x) = g_i(x) - \Psi_i(\xi) \qquad ((\xi,x) \in \Xi(X) \times X).
\end{align}
\end{lm}

\begin{proof}
Suppose that Conditions~\ref{cond:objective-and-constraint-fcts} and~\ref{cond:sensi-fcts} are satisfied and define the functions $g_i$ by
\begin{align} \label{eq:equality-constr-fcts-have-continuous-sensi-fct-def}
g_i(x) := \Psi_i(\delta_x) \qquad (x \in X)
\end{align}
for $i \in \Ieq$. 
It is then straightforward to verify from the assumed continuity of $\Psi_i: \Xi(X) \to \R$ (Condition~\ref{cond:objective-and-constraint-fcts}~(iii)) that $g_i \in C(X,\R)$ for every $i \in \Ieq$. It is further clear by the assumed affinity of $\Psi_i$ (Condition~\ref{cond:sensi-fcts}~(iii)) that 
\begin{align} \label{eq:equality-constr-fcts-have-continuous-sensi-fct-finite-support}
\Psi_i(\xi) = \Psi_i\bigg(\sum_{x\in\supp\xi} \lambda_x \delta_x\bigg) = \sum_{x\in\supp\xi} \lambda_x \Psi_i(\delta_x) = \int_X g_i(x)\d\xi(x)
\end{align} 
holds true for every $\xi = \sum_{x\in\supp\xi} \lambda_x \delta_x \in \Xi(X)$ with finite support. 
Since the probability measures of finite support are dense in $\Xi(X)$ (Theorem~II.6.3 of~\cite{Parthasarathy}), the equality~\eqref{eq:equality-constr-fcts-have-continuous-sensi-fct-finite-support} extends to arbitrary $\xi \in \Xi(X)$ by continuity of $\Psi_i$. So, \eqref{eq:equality-constr-fcts-have-continuous-sensi-fct-1} is proven, and with \eqref{eq:equality-constr-fcts-have-continuous-sensi-fct-1} at hand, \eqref{eq:equality-constr-fcts-have-continuous-sensi-fct-2} and  the continuity of $\psi_i$ are trivial to prove.
\end{proof}

We now show that for design criteria and constraint functions of the form from Condition~\ref{cond:sufficient-conditions-for-standard-assumptions} (with the non-continuous $g_i$ additionally assumed to take only finitely many values), there exists an optimal solution to~\eqref{eq:COED(X)} with finite support. We also give an upper bound $d$ on the number of support points, see~\eqref{eq:opt-design-with-finite-support} below.

\begin{lm} \label{lm:opt-design-with-finite-support}
Suppose that $f_1: X \to V_1$ is a continuous function and $f_2: X \to V_2$ is a (componentwise) lower semicontinuous function with finitely many values, where $X$ is a compact metric space and $V_1$ is a $d_1$-dimensional vector space over $\R$ and $V_2 = \R^{d_2}$ with some $d_1, d_2 \in \N$. Then for every $\xi^* \in \Xi(X)$, there exists a design $\eta^* \in \Xi(X)$ with at most $d_1 + d_2 + 1$ support points such that
\begin{align} \label{eq:opt-design-with-finite-support-lm}
\int_X f_1 \d\eta^* = \int_X f_1 \d\xi^* 
\qquad \text{and} \qquad
\int_X f_2 \d\eta^* \le \int_X f_2 \d\xi^*. 
\end{align}
\end{lm}

\begin{proof}
In the entire proof, we write $f(x) := (f_1(x),f_2(x))$ and $V := V_1 \times V_2$. Since $f_1$ is bounded and $f_2$ takes only finitely many values, for every $n \in \N$ there exists a finite cover $\mathcal{C}_n$ of the range $f(X)$ that is of the form
\begin{align} \label{eq:opt-design-with-finite-support-lm-step1}
\mathcal{C}_n = \{C_1 \times \{c_2\}: C_1 \in \mathcal{C}_{1n} \text{ and } c_2 \in f_2(X) \}
\end{align}
and that consists of pairwise disjoint measurable sets whose diameters go to $0$ as $n \to \infty$, that is, 
\begin{align} \label{eq:opt-design-with-finite-support-lm-step2}
\sup_{C_1 \in \mathcal{C}_{1n}} \diam(C_1) \longrightarrow 0 \qquad (n\to\infty).
\end{align}
In order to see this, we can take an arbitrary coordinate mapping (linear isomorphism) $\kappa: V_1 \to \R^{d_1}$, choose $R \in (0,\infty)$ so large that $\kappa(f_1(X)) \subset [-R,R)^{d_1}$ and define $\mathcal{C}_{1n}$ to be the following finite cover of $f_1(X)$:
\begin{align*}
\mathcal{C}_{1n} := 
\Big\{\kappa^{-1}(Q): Q = \bigtimes_{i=1}^{d_1} [-R+\frac{k_i}{2^n} 2R, -R+\frac{k_i+1}{2^n} 2R) \text{ with } k_1, \dots, k_{d_1} \in \{0,\dots,2^n-1\} \Big\}.
\end{align*}
With the covers $\mathcal{C}_n$ at hand, we now define the functions $f^n: X \to V$ by
\begin{align} \label{eq:opt-design-with-finite-support-lm-step3}
f^n := \sum_{C \in \mathcal{C}_n \text{ s.t. } f^{-1}(C) \ne \emptyset} f(x^n_C) \cdot \mathrm{I}_{f^{-1}(C)},
\end{align}
where $x^n_C$ for every $C \in \mathcal{C}_n$ with $f^{-1}(C) \ne \emptyset$ is chosen such that
\begin{align} \label{eq:opt-design-with-finite-support-lm-step4}
f(x_C^n) \in C. 
\end{align}
Since $f$ is measurable by assumption and $\mathcal{C}_n$ consists of finitely many measurable sets by construction, $f^n$ is a measurable function with finitely many values. Since, moreover, the sets in $\mathcal{C}_n$ are pairwise disjoint and cover $f(X)$, for every $x \in X$ and every $n \in \N$ there exists a unique element $C := C_{f(x)}^n \in \mathcal{C}_n$ such that $C \ni f(x)$, in short:
\begin{align} \label{eq:opt-design-with-finite-support-lm-step5}
C_{f(x)}^n \in \mathcal{C}_n \qquad \text{and} \qquad C_{f(x)}^n \ni f(x).
\end{align}
Consequently, we have
\begin{align} \label{eq:opt-design-with-finite-support-lm-step6}
f^n(x) = f(x^n_{C_{f(x)}^n}) \qquad (x \in X, n \in \N)
\end{align}
and, in particular,
\begin{align} \label{eq:opt-design-with-finite-support-lm-step7}
f^n(X) \subset f(X) \qquad (n\in\N) 
\qquad \text{and} \qquad
\sup_{x\in X,n\in \N} |f_1^n(x)| \le \sup_{x\in X}|f_1(x)| < \infty.
\end{align}
Combining \eqref{eq:opt-design-with-finite-support-lm-step4} \eqref{eq:opt-design-with-finite-support-lm-step5} and \eqref{eq:opt-design-with-finite-support-lm-step6}, we further conclude that for given $x \in X$ and $n \in \N$ the values $f^n(x)$ and $f(x)$ both belong to $C_{f(x)}^n = C_{1f(x)}^n \times \{c_{2f(x)}^n\}$. And therefore, for every $x \in X$ we have
\begin{align}
|f_1^n(x)-f_1(x)| &\le \diam C_{1f(x)}^n \longrightarrow 0 \qquad (n\to\infty)
\label{eq:opt-design-with-finite-support-lm-step8}\\
f_2^n(x) &= c_{2f(x)}^n = f_2(x) \qquad (n \in \N)
\label{eq:opt-design-with-finite-support-lm-step9}
\end{align} 
by virtue of~\eqref{eq:opt-design-with-finite-support-lm-step2}.
With these preparations at hand, we can now conclude the proof of the lemma. As an immediate consequence of the finiteness of $f^n(X)$ and of~(\ref{eq:opt-design-with-finite-support-lm-step7}.a), we see that
\begin{align} \label{eq:opt-design-with-finite-support-lm-step10}
\int_X f^n \d\xi^* \in \conv f(X) \qquad (n\in\N).
\end{align}
And so, by Carath\'{e}odory's theorem, we conclude that $\int_X f^n \d\xi^*$ for every $n \in \N$ can be written as a convex combination of at most $d := \dim V + 1 = d_1 + d_2 + 1$ elements from $f(X)$, that is, there are $x_1^n,\dots, x_d^n \in X$ and $\lambda_1^n, \dots, \lambda_d^n \in [0,1]$ with
\begin{align} \label{eq:opt-design-with-finite-support-lm-step11}
\int_X f^n \d\xi^* = \sum_{i=1}^{d} \lambda_i^n f(x_i^n)
\qquad \text{and} \qquad 
\sum_{i=1}^n \lambda_i^n = 1 \qquad (n\in\N). 
\end{align}
As a consequence of~(\ref{eq:opt-design-with-finite-support-lm-step7}.b) and \eqref{eq:opt-design-with-finite-support-lm-step8} and \eqref{eq:opt-design-with-finite-support-lm-step9}, we further see by the dominated convergence theorem that
\begin{align} \label{eq:opt-design-with-finite-support-lm-step12}
\int_X f_1 \d\xi^* = \lim_{n\to\infty} \int_X f_1^n \d\xi^*
\qquad \text{and} \qquad
\int_X f_2 \d\xi^* = \int_X f_2^n \d\xi^* \qquad (n\in\N).
\end{align}
Combining now~\eqref{eq:opt-design-with-finite-support-lm-step11} with~\eqref{eq:opt-design-with-finite-support-lm-step12} and passing to a subsequence $(n_k)$ such that the sequences $(x_i^{n_k})$ and $(\lambda_i^{n_k})$ become convergent, we conclude by the continuity of $f_1$ and the (componentwise) lower semicontinuity of $f_2$ that
\begin{align} \label{eq:opt-design-with-finite-support-lm-step13}
\int_X f_1 \d\xi^* = \sum_{i=1}^{d} \lambda_i^* f_1(x_i^*)
\qquad \text{and} \qquad
\int_X f_2 \d\xi^* \ge \sum_{i=1}^{d} \lambda_i^* f_2(x_i^*),
\end{align}
where $\lambda_i^*$ and $x_i^*$ denote the limits of $(x_i^{n_k})$ and $(\lambda_i^{n_k})$, of course. 
Setting $\eta^* := \sum_{i=1}^d \lambda_i^* \delta_{x_i^*}$, the assertion of the lemma is then clear by~(\ref{eq:opt-design-with-finite-support-lm-step11}.b) and~\eqref{eq:opt-design-with-finite-support-lm-step13}. 
\end{proof}

\begin{cor} \label{cor:opt-design-with-finite-support}
Suppose that Condition~\ref{cond:sufficient-conditions-for-standard-assumptions} is satisfied and that, in addition, the functions $g_i$ for $i \in \Iineqnc$ take only finitely many values. Suppose further that $\Xifeas(X) \ne \emptyset$ and let $\eps \in [0,\infty)$. Then the design problem~\eqref{eq:COED(X)} has an $\eps$-optimal  solution with at most 
\begin{align} \label{eq:opt-design-with-finite-support}
d := \sum_{i\in I_1} \frac{d_i(d_i+1)}{2} + |I_2| + |I_3| + 1
\end{align} 
support points, where $I_1$ is a minimal subset of $\{0\} \cup \Iineqc$ with $\{m_i: i \in I_1\} = \{m_i: i \in \{0\} \cup \Iineqc\}$ and where $I_2 := \Iineqnc$ and $I_3 := \Ieq$. 
\end{cor}


\begin{proof}
Since $\Xifeas(X) \ne \emptyset$, there exists an $\eps$-optimal solution $\xi^* \in \Xifeas(X)$ to~\eqref{eq:COED(X)} by virtue of Theorem~\ref{thm:solvability} and Lemma~\ref{lm:sufficient-conditions-for-standard-assumptions}. 
We now define $\Vc := \bigtimes_{i\in I_1} \R^{d_i \times d_i}_{\mathrm{sym}} \times \R^{I_3}$ and $\Vnc := \R^{I_2}$ and let 
\begin{align} \label{eq:opt-design-with-finite-support-step1}
\fc(x) := \big( (m_i(x))_{i\in I_1}, (g_i(x))_{i\in I_3})\big) \in \Vc
\qquad \text{and} \qquad
\fnc(x) := (g_i(x))_{i\in I_2} \in \Vnc
\end{align}
for $x \in X$. It is then clear by our assumptions 
that $\fc: X \to \Vc$ is a continuous function and that $\fnc: X \to \Vnc$ is a lower semicontinuous function with finitely many values. So, by Lemma~\ref{lm:opt-design-with-finite-support}, there exists a design $\eta^* \in \Xi(X)$ with at most
\begin{align} \label{eq:opt-design-with-finite-support-step2}
\dim \Vc + \dim \Vnc + 1 = \sum_{i\in I_1} \frac{d_i(d_i+1)}{2} + |I_2| + |I_3| + 1 = d
\end{align}
support points and with
\begin{align} \label{eq:opt-design-with-finite-support-step3}
\int_X \fc \d\eta^* = \int_X \fc \d\xi^*
\qquad \text{and} \qquad 
\int_X \fnc \d\eta^* \le \int_X \fnc \d\xi^*.
\end{align}
As a consequence of~\eqref{eq:opt-design-with-finite-support-step1} and~\eqref{eq:opt-design-with-finite-support-step3}, we conclude that
\begin{gather}
M_i(\eta^*) = M_i(\xi^*) \qquad (i \in \{0\} \cup \Iineqc),
\qquad
\int_X g_i \d\eta^* \le \int_X g_i \d\xi^* \qquad (i \in \Iineqnc), 
\label{eq:opt-design-with-finite-support-step4}\\
\int_X g_i \d\eta^* = \int_X g_i \d\xi^* \qquad (i \in \Ieq),
\label{eq:opt-design-with-finite-support-step5}
\end{gather}
where for the first equality we used that for every $i \in \{0\} \cup \Iineqc$ there is an $i_0 \in I_1$ with $m_i = m_{i_0}$. Since $\xi^* \in \Xifeas(X)$ as an $\eps$-optimal solution is in particular feasible for ~\eqref{eq:COED(X)}, it follows from \eqref{eq:opt-design-with-finite-support-step4} and~\eqref{eq:opt-design-with-finite-support-step5} that
\begin{align} \label{eq:opt-design-with-finite-support-step6}
\eta^* \in \Xifeas(X) \qquad \text{and} \qquad \Psi_0(\eta^*) = \Psi_0(\xi^*).
\end{align}  
And therefore $\eta^*$ is an $\eps$-optimal solution to~\eqref{eq:COED(X)} which by~\eqref{eq:opt-design-with-finite-support-step2} has at most $d$ support points, as desired. 
\end{proof}

\subsection{Characterization of optimal designs}

In this section, we characterize the solutions of~\eqref{eq:COED(X)} in terms of saddle points of~\eqref{eq:COED(X)}, that is, saddle points of the Lagrange function of~\eqref{eq:COED(X)}. 

\subsubsection{Sufficient optimality conditions}

In this paragraph, we show that every saddle point $(\xi^*,\lambda^*)$ of~\eqref{eq:COED(X)} gives rise to a solution $\xi^*$ of~\eqref{eq:COED(X)} (Lemma~\ref{lm:saddle-points-yield-solutions}). 
Additionally, we establish the sufficient $\eps$-optimality condition indicated in the introduction (Corollary~\ref{cor:minimum-of-sensi-fct-determines-optimality-gap}). 
Whenever Condition~\ref{cond:objective-and-constraint-fcts} is satisfied, we will denote the Lagrange function of~\eqref{eq:COED(X)} by $L$. In other words, the function $L: \Xifin(X) \times \Lambda \to \R$ is defined by
\begin{gather} 
	L(\xi,\lambda) := \Psi_0(\xi) + \sum_{i\in I} \lambda_i \Psi_i(\xi) 
	\qquad ((\xi,\lambda) \in \Xifin(X) \times \Lambda),
	\label{eq:Lagrange-fct-def}\\
	\Xifin(X) := \dom \Psi_0 \cap \bigcap_{i\in I} \dom\Psi_i 
	\qquad \text{and} \qquad
	\Lambda := [0,\infty)^{\Iineq} \times \R^{\Ieq}.
	\label{eq:Xifin(X)-and-Lambda-def}
\end{gather} 
In view of the definition~\eqref{eq:Xifeas(X)-definition}  of $\Xifeas(X)$, it is clear that
\begin{align}
	\Xifin(X) \cap \Xifeas(X) = \dom \Psi_0 \cap \Xifeas(X).
\end{align} 
As usual, a saddle point of~\eqref{eq:COED(X)} (or of $L$) is defined to be a point $(\xi^*,\lambda^*) \in \Xifin(X) \times \Lambda$ such that
\begin{align}
L(\xi^*,\lambda) \le L(\xi^*,\lambda^*) \le L(\xi,\lambda^*) 
\end{align}
for every $\xi \in \Xifin(X)$ and every $\lambda \in \Lambda$. It is straightforward to show and well-known that every saddle point of~\eqref{eq:COED(X)} gives rise to a solution of~\eqref{eq:COED(X)}. 

\begin{lm} \label{lm:saddle-points-yield-solutions}
Suppose that Condition~\ref{cond:objective-and-constraint-fcts} is satisfied and that $(\xi^*,\lambda^*) \in \Xifin(X) \times \Lambda$ is a saddle point of~\eqref{eq:COED(X)}. Then $\xi^*$ is a solution of~\eqref{eq:COED(X)} satisfying the complementarity condition
\begin{align}
\lambda_i^* \Psi_i(\xi^*) = 0 \qquad (i \in \Iineq). 
\end{align}
\end{lm}

\begin{proof}
We can argue in exactly the same way as in \cite{Mangasarian} (Theorem~5.3.1 and 5.3.2) or in~\cite{Do11} (Theorem~2), for instance.
\end{proof}

Whenever both Condition~\ref{cond:objective-and-constraint-fcts} and Condition~\ref{cond:sensi-fcts} are satisfied, we write $\psi^L$ for the function $\psi^L: \Xifin(X) \times \Lambda \times X \to \R$ with
\begin{align} \label{eq:psiL-def}
	\psi^L(\xi, \lambda, x) := \psi_0(\xi,x) + \sum_{i \in I} \lambda_i \psi_i(\xi,x)
	\qquad ((\xi,\lambda,x) \in \Xifin(X) \times \Lambda \times X).
\end{align}
It is straightforward to verify that $\Xifin(X) \times X \ni (\xi,x) \mapsto \psi^L(\xi,\lambda,x)$ is the sensitivity function of $L(\cdot,\lambda)|_{\Xifin(X)}$ w.r.t.~$\Xi(X)$ for every $\lambda \in \Lambda$. In this context, it should be noticed, however, that $X \ni x \mapsto \psi^L(\xi,\lambda,x)$ is not continuous in general. After all, Condition~\ref{cond:sensi-fcts} allows the sensitivity functions of the inequality constraint functions to be discontinuous. Specifically, writing 
\begin{align} \label{eq:Ic-and-Inc-definition}
	\Ic := \{i \in I: \psi_i \text{ is continuous}\} 
	\qquad \text{and} \qquad
	\Inc := \{i \in I: \psi_i \text{ is not continuous}\}
\end{align}
for the subsets of $I$ for which the sensitivity function $\psi_i$ of $\Psi_i|_{\dom\Psi_i}$ is continuous or discontinuous, respectively, we see that $I = \Ic \cup \Inc$ and, by Condition~\ref{cond:sensi-fcts} and Lemma~\ref{lm:equality-constr-fcts-have-continuous-sensi-fct}, we also see that
\begin{gather} 
	\Ieq \subset \Ic 
	\qquad \text{and} \qquad
	\Inc \subset \Iineq,
	\label{eq:Ieq-subset-of-Ic-and-Inc-subset-of-Iineq}\\
	\psi_i(\xi,x) = g_i(x) - \Psi_i(\xi) \qquad ((\xi,x) \in \dom\Psi_i \times X \text{ and } i \in \Inc).
	\label{eq:psi_i-for-i-in-Inc}
\end{gather}

\begin{lm} \label{lm:bound-on-differences-of-Psiis-in-terms-of-sensi-fcts}
Suppose that Conditions~\ref{cond:objective-and-constraint-fcts} and~\ref{cond:sensi-fcts} are satisfied. Then
\begin{align} \label{eq:bound-on-differences-of-Psiis-in-terms-of-sensi-fcts-1}
	\Psi_i(\eta) - \Psi_i(\xi) \ge \int_X \psi_i(\xi,x) \d\eta(x) 
\end{align} 
for every $\xi \in \dom \Psi_i$ and $\eta \in \Xi(X)$ for every $i \in \{0\} \cup I$. In particular, 
\begin{align} \label{eq:bound-on-differences-of-Psiis-in-terms-of-sensi-fcts-2}
	L(\eta,\lambda) - L(\xi,\lambda) \ge \int_X \psi^L(\xi,\lambda,x) \d\eta(x)
\end{align}
for every $\xi, \eta \in \Xifin(X)$ and every $\lambda \in \Lambda$. Additionally, $\Xifin(X)$ is directionally open in the following sense: for every $\xi \in \Xifin(X)$ and every $\eta \in \Xi(X)$, there exists an $\alpha^* \in (0,1]$ such that $\xi + \alpha(\eta-\xi) \in \Xifin(X)$ for all $\alpha \in [0,\alpha^*]$. And finally, 
\begin{align} \label{eq:inf-psiL-is-finite}
	\inf_{x \in X} \psi^L(\xi, \lambda, x) \in \R
	\qquad ((\xi,\lambda) \in \Xifin(X) \times \Lambda).
\end{align}
\end{lm}

\begin{proof}
	Since $\Psi_i$ is convex for every $i \in \{0\} \cup \Iineq$ and even affine for every $i \in \Ieq$, we see that
	\begin{align}
		\Psi_i(\xi + \alpha(\eta-\xi)) - \Psi_i(\xi) &\le \alpha \big( \Psi_i(\eta) - \Psi_i(\xi) \big) \qquad (i \in \{0\} \cup \Iineq) \label{eq:bound-on-differences-of-Psiis-in-terms-of-sensi-fcts-step1}\\
		\Psi_i(\xi + \alpha(\eta-\xi)) - \Psi_i(\xi) &= \alpha \big( \Psi_i(\eta) - \Psi_i(\xi) \big) \qquad (i \in \Ieq)
		\label{eq:bound-on-differences-of-Psiis-in-terms-of-sensi-fcts-step2}
	\end{align}
	for all $\xi, \eta \in \Xi(X)$. Since, moreover, $\Psi_i|_{\dom\Psi_i}$ is directionally differentiable w.r.t.~$\Xi(X)$ with sensitivity function $\psi_i$ (Condition~\ref{cond:sensi-fcts} and Lemma~\ref{lm:equality-constr-fcts-have-continuous-sensi-fct}), it follows from~\eqref{eq:bound-on-differences-of-Psiis-in-terms-of-sensi-fcts-step1} and~\eqref{eq:bound-on-differences-of-Psiis-in-terms-of-sensi-fcts-step2} that
	\begin{align} 
		\int_X \psi_i(\xi,x) \d\eta(x) &\le \Psi_i(\eta)-\Psi_i(\xi) \qquad (\xi \in \dom\Psi_i \text{ and } i \in \{0\} \cup \Iineq)
		\label{eq:bound-on-differences-of-Psiis-in-terms-of-sensi-fcts-step3} \\
		\int_X \psi_i(\xi,x) \d\eta(x) &= \Psi_i(\eta)-\Psi_i(\xi) \qquad (\xi \in \dom\Psi_i \text{ and }  i \in \Ieq)
		\label{eq:bound-on-differences-of-Psiis-in-terms-of-sensi-fcts-step4}
	\end{align}
	for 
	every $\eta \in \Xi(X)$. In view of~\eqref{eq:Xifin(X)-and-Lambda-def}, \eqref{eq:psiL-def}, \eqref{eq:bound-on-differences-of-Psiis-in-terms-of-sensi-fcts-step3} and~\eqref{eq:bound-on-differences-of-Psiis-in-terms-of-sensi-fcts-step4}, the assertions~\eqref{eq:bound-on-differences-of-Psiis-in-terms-of-sensi-fcts-1} and~\eqref{eq:bound-on-differences-of-Psiis-in-terms-of-sensi-fcts-2} are now clear. 
	We now move on to prove the directional openness of $\Xifin(X)$. So, let $\xi \in \Xifin(X)$ and let $\eta \in \Xi(X)$ be given. We then know from Condition~\ref{cond:sensi-fcts} and Lemma~\ref{lm:equality-constr-fcts-have-continuous-sensi-fct} that for every $i \in \{0\} \cup I$ the limit
	\begin{align}
		\lim_{\alpha \searrow 0} \frac{\Psi_i(\xi+\alpha(\eta-\xi)) - \Psi_i(\xi)}{\alpha} = \int_X \psi_i(\xi,x)\d\xi(x) \in \R 
	\end{align} 
	exists in $\R$ and therefore $\Psi_i(\xi+\alpha(\eta-\xi))$ must be an element of $\R$ for $\alpha$ close enough to $0$. In other words, $\xi+\alpha(\eta-\xi) \in \Xifin(X)$ for $\alpha$ close enough to $0$, as desired.
	It remains to prove the finiteness assertion~\eqref{eq:inf-psiL-is-finite}. So, let $\xi \in \Xifin(X)$ and $\lambda \in \Lambda$. Since $X$ is compact and non-empty (Condition~\ref{cond:objective-and-constraint-fcts}) and since the sensitivity functions $\psi_i$ are continuous for $i \in \Ic \cup \{0\}$ (Condition~\ref{cond:sensi-fcts}), we obtain the boundedness
	\begin{align} \label{eq:psi-bounded-for-i-in-Ic}
		\sup_{x\in X} |\psi_i(\xi,x)| = \max_{x \in X} |\psi_i(\xi,x)| < \infty 
		\qquad (i \in \Ic \cup \{0\}).
	\end{align}
	Since the sensitivity functions $\psi_i$ for $i \in \Inc$ are of the form~\eqref{eq:psi_i-for-i-in-Inc} with bounded measurable functions $g_i$ (Condition~\ref{cond:sensi-fcts}), we also obtain the boundedness
	\begin{align} \label{eq:psi-bounded-for-i-in-Inc}
		\sup_{x\in X} |\psi_i(\xi,x)| < \infty 
		\qquad (i \in \Inc).
	\end{align}
	In view of~\eqref{eq:psi-bounded-for-i-in-Ic} and~\eqref{eq:psi-bounded-for-i-in-Inc}, it is now clear that $X \ni x \mapsto \psi^L(\xi,\lambda,x)$ is bounded. In particular, it is lower bounded and thus the assertion~\eqref{eq:inf-psiL-is-finite} is proved. 
\end{proof}

With the help of the previous lemma, we can now establish an important sufficient condition for the $\eps$-optimality of a design $\xi^*$, namely in terms of the infimum of the sensitivity function $X \ni x \mapsto \psi^L(\xi^*,\lambda^*,x)$. (As has been pointed out after the definition~\eqref{eq:psiL-def}, the sensitivity function is not continuous in general and therefore the infimum is not a minimum in general.) 

\begin{cor} \label{cor:minimum-of-sensi-fct-determines-optimality-gap}
Suppose that Conditions~\ref{cond:objective-and-constraint-fcts} and~\ref{cond:sensi-fcts} are satisfied. Suppose further that
\begin{align}
	\label{eq:minimum-of-sensi-fct-determines-optimality-gap-1}
	\xi^* \in \Xifin(X) \cap \Xifeas(X) 
	\qquad \text{and} \qquad
	\lambda^* \in \Lambda
\end{align}
and $\eps \in [0,\infty)$ are such that
\begin{align}
	\label{eq:minimum-of-sensi-fct-determines-optimality-gap-2}
	\lambda_i^* \Psi_i(\xi^*) = 0 \qquad (i \in \Iineq)
	\qquad \text{and} \qquad
	\inf_{x \in X} \psi^L(\xi^*, \lambda^*, x) \ge -\eps. 
\end{align}
Then $\xi^*$ is an $\eps$-optimal solution of~\eqref{eq:COED(X)}. In particular, $\xi^*$ is an $\eps^*$-optimal solution of~\eqref{eq:COED(X)} with $\eps^* := |\inf_{x \in X} \psi^L(\xi^*, \lambda^*, x )|$. 
\end{cor}

\begin{proof}
In view of~(\ref{eq:minimum-of-sensi-fct-determines-optimality-gap-1}.a) and of the solvability theorem (Theorem~\ref{thm:solvability}), there exists an (exactly) optimal design $\eta^*$ with $\eta^* \in \Xifin(X)$. In short, 
\begin{align}
	\label{eq:minimum-of-sensi-fct-determines-optimality-gap-3}
	\eta^* \in \Xifin(X) \cap \Xifeas(X)
	\qquad \text{and} \qquad
	\Psi_0(\eta^*) = \min_{\xi \in \Xifeas(X)} \Psi_0(\xi).
\end{align}
With the help of~\eqref{eq:bound-on-differences-of-Psiis-in-terms-of-sensi-fcts-2}, it now follows by~\eqref{eq:minimum-of-sensi-fct-determines-optimality-gap-1}-\eqref{eq:minimum-of-sensi-fct-determines-optimality-gap-3} that
\begin{align}
	\label{eq:minimum-of-sensi-fct-determines-optimality-gap-4}
	-\eps 
	&\le \int_X \psi^L(\xi^*, \lambda^*, x) \d\eta^*(x)
	\le L(\eta^*, \lambda^*) - L(\xi^*, \lambda^*) \notag\\
	&= \Psi_0(\eta^*) + \sum_{i \in \Iineq} \lambda_i^* \Psi_i(\eta^*) - \Psi_0(\xi^*)
	\le \min_{\xi \in \Xifeas(X)} \Psi_0(\xi) - \Psi_0(\xi^*).
\end{align}
In view of~(\ref{eq:minimum-of-sensi-fct-determines-optimality-gap-1}.a) and~\eqref{eq:minimum-of-sensi-fct-determines-optimality-gap-4}, it is now clear that $\xi^*$ is an $\eps$-optimal solution, as desired. 
Applying the first part of the assertion with $\eps := \eps^* := |\inf_{x \in X} \psi^L(\xi^*, \lambda^*, x )| \in [0,\infty)$ (Lemma~\ref{lm:bound-on-differences-of-Psiis-in-terms-of-sensi-fcts}), we finally see that $\xi^*$ is also an $\eps^*$-optimal solution, as desired.
\end{proof}

\subsubsection{A necessary optimality condition}

In this paragraph, we show that under a suitable regularity condition, the converse implication of Lemma~\ref{lm:saddle-points-yield-solutions} holds true as well, that is, every solution $\xi^*$ gives rise to a saddle point $(\xi^*,\lambda^*)$ of~\eqref{eq:COED(X)} (Theorem~\ref{thm:solutions-yield-saddle-points}). We will use this implication to ensure that all the discretized design problems occurring in our algorithm really have a saddle point. See the proof of Lemma~\ref{lm:sufficient-cond-for-saddle-points-existence-and-reg-conditions} below. Compared to the Robinson-Zowe-Kurcyusz regularity condition~\cite{Ro76, ZoKu79} used in~\cite{MoZu02, MoZu04} and to the Daniele-Giuffr\'{e}-Maugeri regularity condition~\cite{DaGiMa09} used in~\cite{Do11}, our regularity condition  is considerably simpler to verify in the experimental-design setting considered here. 
\smallskip

We call a subset $Z$ of $\R^J$ -- with $J$ a finite index set -- \emph{one-sided w.r.t.~$0$} iff there is a vector $\nu \in \R^J \setminus \{0\}$ such that 
\begin{align}
Z \subset \{w \in \R^J: \nu^\top w \ge 0\}.
\end{align}
In other words, a subset $Z$ of $\R^J$ is one-sided w.r.t.~$0$ iff $Z$ is completely contained in one of the two  closed sides (half-spaces) of a hyperplane in $\R^J$ that passes through $0$. 

\begin{lm} \label{lm:non-one-sidedness}
Suppose $h: Y \to \R^J$ is a mapping from some set $Y$ to $\R^J$ with some finite set $J \ne \emptyset$ and let $Z := h(Y)$ be the range of $h$. 
\begin{itemize}
\item[(i)] If $Z$ is not one-sided w.r.t.~$0$, then the component mappings $\{h_j: j \in J\}$ are linearly independent.
\item[(ii)] If $Z$ has a non-empty intersection with the interior of every orthant of $\R^J$, then $Z$ is not one-sided w.r.t.~$0$. In particular, $Z$ is not one-sided w.r.t.~$0$ if $0$ lies in the interior of $Z$. 
\end{itemize}
\end{lm}

\begin{proof}
Assertion (i) immediately follows by the definitions of linear dependence and of one-sidedness w.r.t.~$0$.
Assertion (ii) easily follows by the definition as well. Indeed, let $Z$ have a non-empty intersection with the interior of every orthant of $\R^J$ and assume, for the sake of argument, that $Z$ is one-sided w.r.t.~$0$. It then follows that there is a $\nu \in \R^J \setminus \{0\}$ such that $\nu^\top z \ge 0$ for all $z \in Z$. In particular,
\begin{align} \label{eq:non-one-sidedness-induction-1}
J^* := \{j \in J: \nu_j \ne 0\} \ne \emptyset
\qquad \text{and} \qquad
0 \le \nu^\top z = \sum_{j\in J} \nu_j z_j = \sum_{j\in J^*} \nu_j z_j \qquad (z \in Z).
\end{align}
Choose now a $z^* \in Z$ that belongs to the interior of an orthant opposite from $\nu$ or, more precisely, a point $z^*$ such that
\begin{align}
z^* \in Z 
\qquad \text{and} \qquad
z^*_j \ne 0 \quad (j \in J) 
\qquad \text{and} \qquad
\sgn(z^*_j) = -\sgn(\nu_j) \quad (j \in J^*),
\end{align}
where $\sgn(t)$ denotes the sign of the non-zero real number $t$, of course. Inserting then $z^*$ into (\ref{eq:non-one-sidedness-induction-1}.b), we obtain a contradiction. So, our assumption that $Z$ is one-sided w.r.t.~$0$ was false and the proof is complete.
\end{proof}

We can now formulate our regularity condition. Condition~\ref{cond:slater-and-non-one-sidedness}~(i) is a  strict feasibility condition on our design problem $\CD(X)$ and Condition~\ref{cond:slater-and-non-one-sidedness}~(ii) is a non-one-sidedness condition on the range $\Psieq(\Xi(X))$ of the equality constraint functions. If there are no equality constraints in~\eqref{eq:COED(X)}, this non-one-sidedness condition is automatically satisfied because then $\R^{\Ieq} = \R^\emptyset = \{0\}$ by convention. 

\begin{cond} \label{cond:slater-and-non-one-sidedness}
\begin{itemize}
\item[(i)] $\CD(X)$ has a strictly feasible point, that is, a point $\eta^* \in \Xi(X)$ such that
\begin{align} 
\Psi_i(\eta^*) < 0 \qquad (i \in \Iineq) \qquad \text{and} \qquad \Psi_i(\eta^*) = 0 \qquad (i \in \Ieq). 
\end{align}
\item[(ii)] $\Psieq(\Xi(X)) := \{(\Psi_i(\xi))_{i\in\Ieq}: \xi \in \Xi(X)\}$ is not one-sided w.r.t.~$0$.
\end{itemize}
\end{cond}

In the next lemma, we provide simple sufficient conditions for the above Condition~\ref{cond:slater-and-non-one-sidedness} to be satisfied. In essence, this lemma reduces the verification of the strict feasibility and non-one-sidedness conditions from Condition~\ref{cond:slater-and-non-one-sidedness} to their verification on an arbitrary non-empty closed subset $X^*$ of $X$. In this context, it is important to notice that every design on a measurable subset $X^*$ of $X$ can be canonically identified with a design on $X$. Specifically, the design $\xi^* \in \Xi(X^*)$ can be canonically identified with the design $\iota^*(\xi^*) \in \Xi(X)$ via the canonical embedding $\iota^*: \Xi(X^*) \to \Xi(X)$, which is defined by
\begin{align}
	\iota^*(\xi^*)(E) := \xi^*(E \cap X^*) \qquad (E \in \B_X).
\end{align} 
Clearly, this canonical embedding is a continuous and linear mapping. 
In the following, we will always tacitly make this identification for the sake of notational simplicity. Accordingly, we will always consider $\Xi(X^*)$ as a subset of $\Xi(X)$ and we will always apply the objective and constraint functions $\Psi_i$ from Condition~\ref{cond:objective-and-constraint-fcts} as well as their sensitivity functions $\psi_i$ from Condition~\ref{cond:sensi-fcts} also to designs $\xi^* \in \Xi(X^*)$. In particular, we will write
\begin{align}
	\Xi(X^*) \subset \Xi(X), \qquad
	\Psi_i(\xi^*), \qquad
	\psi_i(\xi^*,\cdot)|_{X^*},
\end{align}
although strictly speaking we would have to write $\iota^*(\Xi(X^*)) \subset \Xi(X)$ and $\Psi_i(\iota^*(\xi^*))$ and $\psi_i(\iota^*(\xi^*),\cdot)|_{X^*}$ instead. 

\begin{lm} \label{lm:sufficient-conds-for-lin-slater-and-non-one-sidedness-cond}
Suppose that Conditions~\ref{cond:objective-and-constraint-fcts} and~\ref{cond:sensi-fcts} are satisfied.
\begin{itemize}
\item[(i)] Condition~\ref{cond:slater-and-non-one-sidedness}~(i) is satisfied if, for some closed subset $X^* \ne \emptyset$ of $X$, the design problem $\CD(X^*)$ has a strictly feasible point $\eta^*$, that is, a point $\eta^* \in \Xi(X^*)$ such that
\begin{align} \label{eq:slater-point}
\Psi_i(\eta^*) < 0 \qquad (i \in \Iineq) \qquad \text{and} \qquad \Psi_i(\eta^*) = 0 \qquad (i \in \Ieq). 
\end{align}
\item[(ii)] Condition~\ref{cond:slater-and-non-one-sidedness}~(ii) is satisfied if, for some closed subset $X^* \ne \emptyset$ of $X$, the set $\Psieq(\Xi(X^*)) := \{(\Psi_i(\xi))_{i\in\Ieq}: \xi \in \Xi(X^*)\}$ is not one-sided w.r.t.~$0$. And this, in turn, is satisfied if the range $\gequ(X^*) := \{(g_i(x))_{i\in\Ieq}: x \in X^*\}$ is not one-sided w.r.t.~$0$. 
\end{itemize}
\end{lm}

\begin{proof}
(i) Suppose $X^* \ne \emptyset$ is a closed subset of $X$ such that the design problem $\CD(X^*)$ has a strictly feasible point $\eta^*$. Suppose further that $X^{**}$ is another closed subset of $X$ with $X^{**} \supset X^*$. We then show that Condition~\ref{cond:slater-and-non-one-sidedness}~(i) is satisfied with $X^{**}$ instead of $X$. (Setting $X^{**} := X$, we then obtain the assertion of part (i) of the lemma -- the more general claim with a general $X^{**}$ is needed only later on in Lemma~\ref{lm:sufficient-cond-for-saddle-points-existence-and-reg-conditions}.) 
Since $\eta^*$ is a strictly feasible point of $\CD(X^*)$ by assumption and since $\Xi(X^*) \subset \Xi(X^{**})$, it immediately follows that $\eta^*$ is a strictly feasible point of $\CD(X^{**})$ as well. We have thus established that Condition~\ref{cond:slater-and-non-one-sidedness}~(i) is satisfied with $X^{**}$ instead of $X$, as claimed. 
\smallskip

(ii) Suppose that $X^* \ne \emptyset$ is a closed subset of $X$ such that the set $\Psieq(\Xi(X^*)) := \{(\Psi_i(\xi))_{i\in\Ieq}: \xi \in \Xi(X^*)\}$ is not one-sided w.r.t.~$0$. Suppose further that $X^{**}$ is another closed subset of $X$ with $X^{**} \supset X^*$. We then show that Condition~\ref{cond:slater-and-non-one-sidedness}~(ii) is satisfied with $X^{**}$ instead of $X$. (Setting $X^{**} := X$, we then obtain the assertion of part (ii) of the lemma -- the more general claim with a general $X^{**}$ is needed only later on in Lemma~\ref{lm:sufficient-cond-for-saddle-points-existence-and-reg-conditions}.) 
Assume, to the contrary, that $\Psieq(\Xi(X^{**}))$ is one-sided w.r.t.~$0$. We would then have a vector $\nu \in \R^{\Ieq}$ such that
\begin{align} \label{eq:sufficient-conds-for-lin-slater-and-non-one-sidedness-cond-(ii)-1}
\nu \ne 0 
\qquad \text{and} \qquad 
\Psieq(\Xi(X^{**})) \subset \{w \in \R^{\Ieq}: \nu^\top w \ge 0\}.
\end{align} 
Since $\Xi(X^*) \subset \Xi(X^{**})$, the relation~\eqref{eq:sufficient-conds-for-lin-slater-and-non-one-sidedness-cond-(ii)-1} would further imply that $\Psieq(\Xi(X^*))$ is one-sided w.r.t.~$0$. Contradiction! 
It remains to prove that the non-one-sidedness of $\gequ(X^*)$ implies the non-one-sidedness of $\Psieq(\Xi(X^*))$. We prove this by contraposition. So, let $\Psieq(\Xi(X^*))$ be one-sided w.r.t.~$0$. It then follows that there is a vector $\nu \in \R^{\Ieq}$ such that
\begin{align} \label{eq:sufficient-conds-for-lin-slater-and-non-one-sidedness-cond-(ii)-2}
\nu \ne 0 
\qquad \text{and} \qquad
0 \le \nu^\top \Psieq(\delta_x) 
= \sum_{i\in\Ieq} \nu_i g_i(x) = \nu^\top \gequ(x)
\qquad (x \in X^*),
\end{align}
where for the first equality we used Lemma~\ref{lm:equality-constr-fcts-have-continuous-sensi-fct}. In view of~\eqref{eq:sufficient-conds-for-lin-slater-and-non-one-sidedness-cond-(ii)-2}, the set $\gequ(X^*)$ is one-sided w.r.t.~$0$, as desired. 
\end{proof}

With these preliminaries at hand, we can now prove that under our regularity condition (Condition~\ref{cond:slater-and-non-one-sidedness}) every solution of~\eqref{eq:COED(X)} yields a saddle point of~\eqref{eq:COED(X)}.

\begin{thm} \label{thm:solutions-yield-saddle-points}
Suppose that Conditions~\ref{cond:objective-and-constraint-fcts} and~\ref{cond:sensi-fcts} are satisfied and that $\Xifin(X) \cap \Xifeas(X) \ne \emptyset$. Suppose further that $\xi^*$ is a solution of~\eqref{eq:COED(X)} and that Condition~\ref{cond:slater-and-non-one-sidedness} is satisfied. Then there exists a $\lambda^* \in \Lambda$ such that $(\xi^*,\lambda^*)$ is a saddle point of~\eqref{eq:COED(X)}.
\end{thm}

\begin{proof}
As a first step, we show that the sets 
\begin{align} \label{eq:solutions-yield-saddle-points-step1-1}
C := \big\{ \big(\partial_{\eta-\xi^*}\Psi_0(\xi^*) + \alpha, \,
&\Psiineq(\xi^*) + \partial_{\eta-\xi^*}\Psiineq(\xi^*) + \beta, \Psieq(\eta) \big): \notag \\ 
&\eta \in \Xi(X), \alpha \in [0,\infty), \beta \in (0,\infty)^{\Iineq} \big\}
\subset \R \times \R^{\Iineq} \times \R^{\Ieq}
\end{align}
and $D := (-\infty,0) \times \{0\} \times \{0\} \subset \R \times \R^{\Iineq} \times \R^{\Ieq}$ have an empty intersection. 
We show this by contradiction. So, assume that $C \cap D \ne \emptyset$ for the sake of argument. Then there exist $\hat{\eta} \in \Xi(X)$, $\hat{\alpha} \in [0,\infty)$, $\hat{\beta} \in (0,\infty)^{\Iineq}$ and $\hat{\eps} \in (0,\infty)$ such that 
\begin{gather}
\partial_{\hat{\eta}-\xi^*}\Psi_0(\xi^*) + \hat{\alpha} = -\hat{\eps},
\label{eq:solutions-yield-saddle-points-step1-2}\\
\Psi_i(\xi^*) + \partial_{\hat{\eta}-\xi^*}\Psi_i(\xi^*) + \hat{\beta}_i = 0 \qquad (i \in \Iineq)
\qquad \text{and} \qquad \Psi_i(\hat{\eta}) = 0 \qquad (i \in \Ieq).
\label{eq:solutions-yield-saddle-points-step1-3}
\end{gather}
Since $\xi^* \in \Xifin(X) \cap \Xifeas(X)$ by assumption, it easily follows from~\eqref{eq:solutions-yield-saddle-points-step1-3} that there exists a $t^* \in (0,1]$ such that
\begin{align} \label{eq:solutions-yield-saddle-points-step1-4}
\xi^* + t (\hat{\eta} - \xi^*) \in \Xifeas(X) \qquad (t \in [0,t^*]).
\end{align}
In order to prove~\eqref{eq:solutions-yield-saddle-points-step1-4}, we use the set $I(\xi^*) := \{i \in \Iineq: \Psi_i(\xi^*) = 0\}$ of active inequality constraints at $\xi^*$. If $i \in I(\xi^*)$, then $\Psi_i(\xi^*) = 0$ and therefore
\begin{align} \label{eq:solutions-yield-saddle-points-step1-4.1}
\lim_{t\searrow 0} \frac{\Psi_i(\xi^*+t(\hat{\eta}-\xi^*))}{t} = \partial_{\hat{\eta}-\xi^*}\Psi_i(\xi^*) = -\hat{\beta}_i < 0
\end{align}
by virtue of~(\ref{eq:solutions-yield-saddle-points-step1-3}.a). If $i \in \Iineq \setminus I(\xi^*)$, then $\Psi_i(\xi^*) < 0$ and therefore
\begin{align} \label{eq:solutions-yield-saddle-points-step1-4.2}
\limsup_{t\searrow 0} \Psi_i(\xi^*+t(\hat{\eta}-\xi^*)) \le \Psi_i(\xi^*) < 0
\end{align}
by the convexity of $\Psi_i$. And finally, if $i \in \Ieq$, then 
\begin{align} \label{eq:solutions-yield-saddle-points-step1-4.3}
\Psi_i(\xi^*+t(\hat{\eta}-\xi^*)) = (1-t) \Psi_i(\xi^*) + t \Psi_i(\hat{\eta}) = 0 \qquad (t \in [0,1])  
\end{align}
by virtue of the affinity of $\Psi_i$ and~(\ref{eq:solutions-yield-saddle-points-step1-3}.b). In view of~\eqref{eq:solutions-yield-saddle-points-step1-4.1}-\eqref{eq:solutions-yield-saddle-points-step1-4.3}, the claimed existence of a $t^* \in (0,1]$ with~\eqref{eq:solutions-yield-saddle-points-step1-4} is now clear. Since $\xi^*$ is a minimizer of~\eqref{eq:COED(X)} by assumption, it follows from~\eqref{eq:solutions-yield-saddle-points-step1-4} that
\begin{align}
0 \le \lim_{t\searrow 0} \frac{\Psi_0(\xi^*+t(\hat{\eta}-\xi^*)) - \Psi_0(\xi^*)}{t} = \partial_{\hat{\eta}-\xi^*}\Psi_0(\xi^*).
\end{align}
Contradiction to~\eqref{eq:solutions-yield-saddle-points-step1-2}!
\smallskip

As a second step, we show that there exists a non-zero vector $(\kappa,\lambda) \in [0,\infty) \times \Lambda$ such that
\begin{gather}
\kappa \partial_{\eta-\xi^*}\Psi_0(\xi^*) + \sum_{i \in \Iineq} \lambda_i \partial_{\eta-\xi^*}\Psi_i(\xi^*) + \sum_{i\in\Ieq} \lambda_i \Psi_i(\eta) \ge 0 \qquad (\eta \in \Xi(X)),
\label{eq:solutions-yield-saddle-points-step2.1}\\
\sum_{i\in\Iineq} \lambda_i \Psi_i(\xi^*) = 0
\label{eq:solutions-yield-saddle-points-step2.2}
\end{gather}
In order to see this, we first observe that $C$ and $D$ from the first step are disjoint convex subsets of $\R \times \R^I$ and can thus be linearly separated. In other words, there exists a vector $\nu = (\kappa,\lambda) \in \R \times \R^I$ such that 
\begin{align} \label{eq:solutions-yield-saddle-points-step2-1}
\nu \ne 0
\qquad \text{and} \qquad
\nu^\top d \le \nu^\top c \qquad (c \in C \text{ and } d \in D). 
\end{align}
Spelled out, the latter inequality means that
\begin{align} \label{eq:solutions-yield-saddle-points-step2-2}
-\kappa \cdot \eps \le \kappa \big(\partial_{\eta-\xi^*}\Psi_0(\xi^*) + \alpha\big) + \sum_{i \in \Iineq} \lambda_i \big( \Psi_i(\xi^*) + \partial_{\eta-\xi^*}\Psi_i(\xi^*) + \beta_i\big) + \sum_{i\in\Ieq} \lambda_i \Psi_i(\eta)
\end{align}
for all $\eta \in \Xi(X)$, $\alpha \in [0,\infty)$, $\beta \in (0,\infty)^{\Iineq}$ and all $\eps \in (0,\infty)$. Setting $\eps := 1$, $\eta := \xi^*$, $\alpha := \alpha_n := n$, $\beta := (1,\dots,1)$ and letting $n \to \infty$, we see from~\eqref{eq:solutions-yield-saddle-points-step2-2} that
\begin{align} \label{eq:solutions-yield-saddle-points-step2-3}
\kappa \ge 0.
\end{align}
Setting $\eps := 1$, $\eta := \xi^*$, $\alpha := 0$, $\beta := \beta_{ni} := (1,\dots,1) + n e_i$ with $e_i$ the $i$th canonical unit vector of $\R^{\Ieq}$ and letting $n \to \infty$, we see from~\eqref{eq:solutions-yield-saddle-points-step2-2} that
\begin{align} \label{eq:solutions-yield-saddle-points-step2-4}
\lambda_i \ge 0 \qquad (i \in \Iineq).
\end{align}
Additionally, by setting $\eps := \eps_n := 1/n$, $\alpha := 0$, $\beta := \beta_n := (1/n,\dots,1/n)$ and by letting $n\to\infty$, we conclude from~\eqref{eq:solutions-yield-saddle-points-step2-2} that
\begin{align} \label{eq:solutions-yield-saddle-points-step2-5}
\kappa \partial_{\eta-\xi^*}\Psi_0(\xi^*) + \sum_{i \in \Iineq} \lambda_i \big( \Psi_i(\xi^*) + \partial_{\eta-\xi^*}\Psi_i(\xi^*)\big) + \sum_{i\in\Ieq} \lambda_i \Psi_i(\eta) \ge 0 \qquad (\eta \in \Xi(X)).
\end{align}
Setting $\eta := \xi^*$ in~\eqref{eq:solutions-yield-saddle-points-step2-5} and using~\eqref{eq:solutions-yield-saddle-points-step2-4}, we immediately obtain the claimed complementarity relation~\eqref{eq:solutions-yield-saddle-points-step2.2}. Inserting~\eqref{eq:solutions-yield-saddle-points-step2.2} into~\eqref{eq:solutions-yield-saddle-points-step2-5}, in turn, we also obtain the claimed relation~\eqref{eq:solutions-yield-saddle-points-step2.1}, as desired.
\smallskip

As a third step, we show that for every non-zero vector $(\kappa,\lambda) \in [0,\infty) \times \Lambda$ as in the second step, we actually have $\kappa > 0$. 
So, let $(\kappa,\lambda) \in [0,\infty) \times \Lambda$ be a non-zero vector such that~\eqref{eq:solutions-yield-saddle-points-step2.1} and~\eqref{eq:solutions-yield-saddle-points-step2.2} are satisfied. We argue by contradiction and therefore assume that $\kappa = 0$. It then follows by~\eqref{eq:solutions-yield-saddle-points-step2-5} that
\begin{align} \label{eq:solutions-yield-saddle-points-step3-1}
\sum_{i \in \Iineq} \lambda_i \big( \Psi_i(\xi^*) + \partial_{\eta-\xi^*}\Psi_i(\xi^*)\big) + \sum_{i\in\Ieq} \lambda_i \Psi_i(\eta) \ge 0
\qquad (\eta \in \Xi(X)).
\end{align}
It further follows by~\eqref{eq:bound-on-differences-of-Psiis-in-terms-of-sensi-fcts-step3} and our strict feasibility condition (Condition~\ref{cond:slater-and-non-one-sidedness}~(i)) that 
\begin{align} \label{eq:solutions-yield-saddle-points-step3-2}
\Psi_i(\xi^*) + \partial_{\eta^*-\xi^*}\Psi_i(\xi^*) \le \Psi_i(\eta^*) < 0 \qquad (i\in\Iineq) 
\qquad \text{and} \qquad
\Psi_i(\eta^*) = 0 \qquad (i\in\Ieq)
\end{align}
for some $\eta^* \in \Xi(X)$. 
With the help of~\eqref{eq:solutions-yield-saddle-points-step2-4} and~\eqref{eq:solutions-yield-saddle-points-step3-2} we then conclude from~\eqref{eq:solutions-yield-saddle-points-step3-1} that we must have
\begin{align} \label{eq:solutions-yield-saddle-points-step3-3}
\lambda_i = 0 \qquad (i\in\Iineq).
\end{align}
Inserting~\eqref{eq:solutions-yield-saddle-points-step3-3} into~\eqref{eq:solutions-yield-saddle-points-step3-1}, we further conclude by our non-one-sidedness condition for $\Psieq(\Xi(X))$ (Condition~\ref{cond:slater-and-non-one-sidedness}~(ii)) that we must also have
\begin{align} \label{eq:solutions-yield-saddle-points-step3-4}
\lambda_i = 0 \qquad (i\in\Ieq).
\end{align}
So, by~\eqref{eq:solutions-yield-saddle-points-step3-3} and~\eqref{eq:solutions-yield-saddle-points-step3-4}, we obtain $(\kappa,\lambda) = (0,0)$. Contradiction to the non-zeroness of $(\kappa,\lambda)$!
\smallskip

As a fourth step, we finally show that there exists a multiplier $\lambda^* \in \Lambda$ such that $(\xi^*,\lambda^*)$ is a saddle point of~\eqref{eq:COED(X)}. 
In order to see this, we invoke the second and the third step to choose a non-zero vector $(\kappa,\lambda) \in (0,\infty) \times \Lambda$ such that~\eqref{eq:solutions-yield-saddle-points-step2.1} and~\eqref{eq:solutions-yield-saddle-points-step2.2} are satisfied, and define
\begin{align} \label{eq:solutions-yield-saddle-points-step4-1}
\lambda^* := \lambda/\kappa \in \Lambda.
\end{align}  
It then follows by~\eqref{eq:bound-on-differences-of-Psiis-in-terms-of-sensi-fcts-step3} and~\eqref{eq:solutions-yield-saddle-points-step2.1} that
\begin{align} \label{eq:solutions-yield-saddle-points-step4-2}
L(\eta,\lambda^*) - L(\xi^*,\lambda^*) 
&\ge \partial_{\eta-\xi^*} \Psi_0(\xi^*) + \sum_{i\in\Iineq} \lambda_i^* \partial_{\eta-\xi^*} \Psi_i(\xi^*) + \sum_{i\in\Ieq} \lambda_i^* \Psi_i(\eta) \notag\\
&\ge 0 \qquad (\eta \in \Xifin(X)).
\end{align}
It further follows by \eqref{eq:solutions-yield-saddle-points-step2.2} that
\begin{align} \label{eq:solutions-yield-saddle-points-step4-3}
L(\xi^*,\lambda^*) - L(\xi^*,\mu) = -\sum_{i\in\Iineq} \mu_i \Psi_i(\xi^*) \ge 0 \qquad (\mu \in \Lambda).
\end{align}
In view of~\eqref{eq:solutions-yield-saddle-points-step4-2} and~\eqref{eq:solutions-yield-saddle-points-step4-3}, it is now clear that $(\xi^*,\lambda^*)$ is indeed a saddle point of~\eqref{eq:COED(X)}, as desired.
\end{proof}

\section{Adaptive discretization algorithm}\label{sec:adaptive-disc-algorithm}

In this section, we formally introduce our algorithms to compute optimal experimental designs for~\eqref{eq:COED(X)}. 
As already indicated in the introduction, our algorithms are adaptive discretization algorithms which in every iteration proceed in two steps.
In the first step, a saddle point $(\xi^k,\lambda^k)$ of the discretized design problem 
\begin{align} 
	\label{eq:COED(X^k)-short}
	\min_{\xi \in \Xifeas (X^k)} \Psi_0(\xi)
\end{align} 
is computed, where $X^k$ is the current discretization of the design space $X$. 
In view of the finiteness of $X^k$, every design $\xi \in \Xifeas(X^k)$ is of the form $\xi = \sum_{x \in X} w_x \delta_x$ and thus, by the convexity and affinity of the objective and constraint functions, the discretized design problem~\eqref{eq:COED(X^k)-short}  is a finite-dimensional (namely, $|X^k|$-dimensional) convex optimization problem, which is equivalent to the convex weight optimization problem
\begin{align}
	\label{eq:weight-optimization-reformulation-of-the-discretized-design-problem}
	\min_{w \in \Delta(X^k)} \tilde{\Psi}_0^k(w)
	\quad \text{s.t.} \quad 
	\tilde{\Psi}_i^k(w) \le 0 \text{ for all } i \in \Iineq \text{ and } \tilde{\Psi}_i^k(w) = 0 \text{ for all } i \in \Ieq
\end{align} 
with $\Delta(X^k) := \{w \in [0,\infty)^{X^k}: \sum_{x \in X^k} w_x = 1\}$ and $\tilde{\Psi}_i^k(w) := \Psi_i(\sum_{x \in X^k} w_x \delta_x)$ for $i \in I \cup \{0\}$. 
In the second step, the discretization is updated in an adaptive manner. Specifically, the new discretization
\begin{align}
	\label{eq:new-discretization}
	X^{k+1} = X^k \cup \{x^k\}
\end{align} 
arises from the old one by adding one point, namely an  approximate violator $x^k$ of the sufficient $\eps$-optimality condition~(\ref{eq:minimum-of-sensi-fct-determines-optimality-gap-2}.b) in the sense of~\eqref{eq:approximate-violator-intro}.
In the first version of our algorithm -- the special algorithm -- we set the optimality tolerance $\eps = 0$  in~(\ref{eq:minimum-of-sensi-fct-determines-optimality-gap-2}.b) and compute $x^k$ as a $\ul{\delta}_k$-approximate solution of the sensitivity minimization problem~\eqref{eq:sensi-minimization-problem-intro}. 
As usual, this means that 
\begin{align}
	x^k \in X 
	\qquad \text{and} \qquad
	\psi^L(\xi^k,\lambda^k,x^k) \le \inf_{x \in X} \psi^L(\xi^k,\lambda^k,x) + \ul{\delta}_k
\end{align}
or, in other words, that $\psi^L(\xi^k,\lambda^k,x^k) - \ul{\delta}_k$ is a lower bound for the sensitivity function $X \ni x \mapsto \psi^L(\xi^k,\lambda^k,x)$ at $(\xi^k,\lambda^k)$. 
In the case where this lower bound is larger than or equal to $0$, the current iterate $\xi^k$ already is an optimal design for~\eqref{eq:COED(X)} (Corollary~\ref{cor:minimum-of-sensi-fct-determines-optimality-gap}) and we can terminate the iteration. In the opposite case, we continue the iteration with the new discretization~\eqref{eq:new-discretization}. 

\begin{algo} \label{algo:COED-special}
	Input: a non-empty finite subset $X^0$ of $X$ and tolerances $\ul{\delta}_k \in (0,\infty)$. Initialize $k=0$ and proceed in the following steps.
	\begin{itemize}
		\item[1.] Compute a saddle point $(\xi^k,\lambda^k) \in \Xi(X^k) \times \Lambda$ of the discretized optimal design problem~\eqref{eq:COED(X^k)-short}
		
		\item[2.] Compute a $\ul{\delta}_k$-approximate solution $x^k \in X$ of the sensitivity minimization problem
		\begin{align} \label{eq:SM(xi^k,lambda^k)-special}
		\min_{x\in X} \psi^L(\xi^k,\lambda^k,x).
		\end{align}
		If $\psi^L(\xi^k,\lambda^k,x^k) - \ul{\delta}_k < 0$, then set $X^{k+1} := X^k \cup \{x^k\}$ and return to Step~1 with $k$ replaced by $k+1$. If $\psi^L(\xi^k,\lambda^k,x^k) - \ul{\delta}_k \ge 0$, then terminate. 
	\end{itemize}
\end{algo}

In almost all applications, the sensitivity minimization problem~\eqref{eq:SM(xi^k,lambda^k)-special} is a nonlinear 
optimization problem so that its approximate solution is computationally expensive, in general. Also, the special algorithm above is not guaranteed to terminate after finitely many iterations. 
In the second version of our algorithm -- the general algorithm -- we therefore confine the approximate solution of the sensitivity minimization problems to as few iterations as possible and, moreover, we work with  positive optimality tolerances $\eps > 0$ to guarantee finite termination. 
Specifically, in every iteration we first search for an arbitrary violator of the sufficient $\eps$-optimality condition~(\ref{eq:minimum-of-sensi-fct-determines-optimality-gap-2}.b) at $(\xi^k,\lambda^k)$, that is, for a point $x^k \in X$ with 
\begin{align}
	\psi^L(\xi^k,\lambda^k,x^k) < -\eps.
\end{align}
We can use an arbitrary finite search routine for this purpose, typically a multistart local search. And only if this search routine finds no violator, do we have to  compute $x^k$ as a $\ul{\delta}_k$-approximate solution of~\eqref{eq:SM(xi^k,lambda^k)-special} as before (to be sure that at termination we have an $\eps$-optimal design for~\eqref{eq:COED(X)}).

\begin{algo} \label{algo:COED-general}
	Input: a non-empty finite subset $X^0$ of $X$, tolerances $\ul{\delta}_k \in (0,\infty)$, a tolerance $\eps \in [0,\infty)$. Initialize $k = 0$ and proceed in the following steps.
	\begin{itemize}
		\item[1.] Compute a saddle point $(\xi^k,\lambda^k) \in \Xi(X^k) \times \Lambda$ of the discretized optimal design problem~\eqref{eq:COED(X^k)-short}
		
		\item[2.] Search for a violator of the $\eps$-optimality condition~(\ref{eq:minimum-of-sensi-fct-determines-optimality-gap-2}.b) at $(\xi^k,\lambda^k)$, that is, for a point in the set $V^k := \{x \in X: \psi^L(\xi^k,\lambda^k,x) < -\eps\}$, using an arbitrary finite search routine. 
		\begin{itemize}
			\item If the search routine finds a point $x^k \in V^k$, then set $X^{k+1} := X^k \cup \{x^k\}$ and return to Step 1 with $k$ replaced by $k+1$.
			
			\item If the search routine finds no point in $V^k$, then compute a $\ul{\delta}_k$-approximate solution $x^k$ of the sensitivity minimization problem
			\begin{align} \label{eq:SM(xi^k,lambda^k)-general}
				\min_{x \in X} \psi^L(\xi^k,\lambda^k,x).
			\end{align}
			If $\psi^L(\xi^k,\lambda^k,x^k) - \ul{\delta}_k < -\eps$, then set $X^{k+1} := X^k \cup \{x^k\}$ and return to Step~1 with $k$ replaced by $k+1$. If $\psi^L(\xi^k,\lambda^k,x^k) - \ul{\delta}_k \ge -\eps$, then terminate.
		\end{itemize}
	\end{itemize} 
\end{algo}

It should be noticed that Algorithm~\ref{algo:COED-special} is a special case of Algorithm~\ref{algo:COED-general}. In order to see this, simply set $\eps = 0$ and choose the search routine to be empty (that is, to return nothing). In particular, our terminology 
of calling Algorithms~\ref{algo:COED-special} and~\ref{algo:COED-general} the special or, respectively, the general algorithm is justified. Additionally, this shows that all results formulated for Algorithm~\ref{algo:COED-general} with general $\eps \in [0,\infty)$ are valid especially for Algorithm~\ref{algo:COED-special}. 
\smallskip

It should also be noticed that in the special case with no constraints in~\eqref{eq:COED(X)} (that is, with $\Iineq = \emptyset$ and $\Ieq = \emptyset$), Algorithm~\ref{algo:COED-general} reduces to the unconstrained-design algorithm without exchange from~\cite{ScSeBo} (Algorithm~3.6 of~\cite{ScSeBo} with $\ol{\delta}_k = 0$ for all $k \in \N_0$).

\subsection{Well-definedness}

In order for our algorithms to be well-defined, we need to impose Conditions~\ref{cond:objective-and-constraint-fcts} and~\ref{cond:sensi-fcts}, of course, because otherwise the sensitivity function $\psi^L$ used in the second step of our algorithms would not even be defined. In addition to that, we need to ensure 
\begin{itemize}
	\item the existence of a saddle point $(\xi^k,\lambda^k)$ for every discretized design problem 
	
	\item the existence of a $\ul{\delta}_k$-approximate solution $x^k$ for every sensitivity minimization problem 
\end{itemize}
occurring in the course of the algorithms, because otherwise the first step and the second step of our algorithms would not be well-defined. 
%
%
%
As we will see in the next lemma, there is a simple sufficient condition for the above requirements, formulated solely in terms of the initial discretized design problem. Indeed, all we need for the above requirements to be satisfied is a strictly feasible point $\eta^0$ for the initial discretized design problem $\CD(X^0)$ with $\eta^0 \in \dom \Psi_0$ and the non-one-sidedness of $\Psieq(\Xi(X^0))$. As has been shown above (Lemma~\ref{lm:sufficient-conds-for-lin-slater-and-non-one-sidedness-cond}), this non-one-sidedness of $\Psieq(\Xi(X^0))$ is implied by the non-one-sidedness of $\gequ(X^0)$ w.r.t.~$0$.

\begin{cond} \label{cond:reg-conditions-on-X^0}
	\begin{itemize}
		\item[(i)] $\CD(X^0)$ has a strictly feasible point $\eta^0$ with $\Psi_0(\eta^0) < \infty$. In other words, there exists an
		\begin{align} \label{eq:strictly-feasible-design-on-X^0}
			\eta^0 \in \Xifin(X^0) 
			\quad \text{with} \quad
			\Psi_i(\eta^0) < 0 \quad (i \in \Iineq) \quad \text{and} \quad
			\Psi_i(\eta^0) = 0 \quad (i \in \Ieq).
		\end{align} 
		
		\item[(ii)] $\Psieq(\Xi(X^0))$ is not one-sided w.r.t.~$0$. 
	\end{itemize}
\end{cond}

It should be noticed that if the original problem~\eqref{eq:COED(X)} has a strictly feasible point $\xi^*$ with $\Psi_0(\xi^*) < \infty$, 
then -- in the situation of Corollary~\ref{cor:opt-design-with-finite-support} -- there also is a finite subset $X^0$ such that the discretized version $\CD(X^0)$ has a strictly feasible point $\eta^0$ with $\Psi_0(\eta^0) < \infty$. 
Indeed, this immediately follows by the proof of Corollary~\ref{cor:opt-design-with-finite-support}. 

\begin{lm} \label{lm:sufficient-cond-for-saddle-points-existence-and-reg-conditions}
Suppose that Conditions~\ref{cond:objective-and-constraint-fcts} and~\ref{cond:sensi-fcts} are satisfied. Suppose further that $\ul{\delta}_k \in (0,\infty)$ and $\eps \in [0,\infty)$ are tolerances and $X^0 \ne \emptyset$ is a finite subset of $X$ such that Condition~\ref{cond:reg-conditions-on-X^0} is satisfied. If Algorithm~\ref{algo:COED-general} is initialized with such tolerances $\ul{\delta}_k, \eps$ and such an initial discretization $X^0$, then the algorithm is well-defined, that is, for every iteration index $k$, 
\begin{itemize}
	\item[(i)] the discretized design problem~\eqref{eq:COED(X^k)-short} has a saddle point, and
	
	\item[(ii)] the sensitivity minimization problem~\eqref{eq:SM(xi^k,lambda^k)-general} has a $\ul{\delta}_k$-approximate solution.
\end{itemize} 
\end{lm} 

\begin{proof}
Suppose that Algorithm~\ref{algo:COED-general} is initialized with tolerances $\ul{\delta}_k, \eps$ and an initial discretization $X^0$ as described above. Also, let $K$ denote the set of iteration indices.
\smallskip

(i) We first show that for every $k \in K$ the discretized design problem $\CD(X^k)$ has a saddle point $(\xi^k,\lambda^k)$. 
In order to do so, we show that the assumptions of our saddle-point existence result (Theorem~\ref{thm:solutions-yield-saddle-points}) are satisfied with $X$ replaced by $X^k$. So, let $k \in K$ be  an arbitrary iteration index of the algorithm. 
We first observe that Conditions~\ref{cond:objective-and-constraint-fcts} and~\ref{cond:sensi-fcts} are also satisfied with $X$ replaced by any closed subset $X^* \ne \emptyset$ of $X$. In particular, Condition~\ref{cond:objective-and-constraint-fcts} is satisfied for the finite (hence closed) subset $X^* := X^k$ (which is non-empty because $X^k \supset X^0 \ne \emptyset$). 
We further observe that 
\begin{align}
\Xifin(X^k) \cap \Xifeas(X^k) \supset \Xifin(X^0) \cap \Xifeas(X^0) 
\ne \emptyset
\end{align} 
because $X^k \supset X^0$ by the definition of our algorithm and because $\eta^0$, as a strictly feasible point of $\CD(X^0)$ with $\eta^0 \in \dom \Psi_0$, belongs to $\Xifin(X^0) \cap \Xifeas(X^0)$. In view of these observations, it follows by the solvability result (Theorem~\ref{thm:solvability}) with $X$ replaced by $X^k$  that the problem $\CD(X^k)$ has a solution $\xi^k$. 
It further follows by the proof of Lemma~\ref{lm:sufficient-conds-for-lin-slater-and-non-one-sidedness-cond} (with $X^* := X^0$ and $X^{**} := X^k$) that Condition~\ref{cond:slater-and-non-one-sidedness} is satisfied with $X$ replaced by $X^k$. 
Consequently, by the saddle-point existence result (Theorem~\ref{thm:solutions-yield-saddle-points}) with $X$ replaced by $X^k$, we wee that there exists a multiplier $\lambda^k \in \Lambda$ such that $(\xi^k,\lambda^k)$ is a saddle point of $\CD(X^k)$. In particular, $\CD(X^k)$ has a saddle point, as desired. 
\smallskip

(ii) We now show that for every $k \in K$ the sensitivity minimization problem~\eqref{eq:SM(xi^k,lambda^k)-general} has a $\ul{\delta}_k$-approximate solution. 
So, let $k \in K$ be an arbitrary iteration index. Since the discretized design problem~\eqref{eq:COED(X^k)-short} has a saddle point $(\xi^k,\lambda^k)$ by part (i) and since $\xi^k \in \Xifin(X^k) \subset \Xifin(X)$ by the very definition of saddle points, we see from~\eqref{eq:inf-psiL-is-finite} that
\begin{align}
	\inf_{x \in X} \psi^L(\xi^k,\lambda^k,x) \in \R
\end{align} 
is a finite real number. So, by the assumed strict positivity of $\ul{\delta}_k$, 
there exists an $x^k \in X$ with $\psi^L(\xi^k,\lambda^k,x^k) \le \inf_{x \in X} \psi^L(\xi^k,\lambda^k,x) + \ul{\delta}_k$, that is, a $\ul{\delta}_k$-approximate solution of~\eqref{eq:SM(xi^k,lambda^k)-general}, as desired.
\end{proof}

\subsection{Comparison with known constrained-design  algorithms}

In this section, we compare our algorithms with the most closely related constrained-design algorithms from the literature, namely the algorithms from~\cite{CoFe95, Ga86, MoZu02}. As these papers, in parts, use a rather different notation and, in the case of~\cite{Ga86, MoZu02}, a  different notion of sensitivity functions, we spell out the algorithms here for the sake of a more convenient and transparent comparison. 
See Sections~\ref{sec:Cook-Fedorov}-\ref{sec:Molchanov-Zuyev} for the spelled-out algorithms and Table~\ref{tab:algorithm-comparison} for a quick overview of the algorithmic differences. 
\smallskip

\begin{center}
	\begin{table}
		\caption{Subproblems occurring in the iterations of the compared algorithms}
		\label{tab:algorithm-comparison}
		\begin{tabularx}{\linewidth}{*{5}{C}}
			\toprule
			& Algorithm by Cook and Fedorov & Algorithm by Gaivoronski & Algorithm by Molchanov and Zuyev & Algorithm proposed here \\
			\midrule
			Update of the design based on & solution of descent direction problem and step-size optimization & solution of descent direction problem and prescribed step sizes & solution of descent direction problem in the space of signed measures & solution of discretized design optimization problem\\
			\midrule
			Structure and dimension of the design updating problem & convex $1$-dimensional (step-size optimization) & - & - & convex $|X^k|$-dimensional (discretized design optimization)\\
			\midrule
			Structure and dimension of the descent direction problem & nonlinear $(|\Iineq|+1) (\dimx + 1)$-dimensional  & linear $|X^{k+1}|$-dimensional & nonlinear $(|\Ieq|+1) \dimx$-dimensional & -\\
			\midrule
			Structure and dimension of the discretization updating problem & - & nonlinear $\dimx$-dimensional & - & nonlinear $\dimx$-dimensional\\
			\midrule
			Approximate or exact solution of discretization updating problem & - & exact solution in every iteration & - & approximate solution, only in iterations where no violator is found\\
			\bottomrule
		\end{tabularx}
	\end{table}
\end{center}


An important -- and arguably the most important --  algorithmic difference 
concerns the way how the design is updated in every iteration. Specifically, in our algorithms, the design is re-optimized in every iteration by re-solving the discretized design problem for the new discretization, analogously to~\cite{ScSeBo, YaBiTa13}. In~\cite{CoFe95, Ga86, MoZu02}, by contrast, the new design 
\begin{align}
	\xi^{k+1} = \xi^k + \alpha_k (\eta^k - \xi^k)
\end{align}
is obtained from the previous design $\xi^k$ by proceeding in a suitable descent direction $\eta^k - \xi^k$ with a suitable step size $\alpha_k$, analogously to the classical vertex-direction algorithm~\cite{Fe, Wy70}. In fact, in the special case without constraints, the algorithms from~\cite{CoFe95, Ga86, MoZu02} all reduce to the classical vertex-direction algorithm (each with a different step-size rule). 
\smallskip

\begin{center}
	\begin{table}[h!]
		\caption{Settings covered by the compared algorithms}
		\label{tab:setting-comparison}
		\begin{tabularx}{\linewidth}{*{5}{C}}
			\toprule
			& Algorithm by Cook and Fedorov & Algorithm by Gaivoronski & Algorithm by Molchanov and Zuyev & Algorithm proposed here \\
			\midrule
			Convex inequality constraints & no & yes, but boundedness of multiplier sequence established only for affine constraints & no & yes\\
			\midrule
			Affine inequality constraints & yes & yes & no & yes\\
			\midrule
			Affine equality constraints & yes & no & yes & yes\\
			\midrule
			Infinite criterion or constraint value & yes & no & no & yes\\
			\midrule
			Discontinuous sensitivity functions & no & no & no & yes\\
			\bottomrule
		\end{tabularx}
	\end{table}
\end{center}

Another important algorithmic difference concerns the computational effort coming with the nonlinear subproblems of the algorithms. Specifically, in our algorithms the nonlinear subproblems are the sensitivity minimization problems~\eqref{eq:SM(xi^k,lambda^k)-general}, which are obviously $\dimx$-dimensional. 
In the algorithms from~\cite{CoFe95, MoZu02}, by contrast, the nonlinear subproblems are the descent-direction problems~\eqref{eq:CoFe-descent-direction-finite-dimensional} and~\eqref{eq:MoZu-descent-direction-proxy}, which have the -- often considerably -- larger dimensions 
\begin{align}
	(|\Iineq|+1) (\dimx + 1)
	\qquad \text{and} \qquad
	(|\Ieq|+1) \dimx,
\end{align}
respectively. In~\cite{Ga86}, in turn, the nonlinear subproblems~\eqref{eq:Ga-discretization-updating-problem} serve the purpose of updating a discretization just like our sensitivity minimization problems~\eqref{eq:SM(xi^k,lambda^k)-general} and, in fact, these problems are equivalent to~\eqref{eq:SM(xi^k,lambda^k)-general} (even though they arise from a different optimality condition of a different optimization problem, namely~\eqref{eq:Ga-linearized-descent-direction} and~\eqref{eq:Ga-minimax}). What is different in~\cite{Ga86} is that the nonlinear subproblems occurring there are solved exactly in every iteration, whereas our algorithms require only approximate solutions and, in the case of the general algorithm, only in those iterations where the local search for violators of~\eqref{eq:eps-optimality-condition-intro} fails.
\smallskip

\begin{center}
	\begin{table}[h!]
		\caption{Convergence analysis carried out for the compared algorithms 
		}
		\label{tab:convergence-analysis-comparison}
		\begin{tabularx}{\linewidth}{*{5}{C}}
			\toprule
			& Algorithm by Cook and Fedorov & Algorithm by Gaivoronski & Algorithm by Molchanov and Zuyev & Algorithm proposed here \\
			\midrule
			Iterates guaranteed to be feasible & yes & no & yes & yes\\
			\midrule
			Convergence or termination results & convergence result for $(\Psi_0(\xi^k))$ & convergence result for $(\Psi_0(\xi^k))$ & none & convergence result for $(\Psi_0(\xi^k))$ and $(\xi^k)$ and termination result\\
			\bottomrule
		\end{tabularx}
	\end{table}
\end{center}

Apart from these purely algorithmic differences, there are also important differences concerning the setting covered by the algorithms and the convergence results for the algorithms. See Sections~\ref{sec:Cook-Fedorov}-\ref{sec:Molchanov-Zuyev} for details and Tables~\ref{tab:setting-comparison} and~\ref{tab:convergence-analysis-comparison} for a quick overview. A short remark on the affine equality constraints in Table~\ref{tab:setting-comparison} is in order here. Affine equality constraints are not explicitly considered in~\cite{CoFe95, Ga86}, but the convergence result of~\cite{CoFe95} allows their reformulation as pairs of affine inequality constraints, whereas the convergence result of~\cite{Ga86} does not (because of its strict feasibility assumption).

\subsubsection{Algorithm by Cook and Fedorov}
\label{sec:Cook-Fedorov}

In~\cite{CoFe95} (Section 2.2), the following algorithm is proposed to solve optimal design problems~\eqref{eq:COED(X)} with affine  inequality constraints $\Psi_i$ with $\dom\Psi_i = \Xi(X)$. 
Specifically, the setting considered there can be described as the special case of Conditions~\ref{cond:objective-and-constraint-fcts} and~\ref{cond:sensi-fcts} where 
\begin{equation} \label{eq:CoFe-setting}
	\begin{gathered} 
		I = \Iineq, \qquad \Ieq = \emptyset, \qquad \Inc = \emptyset, \qquad \Xifin(X) = \dom \Psi_0\\
		\Psi_i(\xi) := \int_X g_i(x) \d\xi(x) 
		\qquad (\xi \in \Xi(X))
	\end{gathered}
\end{equation}
for continuous functions $g_i: X \to \R$ for $i \in I$. 
In essence, this algorithm repeatedly solves the linear descent-direction problems 
\begin{align} \label{eq:CoFe-descent-direction}
	\min_{\eta \in \Xi(X)} \partial_{\eta - \xi^k} \Psi_0(\xi^k) 
	\quad \text{s.t.} \quad
	\Psi_i(\eta) \le 0 \qquad (i \in \Iineq)
\end{align}
of the original problem and exploits the fact that -- by the affinity of the constraints -- each of these problems has a solution with at most $|\Iineq| + 1$ support points (Note 1 of~\cite{CoFe95}). So, by restricting the search space of~\eqref{eq:CoFe-descent-direction} to designs $\eta$ with $|\supp \eta| \le |\Iineq| + 1$, the infinite-dimensional linear optimization problem can be replaced by a finite-dimensional (nonlinear) problem, namely~\eqref{eq:CoFe-descent-direction-finite-dimensional} below. In this context, we abbreviate $m := |\Iineq|$ and denote the $m$-dimensional standard probability simplex by
\begin{align}
	\Delta_m := \{w \in [0,\infty)^{m+1}: w_1 + \dotsb + w_{m+1} = 1\}.
\end{align}

\begin{algo} \label{algo:CoFe}
	Input: a design $\xi^0 \in \Xifin(X) \cap \Xifeas(X)$. 
	With this input at hand, set $k = 0$ and proceed in the following steps:
	\begin{itemize}
		\item[1.] Compute a solution $\eta^k \in \Xi(X)$ of the linear descent-direction problem~\eqref{eq:CoFe-descent-direction} with $|\supp \eta^k| \le m+1$. In order to do so, compute a solution $(\tilde{x}^k,w^k) = (x^{k,1}, \dots, x^{k,m+1}, w^k) \in X^{m+1} \times \Delta_m$ of the nonlinear optimization problem
		\begin{align} \label{eq:CoFe-descent-direction-finite-dimensional}
			\min_{(\tilde{x},w) \in X^{m+1} \times \Delta_m} \sum_{j=1}^{m+1} w_j \psi_0(\xi^k,x^j)
			\quad \text{s.t.} \quad
			\sum_{j=1}^{m+1} w_j g_i(x^j) \le 0 
			\qquad (i \in \Iineq)
		\end{align}
		and then set $\eta^k := \sum_{j=1}^{m+1} w_j^k \delta_{x^{k,j}}$.
		
		\item[2.] Compute a solution $\alpha_k \in [0,1]$ of the convex step-size problem
		\begin{align}
			\min_{\alpha \in [0,1]} \Psi_0(\xi^k + \alpha(\eta^k-\xi^k)).
		\end{align}
		In case the minimal value $\partial_{\eta^k-\xi^k}\Psi_0(\xi^k)$ of~\eqref{eq:CoFe-descent-direction} or, equivalently, of~\eqref{eq:CoFe-descent-direction-finite-dimensional} is non-negative, then terminate. In the opposite case, update the design by setting
		\begin{align}
			\xi^{k+1} := (1-\alpha_k) \xi^k + \alpha_k \eta^k
		\end{align}
		and return to Step 1 with $k$ replaced by $k+1$.
	\end{itemize}
\end{algo}

It is straightforward to verify that the iterates $\xi^k$ of the above algorithm are all feasible for~\eqref{eq:COED(X)} 
and belong to $\Xifin(X)$, in short: $\xi^k \in \Xifin(X) \cap \Xifeas(X)$ for all iteration indices $k$. 
%
It is also not difficult to prove that 
in case of termination, the algorithm terminates at a solution of~\eqref{eq:COED(X)}. Indeed, this follows from the easily established fact that for $\xi^k \in \Xifin(X) \cap \Xifeas(X)$, the non-negativity of the optimal value of~\eqref{eq:CoFe-descent-direction} implies the optimality of $\xi^k$ for~\eqref{eq:COED(X)}. 
\smallskip

Concerning convergence of the algorithm, the following result is established 
in~\cite{CoFe95} (Theorem~7): In the aforementioned special case~\eqref{eq:CoFe-setting} of Conditions~\ref{cond:objective-and-constraint-fcts} and~\ref{cond:sensi-fcts}, the criterion values $\Psi_0(\xi^k)$ converge monotonically decreasingly to the optimal value of~\eqref{eq:COED(X)},
\begin{align}
	\Psi_0(\xi^k) \searrow \min_{\xi\in\Xifeas(X)} \Psi_0(\xi) \qquad (k\to\infty),
\end{align}
provided that a suitable second-order strengthening of the directional differentiability assumption~\eqref{eq:directional-differentiability-def} holds true for $\Psi_0$.

\subsubsection{Algorithm by Gaivoronski}
\label{sec:Gaivoronski}

In~\cite{Ga86} (Section 5), the following algorithm is proposed to solve optimal design problems~\eqref{eq:COED(X)} with convex inequality constraints $\Psi_i$ with $\dom \Psi_i = \Xi(X)$. 
Specifically, the setting considered there can be described as the special case of Conditions~\ref{cond:objective-and-constraint-fcts} and~\ref{cond:sensi-fcts} where 
\begin{align} \label{eq:Ga-setting}
	I = \Iineq, \qquad \Ieq = \emptyset, \qquad \Inc = \emptyset, \qquad \Xifin(X) = \Xi(X).
\end{align}
In essence, this algorithm repeatedly solves discretized versions of the linearized descent-direction problems 
\begin{align} \label{eq:Ga-linearized-descent-direction}
	\min_{\eta \in \Xi(X)} \partial_{\eta - \xi^k} \Psi_0(\xi^k) 
	\quad \text{s.t.} \quad
	\Psi_i(\xi^k) + \partial_{\eta-\xi^k} \Psi_i(\xi^k) \le 0 \qquad (i \in \Iineq)
\end{align}
of the original problem and exploits the fact that -- under a strict feasibility condition -- the optimal value $v^k$ of~\eqref{eq:Ga-linearized-descent-direction} is equal to the optimal value of the minimax problem
\begin{align} \label{eq:Ga-minimax}
	\max_{\lambda \in [0,\infty)^m} \min_{x \in X} \ul{\psi}^L(\xi^k,\lambda,x)
\end{align}
with $m := |\Iineq|$ (Theorem~1 of~\cite{Ga86}). 
In this problem, $\ul{\psi}^L(\xi^k,\cdot,\cdot)$ denotes the integral kernel 
of the Lagrange function of~\eqref{eq:Ga-linearized-descent-direction}, that is, 
\begin{align}
	\ul{\psi}^L(\xi^k,\lambda,x) := \psi_0(\xi^k,x) + \sum_{i \in \Iineq} \lambda_i (\Psi_i(\xi^k) + \psi_i(\xi^k,x))
	\qquad ((\lambda,x) \in [0,\infty)^m \times X).
\end{align}
In order to solve~\eqref{eq:Ga-minimax}, an iterative discretization scheme for $X$ is proposed in~\cite{Ga86}, which then leads to the following algorithm (Section~5 of~\cite{Ga86}).

\begin{algo} \label{algo:Ga}
	Input: a design $\xi^0 \in \Xi(X)$, a multiplier $\lambda^0 \in [0,\infty)^m$, a finite subset $X^0$ of $X$, and step sizes $\alpha_k \in (0,\infty)$ with $\alpha_k \longrightarrow 0$ and with $\sum_{k=0}^\infty \alpha_k = \infty$.  
	With these inputs at hand, set $k = 0$ and proceed in the following steps:
	\begin{itemize}
		\item[1.] Compute a solution $x^k \in X$ of the discretization updating problem
		\begin{align}
			\label{eq:Ga-discretization-updating-problem}
			\min_{x \in X} \ul{\psi}^L(\xi^k,\lambda^k,x)
		\end{align}
		and update the discretization by setting $X^{k+1} := X^k \cup \{x^k\}.$
		
		\item[2.] Compute a solution $\eta^k \in \Xi(X^{k+1})$ of the discretized linearized descent-direction problem
		\begin{align} \label{eq:Ga-discretized-linearized-descent-direction}
			\min_{\eta \in \Xi(X^{k+1})} \int_X \psi_0(\xi^k,x) \d\eta(x)
			\quad \text{s.t.} \quad
			\Psi_i(\xi^k) + \int_X \psi_i(\xi^k,x) \d\eta(x) \le 0 \qquad (i \in \Iineq)
		\end{align}
		and a solution $\lambda^{k+1} \in [0,\infty)^m$ of the optimization problem 
		\begin{align} \label{eq:Ga-dual-of-discretized-linearized-descent-direction}
			\max_{\lambda \in [0,\infty)^m} \min_{x \in X^{k+1}} \ul{\psi}^L(\xi^k,\lambda,x),
		\end{align}
		which is easily seen to be the dual optimization problem of~\eqref{eq:Ga-discretized-linearized-descent-direction}.
		
		\item[3.] Update the design according to the updating rule
		\begin{align}
			\xi^{k+1} := (1-\alpha_k) \xi^k + \alpha_k \eta^k
		\end{align}
		and return to Step 1 with $k$ replaced by $k+1$. 
	\end{itemize}
\end{algo}

It is important to notice that the iterates $\xi^k$ of the above algorithm are generally not feasible for~\eqref{eq:COED(X)}, even if the initial design $\xi^0$ is chosen to be feasible. 
In fact, feasibility is guaranteed only in the limit $k\to\infty$ or, more precisely, only for the accumulation points of the iterate sequence $(\xi^k)$ (Theorem~3 in~\cite{Ga86} and the remark following its proof). 
It is therefore important that the algorithm does not terminate, 
which is why it does not feature a termination condition. 
\smallskip

Concerning convergence of the algorithm, the following result is established 
in~\cite{Ga86} (Theorem~5): In the aforementioned special case~\eqref{eq:Ga-setting} of Conditions~\ref{cond:objective-and-constraint-fcts} and~\ref{cond:sensi-fcts}, the criterion values $\Psi_0(\xi^k)$ converge to the optimal value of~\eqref{eq:COED(X)},
\begin{align}
	\Psi_0(\xi^k) \longrightarrow \min_{\xi\in\Xifeas(X)} \Psi_0(\xi) \qquad (k\to\infty),
\end{align}
provided that (i) there exists an $\eta^0 \in \Xi(X)$ such that $\Psi_i(\eta^0) < 0$ for all $i \in \Iineq$ 
and that (ii) the multiplier sequence $(\lambda^k)$ generated by the algorithm is bounded. As a sufficient condition for this boundedness assumption, the affinity of all constraint functions $\Psi_i$ is mentioned in~\cite{Ga86}. In the case of non-affine constraint functions, however, no sufficient conditions for the boundedness of the multiplier sequence are given in~\cite{Ga86}. 

\subsubsection{Algorithm by Molchanov and Zuyev}
\label{sec:Molchanov-Zuyev}

In~\cite{MoZu02} (Section 6), the following algorithm is proposed to solve optimal design problems~\eqref{eq:COED(X)} with a not necessarily convex criterion $\Psi_0$ and affine equality constraints $\Psi_i$ with $\dom\Psi_i = \Xi(X)$. Specifically, the design problem is formulated in the form
\begin{align} \label{eq:COED(X)-MoZu}
	\min_{\xi \in \nonnegmsre} \Psi_0(\xi) 
	\quad \text{s.t.} \quad
	\Psi_i(\xi) = a_i \text{ for all } i \in \Ieq \text{ and } \xi(X) = 1 
\end{align}
with $\nonnegmsre$ being the set of non-negative finite measures on $\B_X$, and the following setting is assumed:
\begin{itemize}
	\item[(i)] $\Psi_0: \Xi(X) \to \R$ is continuous (in the weak topology as usual) and extends to a Fr\'{e}chet  differentiable mapping $\ol{\Psi}_0: \signedmsre \to \R$ on the Banach space $\signedmsre$ of signed measures (endowed with the total-variation norm) whose directional derivatives $\partial_\eta \ol{\Psi}_0(\xi)$ for $\xi, \eta \in \signedmsre$ are given by integrals of the form
	\begin{align}
		\partial_\eta \ol{\Psi}_0(\xi) = \int_X \ol{\psi}_0(\xi,x) \d\eta(x)
	\end{align}
	with continuous functions $\ol{\psi}_0(\xi,\cdot): X \to \R$. (In analogy to~\eqref{eq:directional-derivative-integral-form}, one could call the function $\ol{\psi}_0: \signedmsre \times X \to \R$ the signed sensitivity function of $\Psi_0$.) 
	
	\item[(ii)] $\Psi_i$ for $i \in \Ieq$ is an affine mapping given by
	\begin{align}
		\Psi_i(\xi) := \int_X g_i(x) \d\xi(x)
		\qquad (\xi \in \Xi(X))
	\end{align}
	with continuous functions $g_i: X \to \R$ for $i \in \Ieq$. 
\end{itemize}
In essence, the algorithm from~\cite{MoZu02} repeatedly computes suitable approximate solutions of the steepest descent-direction problems
\begin{align} \label{eq:MoZu-descent-direction}
	\min_{\eta \in \signedmsre} \partial_\eta \Psi_0(\xi^k) 
	\quad \text{s.t.} \quad
	\xi^k + \eta \in \nonnegmsre \text{ and } 
	\Psi_i(\xi^k + \eta) = a_i \qquad (i \in \{1,\dots,m+1\})
\end{align}
in the space of signed measures, where $\Ieq = \{1,\dots,m\}$ and $\Psi_{m+1}(\xi) := \xi(X)$ and $a_{m+1} := 1$. 
In doing so, a characterization of steepest descent directions for~\eqref{eq:COED(X)-MoZu} is exploited  (Theorem~4.1 of~\cite{MoZu02}). Also, the termination criterion of the algorithm is inspired by the following necessary optimality condition for~\eqref{eq:COED(X)-MoZu}: if $\xi^k$ is a solution of~\eqref{eq:COED(X)-MoZu} that satisfies the Robinson-Zowe-Kurcyusz regularity condition, then there exists a multiplier $\lambda^k \in \R^{m+1}$ such that
\begin{align} \label{eq:MoZu-necessary-optimality-condition}
	\min_{x \in X} \ol{\psi}^L(\xi^k,\lambda^k,x) \ge 0 
	\qquad \text{and} \qquad
	\ol{\psi}^L(\xi^k,\lambda^k,x) = 0 \text{ for $\xi^k$-a.e. } x \in X
\end{align} 
(Theorem~2.1 of~\cite{MoZu02}). In the above relations, $\ol{\psi}^L$ denotes the integral kernel of the directional derivatives -- or, in other words, the signed sensitivity function -- of the Lagrange function of~\eqref{eq:COED(X)-MoZu}, 
that is,
\begin{align}
	\ol{\psi}^L(\xi,\lambda,x) := \ol{\psi}_0(\xi,x) + \sum_{i=1}^{m+1} \lambda_i g_i(x)
	\qquad ((\xi,\lambda,x) \in \signedmsre \times \R^{m+1} \times X).
\end{align}

\begin{algo}
	Input: a design $\xi^0 \in \Xifeas(X)$ with finite support and optimality tolerances $\delta, \delta_{\supp} \in (0,\infty)$. 
	With these inputs at hand, set $k = 0$ and proceed in the following steps:
	\begin{itemize}
		\item[1.] Compute an approximate solution $\eta^k$ of the descent-direction problem~\eqref{eq:MoZu-descent-direction}. In order to do so, compute a solution $\tilde{x}^k = (x^{k,1},\dots,x^{k,m+1}) \in X^{m+1}$ of the nonlinear optimization problem
		\begin{equation} \label{eq:MoZu-descent-direction-proxy}
			\begin{aligned}
				\min_{\tilde{x} \in X^{m+1}} \sum_{j=1}^{m+1} w_j(\tilde{x}) \ol{\psi}_0(\xi^k,x^j) 
				\quad \text{s.t.} \quad 
				w_j(\tilde{x}) \ge 0 \qquad (j \in \{1,\dots,m+1\})
			\end{aligned}
		\end{equation}
		and then set 
		$\eta^k := \sum_{j=1}^{m+1} w_j(x^k) \delta_{x^{k,j}}$, 
		where the following abbreviations have been used:
		\begin{align*}
			w_j(\tilde{x}) := e_j^\top G(\tilde{x})^{-1} \begin{pmatrix}
				a_1 \\ \vdots \\ a_{m+1}
			\end{pmatrix}
			\qquad \text{and} \qquad
			G(\tilde{x}) := \begin{pmatrix}
				g_1(x^1) & \dots & g_1(x^{m+1}) \\
				\vdots   & 		 & \vdots \\
				g_{m+1}(x^1) & \dots & g_{m+1}(x^{m+1})
			\end{pmatrix}
		\end{align*}
		for $\tilde{x} = (x^1,\dots,x^{m+1}) \in X^{m+1}$ with $e_j \in \R^{m+1}$ the $j$th canonical unit vector and with $g_{m+1}(y) := 1$ for $y \in X$.

		\item[2.] Compute a solution $\lambda^k \in \R^{m+1}$ of the quadratically convex least-squares problem
		\begin{align}
			\min_{\lambda \in \R^{m+1}} \sum_{x \in \supp \xi^k} (\ol{\psi}^L(\xi^k,\lambda,x))^2
		\end{align}
		inspired by~(\ref{eq:MoZu-necessary-optimality-condition}.b).
		In case 
		\begin{align}
			\xi^k\Big( \Big\{ x \in X: \ol{\psi}^L(\xi^k,\lambda^k,x) - \min_{x \in X} \ol{\psi}^L(\xi^k,\lambda^k,x) > \delta \Big\} \Big) < \delta_{\supp}
		\end{align}
		and $\Psi_0((1-\alpha)\xi^k + \alpha \eta^k) \ge \Psi_0(\xi^k)$ for all $\alpha \in (0,1]$, terminate. In the opposite case, 
		update the design by setting
		\begin{align}
			\xi^{k+1} := (1-\alpha_k)\xi^k + \alpha_k \eta^k,
		\end{align}
		where $\alpha_k \in (0,1]$ is any step size with $\Psi_0((1-\alpha_k)\xi^k + \alpha_k \eta^k) < \Psi_0(\xi^k)$, and then return to Step 1 with $k$ replaced by $k+1$.		
	\end{itemize}
\end{algo}

It is straightforward to verify that the iterates $\xi^k$ of the above algorithm are all feasible for~\eqref{eq:COED(X)-MoZu}. In this context, it should be noticed, however, 
that the inverse of the matrix $G(\tilde{x})$ used in~\eqref{eq:MoZu-descent-direction-proxy} does not exist in general, 
contrary to what is stated before (6.1) in~\cite{MoZu02}. (Indeed, let $X := [0,(m+2)\pi]$ for some $m \in \N$ and let $g_i := \cos(i\cdot)|_X$ for $i \in \{1,\dots,m\}$ and $g_{m+1} := 1$. Since $g_1,\dots,g_{m+1}$ when restricted to $[0,2\pi]$ form an orthogonal system, the functions $g_1, \dots, g_{m+1}$ are linearly independent on $X$, as required before (2.4) of~\cite{MoZu02}. Setting $x^j := \pi/2 + j \pi \in X$ and $\tilde{x} := (x^1,\dots,x^{m+1})$, however, we see that the $i$th row of $G(\tilde{x})$ is identically zero for every odd index $i \in \{1,\dots,m\}$. And therefore, $G(\tilde{x})$ is not invertible.) 
\smallskip

Concerning convergence of the algorithm, not much is said in~\cite{MoZu02}. Indeed, the algorithm is illustrated and validated by means of several interesting examples in~\cite{MoZu02} (Section 7), which go beyond experimental design (but all of them feature convex objective functions as in experimental design). It should be pointed out, though, that no general convergence or termination results are proved or even formulated in~\cite{MoZu02}. In fact, 
at the very end of~\cite{MoZu04}, it is emphasized that an additional analysis is necessary to ensure that the algorithm really does converge to a solution of~\eqref{eq:COED(X)-MoZu}.

\section{Convergence and termination results}

In this section, we establish convergence and finite termination results for our adaptive discretization algorithm. Specifically, we show that the iterate sequence of the special algorithm (Algorithm~\ref{algo:COED-special}) accumulates at an optimal design for~\eqref{eq:COED(X)} and that the iterate sequence of the general algorithm (Algorithm~\ref{algo:COED-general}), in turn, terminates at an $\eps$-optimal design. We begin with two preparatory lemmas.

\begin{lm} \label{lm:properties-of-constr-fcts-depending-on-sensi-fct-continuity}
Suppose that Conditions~\ref{cond:objective-and-constraint-fcts} and~\ref{cond:sensi-fcts} are satisfied and as in~\eqref{eq:Ic-and-Inc-definition} let $\Ic$ and $\Inc$ denote the subset of $I$ for which the sensitivity function $\psi_i$ is continuous or non-continuous, respectively. Then 
\begin{itemize}
\item[(i)] for $i \in \{0\} \cup \Ic$, the objective or constraint function $\Psi_i|_{\dom \Psi_i}$, respectively, is continuous
\item[(ii)] for $i \in \Inc$, the constraint function $\Psi_i|_{\dom \Psi_i}$  is given by
\begin{align} \label{eq:properties-of-constr-fcts-depending-on-sensi-fct-continuity}
\Psi_i(\xi) = \int_X g_i(x) \d\xi(x) \qquad (\xi \in \dom\Psi_i). 
\end{align}
\end{itemize}
\end{lm}

\begin{proof}
(i) Suppose that $i \in \{0\} \cup \Ic$ and let $\xi^n, \xi \in \dom \Psi_i$ with $\xi^n \longrightarrow \xi$ as $n \to \infty$. 
It then follows 
by Lemma~\ref{lm:bound-on-differences-of-Psiis-in-terms-of-sensi-fcts} that 
\begin{align*}  
\Psi_i(\xi^n) - \Psi_i(\xi) \le -\int_X \psi_i(\xi^n,x) \d\xi(x) 
\qquad \text{and} \qquad
\Psi_i(\xi) - \Psi_i(\xi^n) \le -\int_X \psi_i(\xi,x) \d\xi^n(x) 
\end{align*}
for all $n \in \N$. Consequently,
\begin{align} \label{eq:properties-of-constr-fcts-depending-on-sensi-fct-continuity-step2}
|\Psi_i(\xi^n) - \Psi_i(\xi)| \le  \max\bigg\{ -\int_X \psi_i(\xi^n,x) \d\xi(x), -\int_X \psi_i(\xi,x) \d\xi^n(x) \bigg\} 
\end{align}
for all $n \in \N$. Since $\psi_i(\cdot,x)$ is continuous for every $x \in X$ and since $\psi_i(\xi,\cdot)$ is a bounded continuous function by Condition~\ref{cond:sensi-fcts}~(i) or, respectively, by the definition of $\Ic$, we conclude with the dominated convergence theorem that the right-hand side of~\eqref{eq:properties-of-constr-fcts-depending-on-sensi-fct-continuity-step2} converges to
\begin{align} \label{eq:properties-of-constr-fcts-depending-on-sensi-fct-continuity-step3}
\int_X \psi_i(\xi,x) \d\xi(x) = \partial_{\xi-\xi}\Psi_i(\xi) = 0
\end{align}
as $n \to \infty$. So, the same is true for the left-hand side of~\eqref{eq:properties-of-constr-fcts-depending-on-sensi-fct-continuity-step2}, whence $\Psi_i(\xi^n) \longrightarrow \Psi_i(\xi)$ as $n \to \infty$, as desired.
\smallskip

(ii) Suppose that $i \in \Inc$. It then follows by Condition~\ref{cond:sensi-fcts} that there exists a bounded measurable function such that 
\begin{align}
\psi_i(\xi,x) = g_i(x) - \Psi_i(\xi) \qquad ((\xi,x) \in \dom\Psi_i \times X). 
\end{align}
In particular, it follows that
\begin{align}
0 = \partial_{\xi-\xi}\Psi_i(\xi) = \int_X \psi_i(\xi,x)\d\xi(x) = \int_X g_i(x)\d\xi(x) - \Psi_i(\xi) \qquad (\xi \in \dom\Psi_i),
\end{align}
which is the desired conclusion~\eqref{eq:properties-of-constr-fcts-depending-on-sensi-fct-continuity}. 
\end{proof}

\begin{lm} \label{lm:xik-optimal-and-psiL-ge-0-on-Xk}
Suppose that Conditions~\ref{cond:objective-and-constraint-fcts} and \ref{cond:sensi-fcts} are satisfied. Suppose further that $(\xi^k)$, $(\lambda^k)$, $(x^k)$ are generated by Algorithm~\ref{algo:COED-general} with some tolerances $\ul{\delta}_k \in (0,\infty)$ and $\eps \in [0,\infty)$ and an initial finite subset $X^0 \ne \emptyset$ such that Condition~\ref{cond:reg-conditions-on-X^0} is satisfied. We then have the following assertions.
\begin{itemize}
\item[(i)] $\xi^k$ is an optimal solution of the discretized design problem~\eqref{eq:COED(X^k)-short} and  
\begin{align} \label{eq:xi^k-in-Xifin(Xk)-cap-Xifeas(Xk)-and-complementarity-relation-for-xi^k}
	\xi^k \in \Xifin(X^k) \cap \Xifeas(X^k)
	\qquad \text{and} \qquad
	\lambda_i^k \Psi_i(\xi^k) = 0 \qquad (i \in \Iineq)
\end{align}
for every iteration index $k \in K$.

\item[(ii)] $\Psi_0(\xi^{k+1}) \le \Psi_0(\xi^k) < \infty$ for all iteration indices $k, k+1 \in K$. In particular, $(\Psi_0(\xi^k))_{k\in K}$ is monotonically decreasing.

\item[(iii)] $\psi^L(\xi^k,\lambda^k,x) \ge 0$ for every $x \in X^k$ and every iteration index $k \in K$. In particular, 
\begin{align} \label{eq:psiL(xi^l,lambda^l,x^k)-ge-0}
	\psi^L(\xi^{l},\lambda^l,x^k) \ge 0
\end{align}
for all iteration indices $k,l \in K$ with $k < l$. 
\end{itemize}
\end{lm}

\begin{proof}
(i) Since $(\xi^k,\lambda^k)$ is a saddle point of~\eqref{eq:COED(X^k)-short} for every $k \in K$ by the algorithm's definition, it follows that $\xi^k \in \Xifin(X^k)$ for every $k \in K$. It further follows by Lemma~\ref{lm:saddle-points-yield-solutions} (with $X$ replaced by $X^k$) that $\xi^k$ is an optimal solution of~\eqref{eq:COED(X^k)-short} and $\lambda_i^k \Psi_i(\xi^k) = 0$ for every $i \in \Iineq$. In particular, \eqref{eq:xi^k-in-Xifin(Xk)-cap-Xifeas(Xk)-and-complementarity-relation-for-xi^k} is satisfied. 
\smallskip

(ii) Since $X^k \subset X^{k+1}$ for every $k \in K$ with $k+1 \in K$ by the algorithm's definition, it  follows that
\begin{align}
\Xifeas(X^{k+1}) \supset \Xifeas(X^k) \qquad (k,k+1 \in K)
\end{align}
and therefore 
\begin{align} \label{eq:Psi0(xik)-decreasing}
\Psi_0(\xi^{k+1}) = \min_{\xi \in \Xifeas(X^{k+1})} \Psi_0(\xi) \le \min_{\xi \in \Xifeas(X^{k})} \Psi_0(\xi) = \Psi_0(\xi^k)
\end{align}
for all $k, k+1 \in K$. In particular, \eqref{eq:xi^k-in-Xifin(Xk)-cap-Xifeas(Xk)-and-complementarity-relation-for-xi^k} and~\eqref{eq:Psi0(xik)-decreasing} show that $(\Psi_0(\xi^k))_{k\in K}$ is a monotonically decreasing sequence of (finite) real numbers, as desired.
\smallskip

(iii) Since $(\xi^k,\lambda^k)$ is a saddle point of~\eqref{eq:COED(X^k)-short} for every $k \in K$ by the algorithm's definition, we have
\begin{align} \label{eq:saddle-point-cond-for-all-eta}
L(\xi^k,\lambda) \le L(\xi^k,\lambda^k) \le L(\eta,\lambda^k) 
\end{align}
for every $\eta \in \Xifin(X^k)$ and every $\lambda \in \Lambda$. Since, moreover, $\xi^k \in \Xifin(X^k) \subset \Xifin(X)$ by the definition of saddle points and since $\Xifin(X)$ is directionally open by Lemma~\ref{lm:bound-on-differences-of-Psiis-in-terms-of-sensi-fcts}, we see for every $x \in X^k$ that $\xi^k + \alpha (\delta_x-\xi^k) \in \Xifin(X^k)$ for $\alpha$ close enough to $0$. And therefore we conclude from~(\ref{eq:saddle-point-cond-for-all-eta}.b) that
\begin{align} \label{eq:psiL-ge-0-on-Xk}
\psi^L(\xi^k,\lambda^k,x) = \int_X \psi^L(\xi^k,\lambda^k,y) \d \delta_x(y) = \lim_{\alpha\searrow 0} \frac{L(\xi^k+\alpha(\delta_x-\xi^k),\lambda^k) - L(\xi^k,\lambda^k)}{\alpha} \ge 0
\end{align}
for every $x \in X^k$ and every $k \in K$. Since $x^k  \in X^{k+1} \subset X^l$ for all $k,l \in K$ with $k < l$, the relation~\eqref{eq:psiL-ge-0-on-Xk} implies in particular that $\psi^L(\xi^{l},\lambda^l,x^k) \ge 0$ for all iteration indices $k,l \in K$ with $k < l$, as desired.
\end{proof}

With these preparatory lemmas at hand, we can now establish our core convergence lemma about the sequences $(\xi^k)$, $(\lambda^k)$ and $(x^k)$ generated by our algorithm.    

\begin{lm} \label{lm:subsequence-lm}
Suppose that Conditions~\ref{cond:objective-and-constraint-fcts} and \ref{cond:sensi-fcts} are satisfied. Suppose further that $(\xi^k)$, $(\lambda^k)$, $(x^k)$ are generated by Algorithm~\ref{algo:COED-general} with some tolerances $\ul{\delta}_k \in (0,\infty)$ and $\eps \in [0,\infty)$ and an initial finite subset $X^0 \ne \emptyset$ such that Condition~\ref{cond:reg-conditions-on-X^0} is satisfied.
If $(\xi^k)$ is not terminating and $\xi^*$ is any accumulation point of $(\xi^k)$, then 
\begin{align} \label{eq:subsequence-lm-1}
	\xi^* \in \Xifin(X) \cap \Xifeas(X)
\end{align}
and there exist a subsequence $(k_l)$ and points $\lambda^* \in \Lambda$, $x^* \in X$, and $\psi_i^* \in \R$ for $i \in I \cup \{0\}$ such that
\begin{gather} 
	\xi^{k_l} \longrightarrow \xi^*, \qquad \lambda^{k_l} \longrightarrow \lambda^*, \qquad x^{k_l} \longrightarrow x^* \label{eq:subsequence-lm-2}\\
	\psi_i(\xi^{k_l},x^{k_l}) \longrightarrow \psi_i^* \qquad \text{and} \qquad \psi_i(\xi^{k_{l+1}},x^{k_l}) \longrightarrow \psi_i^* \label{eq:subsequence-lm-3}
\end{gather}
for all $i \in I \cup \{0\}$ as $l\to\infty$. In particular,  
\begin{align} \label{eq:subsequence-lm-4}
	\lim_{l\to\infty} \psi^L(\xi^{k_l},\lambda^{k_l},x^{k_l})
	= \lim_{l\to\infty} \psi^L(\xi^{k_{l+1}},\lambda^{k_{l+1}},x^{k_l}).
\end{align}
\end{lm}

\begin{proof}
We proceed in three steps and in the entire proof, we continually use the decomposition of the constraint index set $I$ into $\Ic$ and $\Inc$ as in~\eqref{eq:Ic-and-Inc-definition}. 
\smallskip

As a first step, we show that $\xi^* \in \Xifin(X) \cap \Xifeas(X)$ and that there exist a subsequence $(k_l)$ and and points $x^* \in X$, and $\psi_i^* \in \R$ for $i \in I$ such that~(\ref{eq:subsequence-lm-2}.a), (\ref{eq:subsequence-lm-2}.c) and \eqref{eq:subsequence-lm-3} are satisfied.
Indeed, as $\xi^* \in \Xi(X)$ is an accumulation point of $(\xi^k)$ and $X$ is compact (Condition~\ref{cond:objective-and-constraint-fcts}), there exist a subsequence $(k_l)$ and and a point $x^* \in X$ such that
\begin{align} \label{eq:subsequence-lm-step1.1}
	\xi^{k_l} \longrightarrow \xi^* \qquad \text{and} \qquad x^{k_l} \longrightarrow x^* 
\end{align}
as $l\to\infty$. Since $\Xifeas(X)$ is compact (Theorem~\ref{thm:solvability}) and $\xi^k \in \Xifin(X) \cap \Xifeas(X)$ for all $k \in \N_0$ (Lemma~\ref{lm:xik-optimal-and-psiL-ge-0-on-Xk}~(i)), it follows that $\xi^* \in \Xifeas(X)$. Since, moreover, $\Psi_0$ is lower semicontinuous and $(\Psi_0(\xi^k))$ is monotonically decreasing (Lemma \ref{lm:xik-optimal-and-psiL-ge-0-on-Xk}~(ii)), it follows from~\eqref{eq:xi^k-in-Xifin(Xk)-cap-Xifeas(Xk)-and-complementarity-relation-for-xi^k} and~(\ref{eq:subsequence-lm-step1.1}.a) that
\begin{align}
	\Psi_0(\xi^*) \le \liminf_{l\to\infty} \Psi_0(\xi^{k_l}) \le \Psi_0(\xi^0) < \infty.
\end{align}
And therefore
\begin{align} \label{eq:subsequence-lm-step1.2}
	\xi^* \in \dom\Psi_0 \cap \Xifeas(X) = \Xifin(X) \cap \Xifeas(X). 
\end{align}
It remains to establish the convergences~\eqref{eq:subsequence-lm-3} for a suitable subsequence of $(k_l)$ and every $i \in I \cup \{0\}$ 
and to do so we distinguish the case $i \in \Ic \cup \{0\}$ and the case $i \in \Inc$. 
In the case where $i \in \Ic \cup \{0\}$, it immediately follows from~\eqref{eq:subsequence-lm-step1.1} and \eqref{eq:subsequence-lm-step1.2} that the convergences~\eqref{eq:subsequence-lm-3} hold true for $i \in \Ic$ with $\psi_i^* := \psi_i(\xi^*,x^*) \in \R$. 
In the case where $i \in \Inc$, it follows from~\eqref{eq:subsequence-lm-step1.1} and \eqref{eq:subsequence-lm-step1.2} in conjunction with Condition~\ref{cond:sensi-fcts}~(ii) and Lemma \ref{lm:properties-of-constr-fcts-depending-on-sensi-fct-continuity}~(ii) that the sequences $(g_i(x^{k_l}))$ and $(\Psi_i(\xi^{k_l})) = (\int_X g_i(x) \d\xi^{k_l}(x))$ are bounded 
and therefore there exists a subsequence of $(k_l)$ -- again denoted by $(k_l)$ for simplicity -- such that
\begin{align} \label{eq:subsequence-lm-step1.3}
	g_i(x^{k_l}) \longrightarrow g_i^*
	\qquad \text{and} \qquad
	\Psi_i(\xi^{k_l}) \longrightarrow \Psi_i^* 
\end{align}
for some $g_i^*, \Psi_i^* \in \R$. Consequently, the convergences~\eqref{eq:subsequence-lm-3} hold true also for $i \in \Inc$ with $\psi_i^* := g_i^*-\Psi_i^* \in \R$.
\smallskip

As a second step, we show that for every subsequence $(k_l)$ as in the first step, there exists another subsequence $(k_l')$ and a $\lambda^* \in \Lambda$ such that~(\ref{eq:subsequence-lm-2}.b) holds true along this subsequence, that is,
\begin{align} \label{eq:subsequence-lm-step2}
	\lambda^{k_l'} \longrightarrow \lambda^* \qquad (l\to\infty). 
\end{align}
So, let $(k_l)$ and $x^* \in X$, and $\psi_i^* \in \R$ be such that the convergences (\ref{eq:subsequence-lm-2}.a), (\ref{eq:subsequence-lm-2}.c), \eqref{eq:subsequence-lm-3} hold true. In order to establish~\eqref{eq:subsequence-lm-step2}, it is sufficient to show that $(\lambda^{k_l})$ is bounded (because $\Lambda = [0,\infty)^{\Iineq} \times \R^{\Ieq}$ is closed). We show this boundedness by contradiction. So, assume that $(\lambda^{k_l})$ is unbounded. 
It then follows by the closedness of $\Lambda$ in $\R^{I}$ that 
there exists a subsequence $(k_l')$ of $(k_l)$ and a point $\tilde{\lambda}^* \in \Lambda$ such that 
\begin{align} \label{eq:subsequence-lm-step2-unboundedness-assumption}
	|\lambda^{k_l'}| \longrightarrow \infty 
	\qquad \text{and} \qquad
	\frac{\lambda^{k_l'}}{|\lambda^{k_l'}|} \longrightarrow \tilde{\lambda}^*
	\qquad (l\to\infty). 
\end{align} 
We deduce the desired contradiction in three substeps. 
As a first substep, we show that
\begin{align} \label{eq:subsequence-lm-step2.1}
	\tilde{\lambda}_i^* = 0 \qquad (i \in \Iineq).
\end{align}
In order to prove this, we first observe that $\min_{x \in X^k} \psi^L(\xi^k,\lambda^k,x) \ge 0$ for all $k \in \N_0$ by Lemma~\ref{lm:xik-optimal-and-psiL-ge-0-on-Xk}~(iii). In particular, this implies that
\begin{align} \label{eq:subsequence-lm-step2.1.1}
	\min_{x \in X^0} \psi^L(\xi^k,\lambda^k,x) \ge 0  \qquad (k \in \N_0)
\end{align}
because $X^0 \subset X^k$ for every $k \in \N_0$ by the algorithm's definition. It follows from~\eqref{eq:subsequence-lm-step2.1.1} with the help of Lemma~\ref{lm:bound-on-differences-of-Psiis-in-terms-of-sensi-fcts} and Lemma~\ref{lm:xik-optimal-and-psiL-ge-0-on-Xk}~(i) that
\begin{align} \label{eq:subsequence-lm-step2.1.2}
	0 &\le \int_{X^0} \frac{\psi^L(\xi^{k_l'}, \lambda^{k_l'}, x)}{|\lambda^{k_l'}|} \d \eta(x) 
	= \int_{X} \frac{\psi^L(\xi^{k_l'}, \lambda^{k_l'}, x)}{|\lambda^{k_l'}|} \d \eta(x) 
	\le \frac{L(\eta,\lambda^{k_l'}) - L(\xi^{k_l'}, \lambda^{k_l'})}{|\lambda^{k_l'}|} \notag\\
	&= \frac{\Psi_0(\eta)}{|\lambda^{k_l'}|} + \sum_{i \in I} \frac{\lambda_i^{k_l'}}{|\lambda^{k_l'}|} \Psi_i(\eta) - \frac{\Psi_0(\xi^{k_l'})}{|\lambda^{k_l'}|}
	\qquad (\eta \in \Xifin(X^0)).
\end{align}
Since $\Psi_0(\xi^{k_l'}) \ge \liminf_{n\to\infty} \Psi_0(\xi^{k_n'}) - 1 \ge \Psi_0(\xi^*) - 1$ for $l$ large enough by the lower semicontinuity of $\Psi_0$, it further follows from~\eqref{eq:subsequence-lm-step2.1.2} with the help of~\eqref{eq:subsequence-lm-step2-unboundedness-assumption} that
\begin{align} \label{eq:subsequence-lm-step2.1.3}
	0 &\le \limsup_{l\to\infty} \frac{\Psi_0(\eta)}{|\lambda^{k_l'}|} + \sum_{i \in I}  \limsup_{l\to\infty} \frac{\lambda_i^{k_l'}}{|\lambda^{k_l'}|} \Psi_i(\eta) - \liminf_{l\to\infty} \frac{\Psi_0(\xi^{k_l'})}{|\lambda^{k_l'}|} \notag\\
	&\le \sum_{i \in I} \tilde{\lambda}_i^* \Psi_i(\eta)
	\qquad (\eta \in \Xifin(X^0)).
\end{align}
Choose now $\eta := \eta^0$ as in Condition~\ref{cond:reg-conditions-on-X^0}~(i). In view of~(\ref{eq:strictly-feasible-design-on-X^0}.a) and~(\ref{eq:strictly-feasible-design-on-X^0}.c), it then follows from~\eqref{eq:subsequence-lm-step2.1.3} that
\begin{align} \label{eq:subsequence-lm-step2.1.4}
	0 \le \sum_{i \in I} \tilde{\lambda}_i^* \Psi_i(\eta^0) = \sum_{i \in \Iineq} \tilde{\lambda}_i^* \Psi_i(\eta^0).
\end{align}
Since $\tilde{\lambda}_i^* \ge 0$ and $\Psi_i(\eta^0) < 0$ for $i \in \Iineq$, we conclude from~\eqref{eq:subsequence-lm-step2.1.4} that the assertion~\eqref{eq:subsequence-lm-step2.1} of the first substep 
must hold true, as desired.
As a second substep, we show that 
\begin{align} \label{eq:subsequence-lm-step2.2}
	\sum_{i \in \Ieq} \tilde{\lambda}_i^* \Psi_i(\eta) \ge 0 
	\qquad (\eta \in \Xi(X^0)).
\end{align}
In order to prove this, we first observe from~\eqref{eq:subsequence-lm-step2.1.1} with the help of~\eqref{eq:psi_i-for-i-in-Inc} that
\begin{align} \label{eq:subsequence-lm-step2.2.1}
	0 &\le \limsup_{l\to\infty} \int_{X^0} \frac{\psi^L(\xi^{k_l'}, \lambda^{k_l'}, x)}{|\lambda^{k_l'}|} \d \eta(x)
	\le \limsup_{t\to\infty} \sum_{x \in X^0} \frac{\psi_0(\xi^{k_l'}, x)}{|\lambda^{k_l'}|} \eta(\{x\}) \notag\\
	&\quad + \sum_{i \in \Ic} \limsup_{t\to\infty} \sum_{x \in X^0} \frac{\lambda_i^{k_l'}}{|\lambda^{k_l'}|} \psi_i(\xi^{k_l'}, x) \eta(\{x\}) \notag\\
	&\quad + \sum_{i \in \Inc} \limsup_{t\to\infty} \sum_{x \in X^0} \frac{\lambda_i^{k_l'}}{|\lambda^{k_l'}|} \big( g_i(x) - \Psi_i(\xi^{k_l'}) \big) \eta(\{x\})
	\qquad (\eta \in \Xi(X^0)).
\end{align} 
Since $\Psi_i(\xi^{k_l'}) \ge \liminf_{n\to\infty} \Psi_i(\xi^{k_n'}) - 1 \ge \Psi_i(\xi^*) - 1$ for $l$ large enough by the lower semicontinuity of $\Psi_i$ and since $\tilde{\lambda}_i^k \ge 0$ for all $i \in \Inc$ by~\eqref{eq:Ieq-subset-of-Ic-and-Inc-subset-of-Iineq}, it further follows from~\eqref{eq:subsequence-lm-step2.2.1} with the help of~\eqref{eq:Ieq-subset-of-Ic-and-Inc-subset-of-Iineq} and~\eqref{eq:subsequence-lm-step2-unboundedness-assumption} and~\eqref{eq:subsequence-lm-step2.1} that
\begin{align} \label{eq:subsequence-lm-step2.2.2}
	0 &\le \sum_{i \in \Ic} \tilde{\lambda}_i^* \int_X \psi_i(\xi^*,x) \d\eta(x) + \sum_{i \in \Inc} \tilde{\lambda}_i^* \int_X g_i(x) - \Psi_i(\xi^*) + 1 \d\eta(x) \notag\\
	&= \sum_{i \in \Ic} \tilde{\lambda}_i^* \int_X \psi_i(\xi^*,x) \d\eta(x)
	= \sum_{i \in \Ieq} \tilde{\lambda}_i^* \int_X \psi_i(\xi^*,x) \d\eta(x)
	\qquad (\eta \in \Xi(X^0)).
\end{align}
Since $\xi^* \in \Xifin(X) \cap \Xifeas(X)$ by the first step, if finally follows by the affinity of $\Psi_i$ for $i \in \Ieq$ that
\begin{align} \label{eq:subsequence-lm-step2.2.3}
	\int_X \psi_i(\xi^*,x) \d\eta(x) 
	= \Psi_i(\eta) - \Psi_i(\xi^*)
	= \Psi_i(\eta)
	\qquad (i \in \Ieq \text{ and } \eta \in \Xi(X^0))
\end{align}
Combining now~\eqref{eq:subsequence-lm-step2.2.2} and~\eqref{eq:subsequence-lm-step2.2.3}, we immediately obtain the assertion~\eqref{eq:subsequence-lm-step2.2} of the second substep, as desired.
As a third and last substep, we finally deduce the desired contradiction. In order to do so, we first notice that $\tilde{\lambda}^* \ne 0$ by virtue of~\eqref{eq:subsequence-lm-step2-unboundedness-assumption}. In view of the first substep~\eqref{eq:subsequence-lm-step2.1}, we conclude that
\begin{align} \label{eq:subsequence-lm-step2.2.4}
	(\tilde{\lambda}_i^*)_{i \in \Ieq} \ne 0. 
\end{align} 
So, by the second substep~\eqref{eq:subsequence-lm-step2.2} and~\eqref{eq:subsequence-lm-step2.2.4}, the set $\Psieq(\Xi(X^0))$ is one-sided w.r.t.~$0$. Contradiction to our  Condition~\ref{cond:reg-conditions-on-X^0}~(ii)! With this contradiction, the proof of the boundedness of $(\lambda^{k_l'})$ and thus also of the second step is complete. 
\smallskip

As a third step, we show that for every subsequence $(k_l)$ as in the first two steps, the convergence~\eqref{eq:subsequence-lm-4} holds true.
So, let $(k_l)$ and $x^* \in X$, $\lambda^* \in \Lambda$, $\psi_i^* \in \R$ be such that the convergences~(\ref{eq:subsequence-lm-2}.a), (\ref{eq:subsequence-lm-2}.b), (\ref{eq:subsequence-lm-2}.c) and \eqref{eq:subsequence-lm-3} hold true. Then the asserted convergence~\eqref{eq:subsequence-lm-4} immediately follows and thus the proof of the third step is complete. 
\end{proof}

\begin{lm} \label{lm:subsequence-lm-supplement-for-special-algo}
Suppose that Conditions~\ref{cond:objective-and-constraint-fcts} and \ref{cond:sensi-fcts} are satisfied. Suppose further that $(\xi^k)$, $(\lambda^k)$, $(x^k)$ are generated by Algorithm~\ref{algo:COED-special} with tolerances $\ul{\delta}_k \in (0,\infty)$ and $\eps = 0 $ and an initial finite subset $X^0 \ne \emptyset$ such that
\begin{align} \label{eq:limsup-delta_k-is-zero}
	\limsup_{k\to\infty} \ul{\delta}_k = 0
\end{align} and such that Condition~\ref{cond:reg-conditions-on-X^0} is satisfied.
If $(\xi^k)$ is not terminating and if $\xi^*$ is any accumulation point of $(\xi^k)$ and $\lambda^*$ is as in the previous lemma, then 
\begin{align} \label{eq:subsequence-lm-supplement-for-special-algo}
	\lambda_i^* \Psi_i(\xi^*) = 0 \qquad (i \in \Iineq) 
	\qquad \text{and} \qquad
	\inf_{x \in X} \psi^L(\xi^*,\lambda^*,x) \ge 0.
\end{align}
\end{lm}

\begin{proof}
Suppose that $(\xi^k)$ is non-terminating and that $\xi^*$ is any accumulation point of $(\xi^k)$. Also, let $(k_l)$ be a subsequence and let $x^* \in X$, $\lambda^* \in \Lambda$, $\psi_i^* \in \R$ be points as in the previous lemma. In particular, this means that the relation~\eqref{eq:subsequence-lm-1} and the  convergences~\eqref{eq:subsequence-lm-2}-\eqref{eq:subsequence-lm-4} hold true. Since the tolerances $\ul{\delta}_k$ are all positive, the assumption~\eqref{eq:limsup-delta_k-is-zero} implies that the tolerances even converge to $0$, that is,
\begin{align} \label{eq:lim-delta_k-is-zero}
	\lim_{k\to\infty} \ul{\delta}_k = 0. 
\end{align}  
With these preliminary observations, we can now prove the assertion of the lemma. We do this in two steps.
\smallskip

As a first step, we establish the complementarity relation~(\ref{eq:subsequence-lm-supplement-for-special-algo}.a). 
In order to do so, we first notice that
\begin{align} \label{eq:supplement-subsequence-lm-step1.1}
	\psi^L(\xi^k, \lambda^k, x^k) - \ul{\delta}_k \le \inf_{x \in X} \psi^L(\xi^k, \lambda^k, x) 
	\qquad (k \in \N_0)
\end{align}
because $x^k$ is a $\ul{\delta}_k$-approximate solution of~\eqref{eq:SM(xi^k,lambda^k)-special} for every $k \in \N_0$ by the special algorithm's definition. It follows from~\eqref{eq:supplement-subsequence-lm-step1.1} with the help of Lemma~\ref{lm:bound-on-differences-of-Psiis-in-terms-of-sensi-fcts} and Lemma~\ref{lm:xik-optimal-and-psiL-ge-0-on-Xk}~(i) that
\begin{align}
\label{eq:supplement-subsequence-lm-step1.2}
	\psi^L(\xi^{k_l}, \lambda^{k_l}, x^{k_l}) - \ul{\delta}_{k_l} 
	&\le \int_X \psi^L(\xi^{k_l}, \lambda^{k_l}, x) \d\eta(x) 
	\le L(\eta,\lambda^{k_l}) - L(\xi^{k_l}, \lambda^{k_l}) \notag\\
	&= \Psi_0(\eta) + \sum_{i \in I} \lambda_i^{k_l} \Psi_i(\eta) - \Psi_0(\xi^{k_l})
	\qquad (l \in \N_0 \text{ and } \eta \in \Xifin(X)).
\end{align}
In view of~\eqref{eq:psiL(xi^l,lambda^l,x^k)-ge-0} and \eqref{eq:subsequence-lm-4}, it further follows from~\eqref{eq:supplement-subsequence-lm-step1.2} by~\eqref{eq:subsequence-lm-2} 
and the lower semicontinuity of $\Psi_0$ that
\begin{align} \label{eq:supplement-subsequence-lm-step1.3}
	0  &\le \lim_{l\to\infty} \psi^L(\xi^{k_{l+1}}, \lambda^{k_{l+1}}, x^{k_l}) - \lim_{l\to\infty} \ul{\delta}_{k_l}
	= \lim_{l\to\infty} \big( \psi^L(\xi^{k_l}, \lambda^{k_l}, x^{k_l}) - \ul{\delta}_{k_l} \big) \notag\\
	&\le \Psi_0(\eta) + \sum_{i \in I} \lambda_i^* \Psi_i(\eta) - \Psi_0(\xi^*)
	\qquad (\eta \in \Xifin(X)).
\end{align}
Choosing now $\eta := \xi^*$, we can conclude from~\eqref{eq:supplement-subsequence-lm-step1.3} with the help of~\eqref{eq:subsequence-lm-1} that
\begin{align} \label{eq:supplement-subsequence-lm-step1.4}
	0 \le \sum_{i \in I} \lambda_i^* \Psi_i(\xi^*) = \sum_{i \in \Iineq} \lambda_i^* \Psi_i(\xi^*).
\end{align}
Since $\lambda_i^* \ge 0$ and $\Psi_i(\xi^*) \le 0$ for all $i \in \Iineq$ by virtue of~\eqref{eq:subsequence-lm-1}, we finally see from~\eqref{eq:supplement-subsequence-lm-step1.4} that the asserted complementarity relation~(\ref{eq:subsequence-lm-supplement-for-special-algo}.a) must be satisfied.
\smallskip

As a second step, we establish the estimate (\ref{eq:subsequence-lm-supplement-for-special-algo}.b). %
In order to do so, we decompose the constraint index set $I$ into $\Ic$ and $\Inc$ as in~\eqref{eq:Ic-and-Inc-definition}. Since $\lambda_i^k \ge 0$ for $i \in \Inc$ by virtue of~(\ref{eq:Ieq-subset-of-Ic-and-Inc-subset-of-Iineq}.b) and since $\Psi_i(\xi^*) \le \liminf_{l\to\infty} \Psi_i(\xi^{k_l}) \le 0$ 
for $i \in I$ by the lower semicontinuity of $\Psi_i$ and by~(\ref{eq:xi^k-in-Xifin(Xk)-cap-Xifeas(Xk)-and-complementarity-relation-for-xi^k}.a), it follows with the help of~\eqref{eq:subsequence-lm-2} and \eqref{eq:psi_i-for-i-in-Inc} that
\begin{align} \label{eq:supplement-subsequence-lm-step2.1}
	\psi^L(\xi^*, \lambda^*, x) 
	&= \psi_0(\xi^*, \lambda^*, x) + \sum_{i \in \Ic} \lambda_i^* \psi_i(\xi^*, x) + \sum_{i \in \Inc} \lambda_i^* (g_i(x) - \Psi_i(\xi^*)) \notag\\
	&\ge \lim_{l\to\infty} \psi_0(\xi^{k_l}, x) + \sum_{i \in \Ic} \lim_{l\to\infty} \lambda_i^{k_l} \psi_i(\xi^{k_l}, x) + \sum_{i \in \Inc} \limsup_{l\to\infty} \lambda_i^{k_l} \big( g_i(x) - \Psi_i(\xi^{k_l}) \big) \notag\\
	&\ge \limsup_{l\to\infty} \psi^L(\xi^{k_l}, \lambda^{k_l}, x)
	\qquad (x \in X).
\end{align} 
In view of~\eqref{eq:psiL(xi^l,lambda^l,x^k)-ge-0} and~\eqref{eq:subsequence-lm-4}, it further follows from~\eqref{eq:supplement-subsequence-lm-step1.1} and \eqref{eq:supplement-subsequence-lm-step2.1} that
\begin{align}
	0 &\le \lim_{l\to\infty} \psi^L(\xi^{k_{l+1}}, \lambda^{k_{l+1}}, x^{k_l}) - \lim_{l\to\infty} \ul{\delta}_{k_l}
	= \lim_{l\to\infty} \big( \psi^L(\xi^{k_l}, \lambda^{k_l}, x^{k_l}) - \ul{\delta}_{k_l} \big) \notag\\
	&\le \psi^L(\xi^*, \lambda^*, x)
	\qquad (x \in X).
\end{align}
Consequently, the asserted estimate (\ref{eq:subsequence-lm-supplement-for-special-algo}.b) is satisfied, as desired. 
\end{proof}

With the help of Corollary~\ref{cor:minimum-of-sensi-fct-determines-optimality-gap} and Lemmas~\ref{lm:properties-of-constr-fcts-depending-on-sensi-fct-continuity}-\ref{lm:subsequence-lm-supplement-for-special-algo}, we can now establish our convergence result for the special algorithm.

\begin{thm} \label{thm:convergence-thm-for-special-algo-with-eps=0}
Suppose that Conditions~\ref{cond:objective-and-constraint-fcts} and \ref{cond:sensi-fcts} are satisfied. Suppose further that $(\xi^k)$, $(\lambda^k)$, $(x^k)$ are generated by Algorithm~\ref{algo:COED-special} with tolerances $\ul{\delta}_k \in (0,\infty)$ and $\eps = 0$ and an initial finite subset $X^0 \ne \emptyset$ such that
\begin{align}
	\limsup_{k\to\infty} \ul{\delta}_k = 0
\end{align}
and such that Condition~\ref{cond:reg-conditions-on-X^0} is satisfied.
\begin{itemize}
\item[(i)] If $(\xi^k)$ is terminating, then it terminates at an optimal solution of~\eqref{eq:COED(X)}.
\item[(ii)] If $(\xi^k)$ is not terminating, then it has an accumulation point and every accumulation point is an optimal solution of~\eqref{eq:COED(X)}. Additionally, the criterion values $\Psi_0(\xi^k)$ converge monotonically decreasingly to the optimal value of~\eqref{eq:COED(X)}, in short:
\begin{align} \label{eq:convergence-of-criterion-values}
	\Psi_0(\xi^k) \searrow \min_{\xi \in \Xifeas(X)} \Psi_0(\xi) \qquad (k \to \infty).
\end{align}
\end{itemize}
\end{thm}

\begin{proof}
(i) Suppose that $(\xi^k)$ is terminating and let $k^* := \max K$ be the termination index. We then know that the special algorithm's termination condition is satisfied for $k = k^*$ and therefore
\begin{align} \label{eq:convergence-thm-for-algo-with-eps=0-(i)-1}
	\inf_{x\in X} \psi^L(\xi^{k^*}, \lambda^{k^*}, x)  
	\ge \psi^L(\xi^{k^*}, \lambda^{k^*}, x^{k^*}) - \ul{\delta}_{k^*} 
	\ge 0.
\end{align}
We further know from Lemma~\ref{lm:xik-optimal-and-psiL-ge-0-on-Xk}~(i) that
\begin{align} \label{eq:convergence-thm-for-algo-with-eps=0-(i)-2}
	\xi^{k^*} \in 
	\Xifin(X) \cap \Xifeas(X)
	\qquad \text{and} \qquad
	\lambda_i^{k^*} \Psi_i(\xi^{k^*}) = 0 \qquad (i \in \Iineq). 
\end{align} Corollary~\ref{cor:minimum-of-sensi-fct-determines-optimality-gap} thus implies that $\xi^{k^*}$ is an optimal solution of~\eqref{eq:COED(X)}, as desired.
\smallskip

(ii) Suppose that $(\xi^k)$ is not terminating, that is, $K = \N_0$, and let $\xi^*$ be any accumulation point of $(\xi^k)$. Such an accumulation point exists by the compactness of $\Xi(X)$, of course. Also, let $\lambda^* \in \Lambda$ be as in Lemma~\ref{lm:subsequence-lm}. We then know from Lemma~\ref{lm:subsequence-lm} that
\begin{align} \label{eq:convergence-thm-for-algo-with-eps=0-(ii)-1}
	\xi^* \in \Xifin(X) \cap \Xifeas(X)
\end{align}
and, moreover, we know from Lemma~\ref{lm:subsequence-lm-supplement-for-special-algo} that
\begin{align} \label{eq:convergence-thm-for-algo-with-eps=0-(ii)-2}
	\lambda_i^* \Psi_i(\xi^*) = 0 \qquad (i \in \Iineq)
	\qquad \text{and} \qquad 
	\inf_{x \in X} \psi^L(\xi^*, \lambda^*, x) \ge 0. 
\end{align}
Corollary~\ref{cor:minimum-of-sensi-fct-determines-optimality-gap} thus implies that $\xi^*$ is an optimal solution of~\eqref{eq:COED(X)}, as desired.
It remains to establish the monotonic convergence~\eqref{eq:convergence-of-criterion-values} of the criterion values to the optimal value of~\eqref{eq:COED(X)}. Choose a subsequence $(\xi^{k_l})$ with $\xi^{k_l} \longrightarrow \xi^*$ as $l \to \infty$. Since $\Psi_0|_{\dom\Psi_0}$ is continuous (Lemma~\ref{lm:properties-of-constr-fcts-depending-on-sensi-fct-continuity}~(i)), it follows by~(\ref{eq:xi^k-in-Xifin(Xk)-cap-Xifeas(Xk)-and-complementarity-relation-for-xi^k}.a) and~\eqref{eq:convergence-thm-for-algo-with-eps=0-(ii)-1} that 
\begin{align}
	\label{eq:convergence-thm-for-algo-with-eps=0-(ii)-3}
	\lim_{l\to\infty} \Psi_0(\xi^{k_l}) = \Psi_0(\xi^*).
\end{align}
Since, moreover, $(\Psi_0(\xi^k))$ is monotonically decreasing (Lemma~\ref{lm:xik-optimal-and-psiL-ge-0-on-Xk}~(ii)), it further follows that
\begin{align}
	\label{eq:convergence-thm-for-algo-with-eps=0-(ii)-4}
	\Psi_0(\xi^k) \searrow \lim_{l\to\infty} \Psi_0(\xi^{k_l}) \qquad (k \to \infty).
\end{align}
Combining now~\eqref{eq:convergence-thm-for-algo-with-eps=0-(ii)-3} and~\eqref{eq:convergence-thm-for-algo-with-eps=0-(ii)-4} with the already proven fact that $\xi^*$ is an optimal solution of~\eqref{eq:COED(X)}, we immediately obtain the remaining   assertion~\eqref{eq:convergence-of-criterion-values}.
\end{proof}

With the help of Corollary~\ref{cor:minimum-of-sensi-fct-determines-optimality-gap} and Lemmas~\ref{lm:xik-optimal-and-psiL-ge-0-on-Xk}-\ref{lm:subsequence-lm}, we can also establish our termination result for the general algorithm.

\begin{thm} \label{thm:termination-thm-for-general-algo-with-eps>0}
Suppose that Conditions~\ref{cond:objective-and-constraint-fcts} and \ref{cond:sensi-fcts} are satisfied. Suppose further that $(\xi^k)$, $(\lambda^k)$, $(x^k)$ are generated by Algorithm~\ref{algo:COED-general} with tolerances $\ul{\delta}_k \in (0,\infty)$ and $\eps > 0$ and an initial finite subset $X^0 \ne \emptyset$ such that
\begin{align}
	\label{eq:condition-on-tolerances-for-termination}
	\limsup_{k\to\infty} \ul{\delta}_k < \eps 
\end{align}
and such that Condition~\ref{cond:reg-conditions-on-X^0} is satisfied. Then $(\xi^k)$ terminates at an $\eps$-optimal solution of~\eqref{eq:COED(X)}.
\end{thm}

\begin{proof}
As a first step, we show that $(\xi^k)$ is terminating. Assume that it is not. It then follows that $K = \N_0$ and that the general algorithm's termination condition (with $\eps > 0$) is not satisfied for any $k \in \N_0$. In view of the algorithm's definition, this means for every given $k \in \N_0$ that $\psi^L(\xi^k, \lambda^k, x^k) < -\eps$ in case the search routine finds a point in $V^k$ or that $\psi^L(\xi^k, \lambda^k, x^k) - \ul{\delta}_k < -\eps$ in case the search routine finds no point in $V^k$. So, in any case, 
\begin{align} \label{eq:termination-of-algo-for-eps>0-step1-1}
	\psi^L(\xi^k,\lambda^k,x^k) < -\eps + \ul{\delta}_k \qquad (k\in\N_0).
\end{align}
Choose a subsequence $(k_l)$ as in Lemma~\ref{lm:subsequence-lm}. It then follows by Lemma~\ref{lm:xik-optimal-and-psiL-ge-0-on-Xk}~(iii) that
\begin{align} \label{eq:termination-of-algo-for-eps>0-step1-2}
	\psi^L(\xi^{k_{l+1}},x^{k_l},\lambda^{k_{l+1}}) \ge 0 \qquad (l\in\N_0).
\end{align}
Combining now~\eqref{eq:termination-of-algo-for-eps>0-step1-1} and~\eqref{eq:termination-of-algo-for-eps>0-step1-2} with~\eqref{eq:subsequence-lm-4} from Lemma~\ref{lm:subsequence-lm}, we obtain 
\begin{align}
	0 > -\eps + \limsup_{k\to\infty} \ul{\delta}_k 
	\ge \lim_{l\to\infty} \psi^L(\xi^{k_l},x^{k_l},\lambda^{k_l})
	= \lim_{l\to\infty} \psi^L(\xi^{k_{l+1}},x^{k_l},\lambda^{k_{l+1}}) \ge 0.
\end{align}
Contradiction! So, our assumption that $(\xi^k)$ is not terminating must be false and the proof of the first step is complete.
\smallskip

As a second step, we show that the terminal iterate $\xi^{k^*}$ is an $\eps$-optimal solution of~\eqref{eq:COED(X)}. Indeed, by our general algorithm's termination condition, we have
\begin{align} \label{eq:termination-of-algo-for-eps>0-step2-1}
	\inf_{x\in X} \psi^L(\xi^{k^*}, \lambda^{k^*}, x) 
	\ge \psi^L(\xi^{k^*}, \lambda^{k^*}, x^{k^*}) - \ul{\delta}_{k^*} 
	\ge -\eps.
\end{align}
We further know from Lemma~\ref{lm:xik-optimal-and-psiL-ge-0-on-Xk}~(i) that
\begin{align} \label{eq:termination-of-algo-for-eps>0-step2-2}
	\xi^{k^*} \in 
	\Xifin(X) \cap \Xifeas(X)
	\qquad \text{and} \qquad
	\lambda_i^{k^*} \Psi_i(\xi^{k^*}) = 0 \qquad (i \in \Iineq). 
\end{align} Corollary~\ref{cor:minimum-of-sensi-fct-determines-optimality-gap} thus implies that $\xi^{k^*}$ is an $\eps$-optimal solution of~\eqref{eq:COED(X)}, as desired.
\end{proof}

\section{Application examples}

In this section, we apply our algorithm to locally optimal experimental design problems with two kinds of models $f: X \times \Theta \to \R^{\dimy}$, namely an exponential growth model (which is defined explicitly by a closed-form expression) and a reaction kinetics model (which is defined implicitly by the solutions of ordinary differential equations). 
\smallskip

We begin by briefly discussing the implementation of the adaptive discretization algorithms (Algorithms~\ref{algo:COED-special} and~\ref{algo:COED-general}). 
All implementations were done in Python and the algorithms were run on a standard office laptop with an Intel i7 processor and 16 GB of RAM. 
In our implementation, we follow the common practice from the optimal experimental design literature~\cite{HaFiRi20, MoZu02, MoZu04, PrPa, YaBiTa13, Yu11} of choosing the design space $X$ to be a large but finite set of candidate experiments. 
As a consequence, we can -- and do -- solve the occurring sensitivity minimization problems ~\eqref{eq:SM(xi^k,lambda^k)-special} and~\eqref{eq:SM(xi^k,lambda^k)-general} exactly and by enumeration. 
In order to compute saddle points of the convex 
discretized design problems~\eqref{eq:COED(X^k)-short} in the form~\eqref{eq:weight-optimization-reformulation-of-the-discretized-design-problem}, in turn, we use the sequential quadratic programming solver from~\cite{Sc11, Sc14}. It returns a Karush-Kuhn-Tucker point of the finite-dimensional convex optimization problem~\eqref{eq:weight-optimization-reformulation-of-the-discretized-design-problem} and, as is well-known, for finite-dimensional convex optimization problems, Karush-Kuhn-Tucker points and saddle points are one and the same thing. 
\smallskip

As the design criterion for all problem instances, we use the probably most common criterion, namely the (logarithmic) D-criterion $\Psi_0\colon \Xi(X) \to \R \cup \{\infty\}$, which is defined by
\begin{align} \label{eq:D-criterion}
	\Psi_0(\xi) := \Phi_0(M(\xi)) := 
	\begin{cases}
		\log\left(\det\left( M(\xi)^{-1} \right) \right), \quad M(\xi) \text{ is non-singular} \\
		\infty,  \quad \text{else}
	\end{cases}
\end{align}
As usual in locally optimal design~\cite{FeLe, PrPa, Pu, Silvey}, $M(\xi) := \int_X m(x) \d x$ denotes the information matrix of the design $\xi$ and
\begin{align}
	\label{eq:one-point-information-matrix}
	m(x) := m_f(x,\ol{\theta}) := D_\theta f(x,\ol{\theta})^\top \varsigma(x)^{-1} D_\theta f(x,\ol{\theta}) 
	\in \R_\psd^{\dimtheta \times \dimtheta}
\end{align} 
denotes the one-point information matrix at $x$, where $\ol{\theta}$ is a suitable reference parameter value and $\varsigma(x) \in \R_\pd^{\dimtheta \times \dimtheta}$ is the covariance matrix of the measurement errors for the modeled output quantity at $x$. See the introductory sections of~\cite{BuSc24, ScSeBo}, for instance, for a very brief summary of locally optimal design explaining, in particular, why and in which sense minimizing~\eqref{eq:D-criterion} yields a maximally informative design.
\smallskip

We evaluate the performance of the algorithm by means of several performance indicators. Specifically, for each problem instance, we report the criterion value $\Psi_0(\xi^*)$ at the terminal design $\xi^*$, the number of iterations required and the computation time (runtime) of the algorithm (in seconds). Additionally, we list the number of support points of the design $\xi^*$ and compare it to the bounds obtained from Corollary~\ref{cor:opt-design-with-finite-support}. And finally, we report the value 
\begin{align}
	\eps^* := \Big| \min_{x \in X} \psi^L(\xi^*, \lambda^*, x) \Big|,
\end{align}
which is an upper bound on the approximation error $\Psi(\xi^*) - \min_{\xi \in \Xifeas(X)} \Psi_0(\xi)$ by Corollary~\ref{cor:minimum-of-sensi-fct-determines-optimality-gap}. 

\subsection{Applications with an exponential growth model}

In this section, we apply our algorithm to experimental design problems with the explicitly defined exponential growth model
\begin{align} \label{eq:exponential-model-function}
f: \X \times \Theta \to \R,
\qquad
f(x, \theta) := \theta_1  \exp(\theta_2 x).
\end{align} 
with input space $\X := [-1,1]$ and model parameter space $\Theta := \R^2$. Such exponential growth models are the building blocks of many models from chemical engineering, like the Antoine model for the saturation pressure of a substance, the non-random two-liquid model for the activity coefficients of substance mixtures, or Arrhenius models for reaction kinetics. As pointed out above, we take the design space $X$ to be a large but finite subset of $\X$, namely the uniform grid
\begin{align}
X := \{-1 + i/1000: i \in \{0, \dots, 2000\}\}
\end{align}
consisting of $2001$ grid points. As the reference parameter value $\ol{\theta}$, in turn, we take $\ol{\theta} = (1,3)^\top \in \Theta$ as in~\cite{VaSe21}, and the covariance $\varsigma(x)$ of the measurement errors we assume to be independent of the measurement point $x$, namely $\varsigma(x) := \varsigma := 1$. Clearly, $f$ has the Jacobian matrix 
	\begin{align}
		D_\theta f(x,\theta) = (\exp(\theta_2 x)),\ \theta_1 x \exp(\theta_2 x)).
	\end{align} 
and thus the one-point information matrix $m_f(x,\ol{\theta})$ from~\eqref{eq:one-point-information-matrix} is given by
	\begin{align}
		m(x) = m_f(x,\ol{\theta}) = 
		\begin{pmatrix}
			\exp(6x) & x\exp(6x) \\
			x\exp(6x) & x^2 \exp(6x)
		\end{pmatrix}.
	\end{align}	
Conditions~\ref{cond:objective-and-constraint-fcts} and~\ref{cond:sensi-fcts} are trivially verified using Lemma~\ref{lm:sufficient-conditions-for-standard-assumptions}. Condition~\ref{cond:slater-and-non-one-sidedness}, in turn, will be verified using Lemma~\ref{lm:sufficient-cond-for-saddle-points-existence-and-reg-conditions} with the designs
\begin{align}
	\eta_\alpha &:= \alpha \delta_{-1} + (1-\alpha) \delta_0 \in \Xi(X) \qquad (\alpha \in \Delta_1')
	\label{eq:eta_alpha}\\
	\eta_{\alpha,\beta} &:= \alpha \delta_{-1} + \beta \delta_1 + (1-\alpha-\beta) \delta_0 \in \Xi(X) \qquad ((\alpha,\beta) \in \Delta_2'),
	\label{eq:eta_alpha,beta}
\end{align}
where $\Delta_n' := \{w \in [0,1]^n: w_1 + \dotsb + w_n \le 1\}$ denotes the $n$-dimensional standard simplex. Specifically, we will use that
\begin{gather}
	\det(M(\eta_\alpha)) = \alpha (1-\alpha) \e^6,
	\label{eq:det(M(eta_alpha))}\\
	\det(M(\eta_{\alpha,\beta})) = 4 \alpha \beta + \alpha \gamma \e^{-6} + \beta \gamma \e^6 \qquad (\gamma := 1-\alpha-\beta).
	\label{eq:det(M(eta_alpha,beta))}
\end{gather}
It immediately follows from~\eqref{eq:det(M(eta_alpha))} and~\eqref{eq:det(M(eta_alpha,beta))} that $M(\eta_\alpha)$ is invertible if and only if $\alpha \in  \Delta_1' \setminus \{0,1\}$ and that $M(\eta_{\alpha,\beta})$ is invertible if and only if $(\alpha, \beta) \in \Delta_2' \setminus \{(0,0), (0,1), (1,0)\}$. It further follows from~\eqref{eq:det(M(eta_alpha))} and~\eqref{eq:det(M(eta_alpha,beta))} that
\begin{gather}
	\tr(M(\eta_\alpha)^{-1}) = \frac{2}{1-\alpha} + \frac{\e^6}{\alpha} \qquad (\alpha \in \Delta_1' \setminus \{0,1\})
	\label{eq:tr(M(eta_alpha)-inverse)}\\
	\tr(M(\eta_{\alpha,\beta})^{-1}) = \frac{2\alpha \e^{-6} + 2\beta \e^6 + \gamma}{4 \alpha \beta + \alpha \gamma \e^{-6} + \beta \gamma \e^6} \qquad ((\alpha, \beta) \in \Delta_2' \setminus \{(0,0), (0,1), (1,0)\}),
	\label{eq:tr(M(eta_alpha,beta)-inverse)}
\end{gather}
where we again used the shorthand notation $\gamma := 1-\alpha-\beta$.

\begin{ex}\label{ex:unconstrained-example-1}
In this example, we consider the unconstrained design problem
\begin{align} \label{eq:UOED(X)-1}
	\min_{\xi \in \Xi(X)} \Psi_0(\xi)
\end{align}	
with the design criterion~\eqref{eq:D-criterion} for reference. It is trivial to verify Conditions~\ref{cond:objective-and-constraint-fcts} and~\ref{cond:sensi-fcts} using Lemma~\ref{lm:sufficient-conditions-for-standard-assumptions}. 
Condition~\ref{cond:slater-and-non-one-sidedness}  is trivially verified using
\begin{align} \label{eq:X^0-unconstrained-example-1}
	X^0 := \{-1, 0\}
	\qquad \text{and} \qquad
	\eta^0 := 1/2 \cdot \delta_{-1} + 1/2 \cdot \delta_1.
\end{align}
and recalling~\eqref{eq:det(M(eta_alpha))}. 
As a consequence, 
Algorithm~\ref{algo:COED-general} with the initial discretization $X^0$ from~\eqref{eq:X^0-unconstrained-example-1} and with arbitrary tolerances $\ul{\delta}_k, \eps > 0$ satisfying~\eqref{eq:condition-on-tolerances-for-termination} terminates after finitely many iterations at an $\eps$-optimal design for~\eqref{eq:UOED(X)-1} (Theorem~\ref{thm:termination-thm-for-general-algo-with-eps>0}). Specifically, with $\eps := 10^{-3}$ and $\ul{\delta}_k := 10^{-4}$, our algorithm terminates after $2$ iterations with the $\eps^*$-optimal design $\xi^*$ from Table~\ref{tab:exponential-growth-examples} with $\eps^* := 6.0104 \cdot 10^{-5} < \eps$. It consists of $2$ design points, while 
\begin{align}
	\frac{d_\theta(d_\theta + 1)}{2} + 1 = 4
\end{align}
is the general upper bound on the support cardinality of $\eps$-optimal designs  for~\eqref{eq:UOED(X)-1} from Corollary~\ref{cor:opt-design-with-finite-support}. 
It is worth noting that this approximately optimal design is practically equal to the exactly optimal design of the relaxed unconstrained design problem
\begin{align} \label{eq:UOED(X)-1-relaxed}
	\min_{\xi \in \Xi(\X)} \Psi_0(\xi)
\end{align}	
which arises from~\eqref{eq:UOED(X)-1} by replacing the discrete set $X$ by the continuous superset $\X$. Indeed, the exactly optimal design of~\eqref{eq:UOED(X)-1-relaxed} is given by $1/2 \cdot \delta_{2/3} + 1/2 \cdot \delta_1$ by virtue of \cite{VaSe21}, for instance. 
Figure~\ref{fig:sensitivity-plots-at-termination} (a) shows the sensitivity function $X \ni x \mapsto \psi^L(\xi^*,\lambda^*,x)$ for the saddle point $(\xi^*, \lambda^*)$ from the terminal iteration of the algorithm. In particular, this figure shows that the terminal sensitivity function's minimum $\min_{x \in X} \psi^L(\xi^*,\lambda^*,x) = -\eps^*$ is basically $0$. 
$\blacktriangleleft$
\end{ex}

\begin{figure}[h]\captionsetup[subfigure]{font=normalsize}
	\centering
	\begin{subfigure}[t]{0.45\textwidth}
		\centering
		\includegraphics[width=\textwidth]{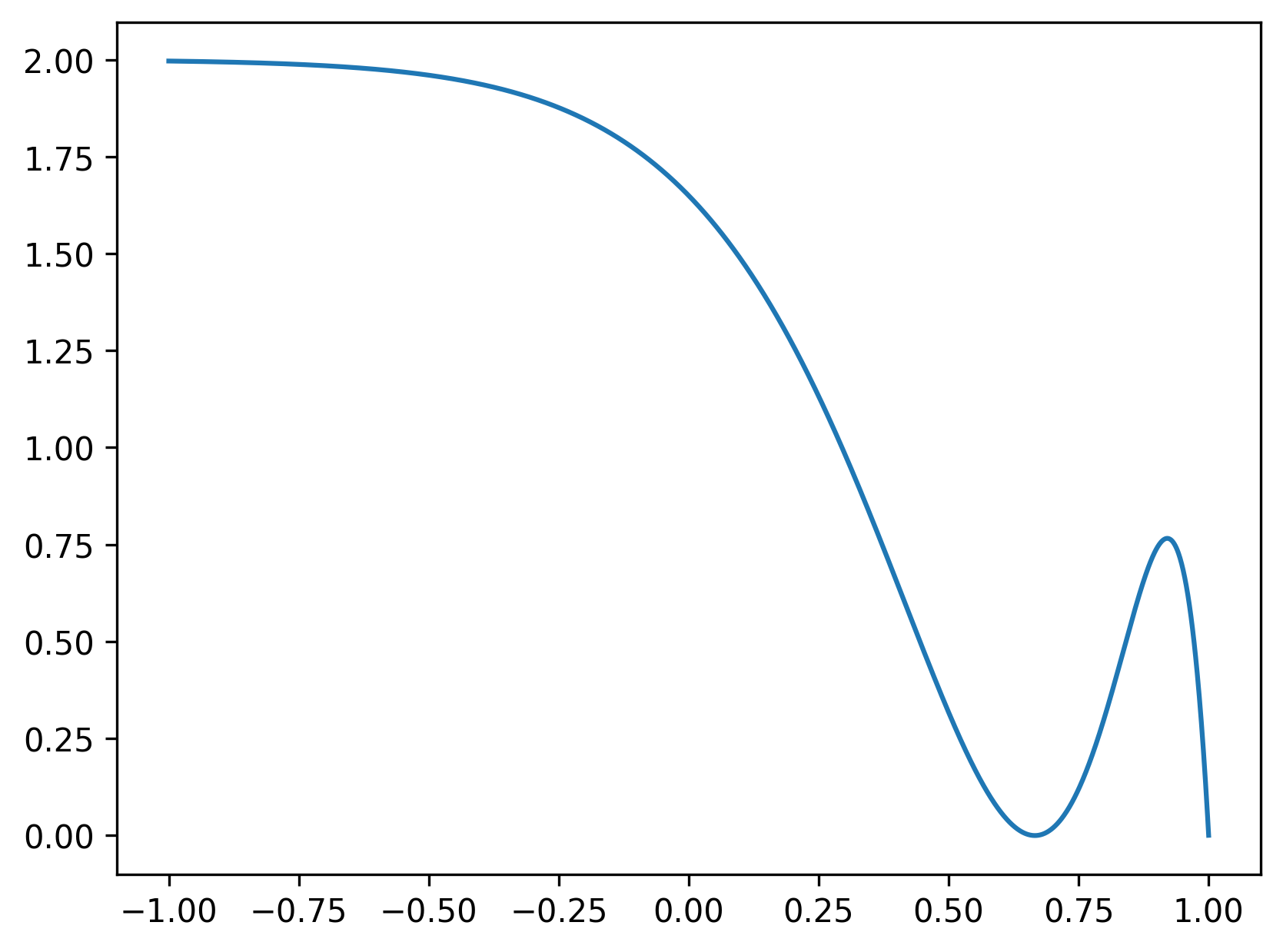}
		\subcaption{Example~\ref{ex:unconstrained-example-1}.}
		\label{fig:unconstrained-doe}
	\end{subfigure}
	\hfill
	\begin{subfigure}[t]{0.45\textwidth}
		\centering
		\includegraphics[width=\textwidth]{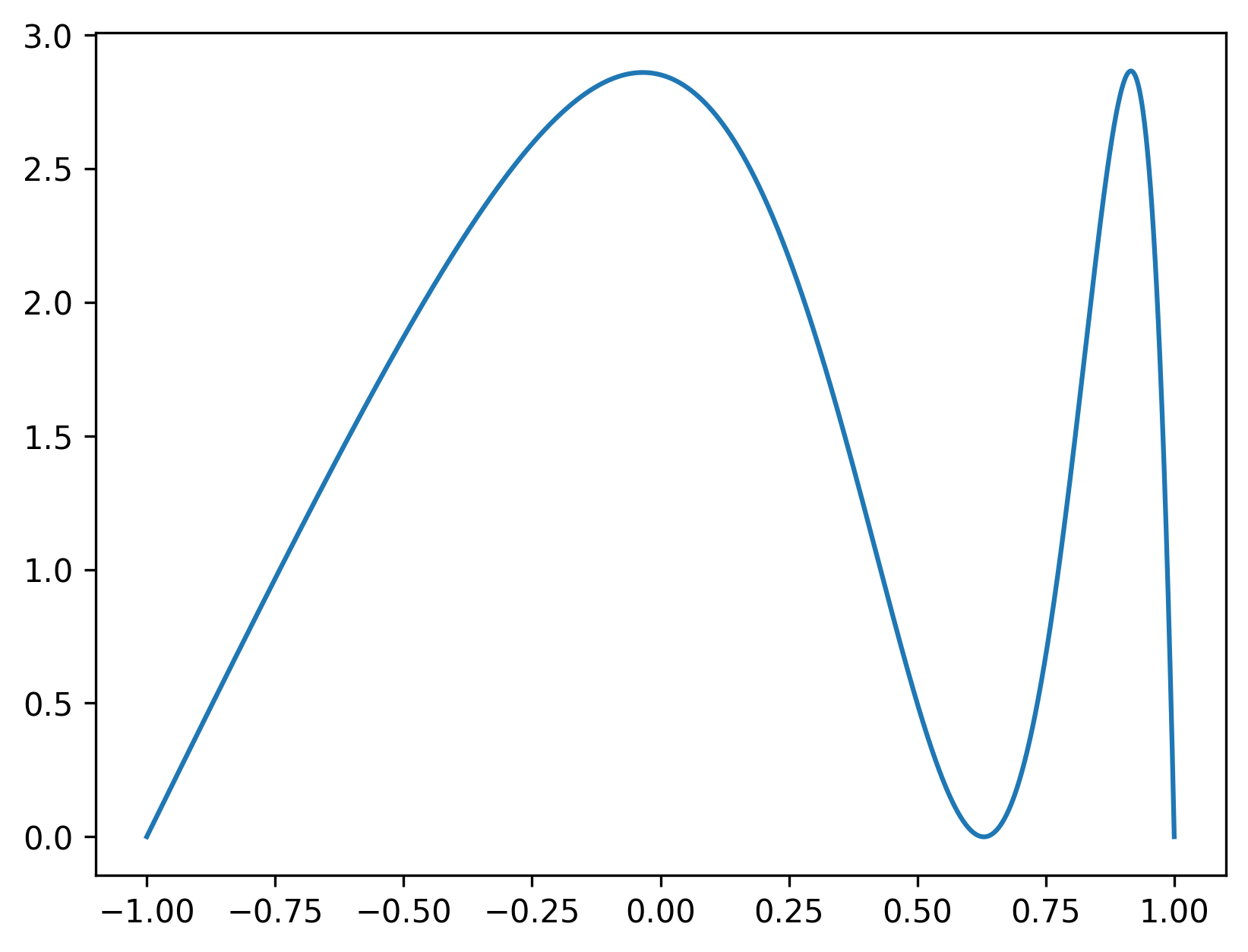}
		\subcaption{Example~\ref{ex:constrained-example-2}.}
		\label{fig:constrained-doe-ex2}
	\end{subfigure}
	\caption{Sensitivity function $x \mapsto \psi^L(\xi^*,\lambda^*,x)$ of the Lagrange function of (a) \eqref{eq:UOED(X)-1} or (b) \eqref{eq:COED(X)-example-2}, respectively, for the design $\xi^*$ and the Lagrange multiplier $\lambda^*$ from the algorithm's terminal iteration.}
	\label{fig:sensitivity-plots-at-termination}
\end{figure}

\begin{ex}\label{ex:constrained-example-1}
In this example, we consider the constrained design problem 
\begin{align}
	\label{eq:COED(X)-example-1}
	\min_{\xi \in \Xi(X)} \Psi_0(\xi) \quad 
	\text{s.t.} \quad \Psi_1(\xi) \le 0 \text{ and } \Psi_2(\xi) = 0
\end{align}
with the design criterion~\eqref{eq:D-criterion} and the affine inequality constraints and affine equality constraints defined by
\begin{align}
	\Psi_1(\xi) := \int_X g_1(x) \d\xi(x)
	\qquad \text{and} \qquad
	\Psi_2(\xi) := \int_X g_2(x) \d\xi(x),
\end{align}
where $g_1(x) := \indfct_{(0,\infty)}(x) - 0.1$ and $g_2(x) := x + 0.5$ for $x \in \X$.
It is trivial to verify Conditions~\ref{cond:objective-and-constraint-fcts} and~\ref{cond:sensi-fcts} using Lemma~\ref{lm:sufficient-conditions-for-standard-assumptions}. Condition~\ref{cond:slater-and-non-one-sidedness}, in turn, is easily verified as well. Indeed, let
\begin{align} \label{eq:X^0-constrained-example-1}
	X^0 := \{-1, 0\}
	\qquad \text{and} \qquad
	\eta^0 := 1/2 \cdot \delta_{-1} + 1/2 \cdot \delta_1.
\end{align}
It then immediately follows that $X^0 \subset X$ and that 
\begin{align} \label{eq:eta^0-strictly-feasible-constrained-example-1}
	\eta^0 \in \dom \Psi_0, \qquad
	\Psi_1(\eta^0) = -1/10 < 0, \qquad
	\Psi_2(\eta^0) = 0
\end{align}
by virtue of~\eqref{eq:det(M(eta_alpha))}. Additionally, the range
\begin{align} \label{eq:not-one-sided-constrained-example-1}
	g_2(X^0) = \{-1/2, 1/2\}
\end{align}
of $X^0$ under the equality-constraint-defining function $g_2$ is obviously not one-sided w.r.t.~$0$. In view of~\eqref{eq:eta^0-strictly-feasible-constrained-example-1} and~\eqref{eq:not-one-sided-constrained-example-1},  Condition~\ref{cond:slater-and-non-one-sidedness} is thus satisfied.
As a consequence, 
Algorithm~\ref{algo:COED-general} with the initial discretization $X^0$ from~\eqref{eq:X^0-constrained-example-1} and with arbitrary tolerances $\ul{\delta}_k, \eps > 0$ satisfying~\eqref{eq:condition-on-tolerances-for-termination} terminates after finitely many iterations at an $\eps$-optimal design for~\eqref{eq:COED(X)-example-1} (Theorem~\ref{thm:termination-thm-for-general-algo-with-eps>0}).  Specifically, with $\eps := 10^{-3}$ and $\ul{\delta}_k := 10^{-4}$, our algorithm terminates after $4$ iterations with the $\eps^*$-optimal design $\xi^*$ from Table~\ref{tab:exponential-growth-examples} with $\eps^* := 3.4068 \cdot 10^{-5} < \eps$. It consists of $4$ design points, while 
\begin{align}
	\frac{d_\theta(d_\theta + 1)}{2} + |I_{\mathrm{ineq},2}| + |\Ieq| + 1 = 6
\end{align}
is the general upper bound on the support cardinality of $\eps$-optimal designs  for~\eqref{eq:COED(X)-example-1} from Corollary~\ref{cor:opt-design-with-finite-support}. 
Figure~\ref{fig:sensitivity-plots-over-iterations} shows the sensitivity function $X \ni x \mapsto \psi^L(\xi^k,\lambda^k,x)$ for all four iterations of our algorithm along with the corresponding minimizers $x^k$ (red crosses). As can be seen, the sensitivity function (when extended from $X$ to the continuous set $\X$) is discontinuous at $x = 0$ provided that
the corresponding Lagrange multiplier $\lambda_1^k > 0$ is strictly positive. Indeed, this is because
\begin{align}
	\psi^L(\xi^k,\lambda^k,x) = \psi_0(\xi^k,x) + \lambda_1^k \cdot (g_1(x)-\Psi_1(\xi^k)) + \lambda_2^k \cdot (g_2(x)-\Psi_2(\xi^k))
	\qquad (x \in \X)
\end{align}
by Lemma~\ref{lm:sufficient-conditions-for-standard-assumptions} and because $g_1$ is discontinuous at $x = 0$, while $\psi_0$ and $g_2$ are continuous. 
$\blacktriangleleft$
\end{ex}
	
\begin{figure}[h]\captionsetup[subfigure]{font=normalsize}
	\centering
	\begin{subfigure}[t]{0.45\textwidth}
		\centering
		\includegraphics[width=\textwidth]{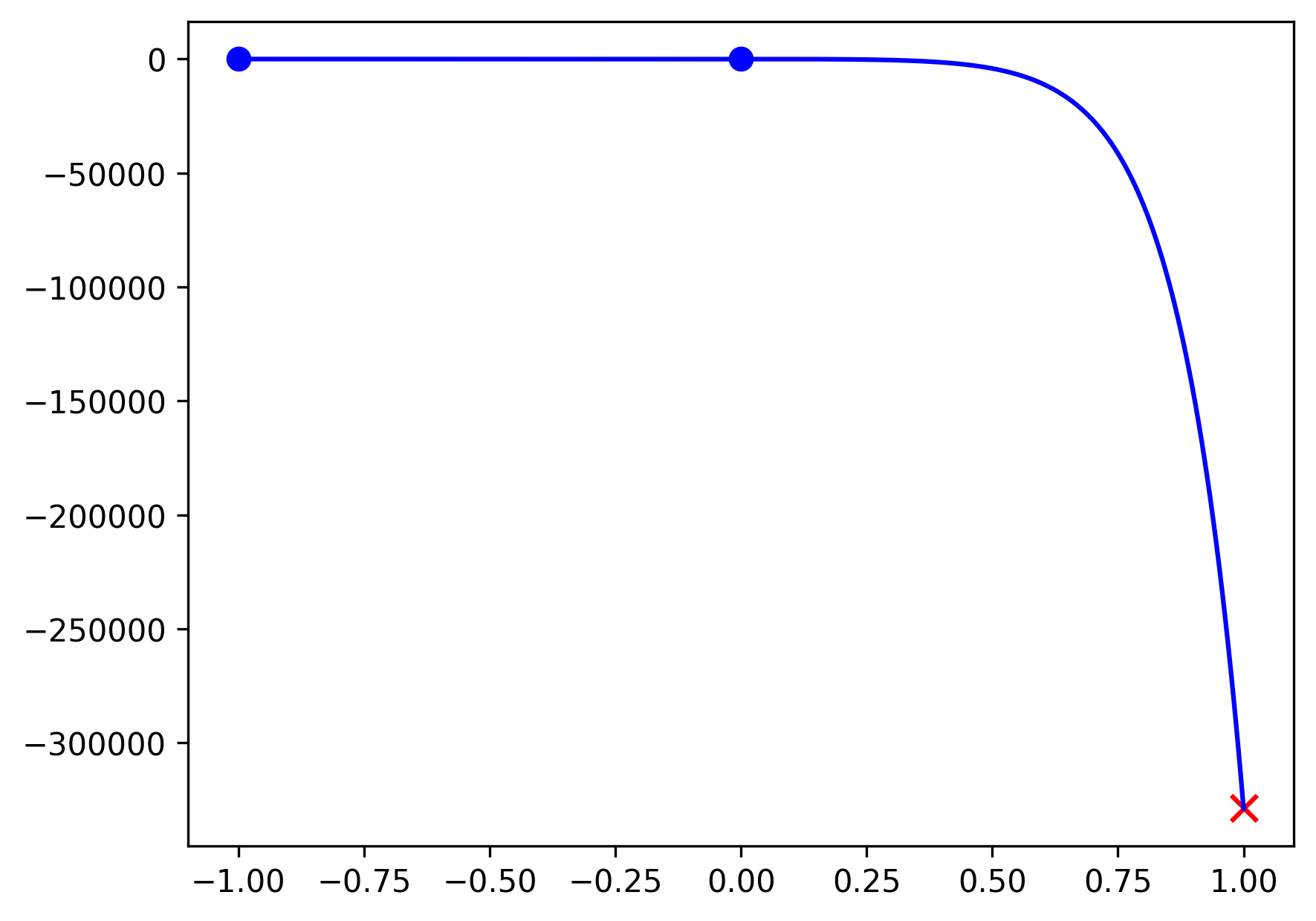}
		\subcaption{Iteration 1.}
		\label{fig:constrained-doe-iter-1}
	\end{subfigure}
	\hfill
	\begin{subfigure}[t]{0.45\textwidth}
		\centering
		\includegraphics[width=\textwidth]{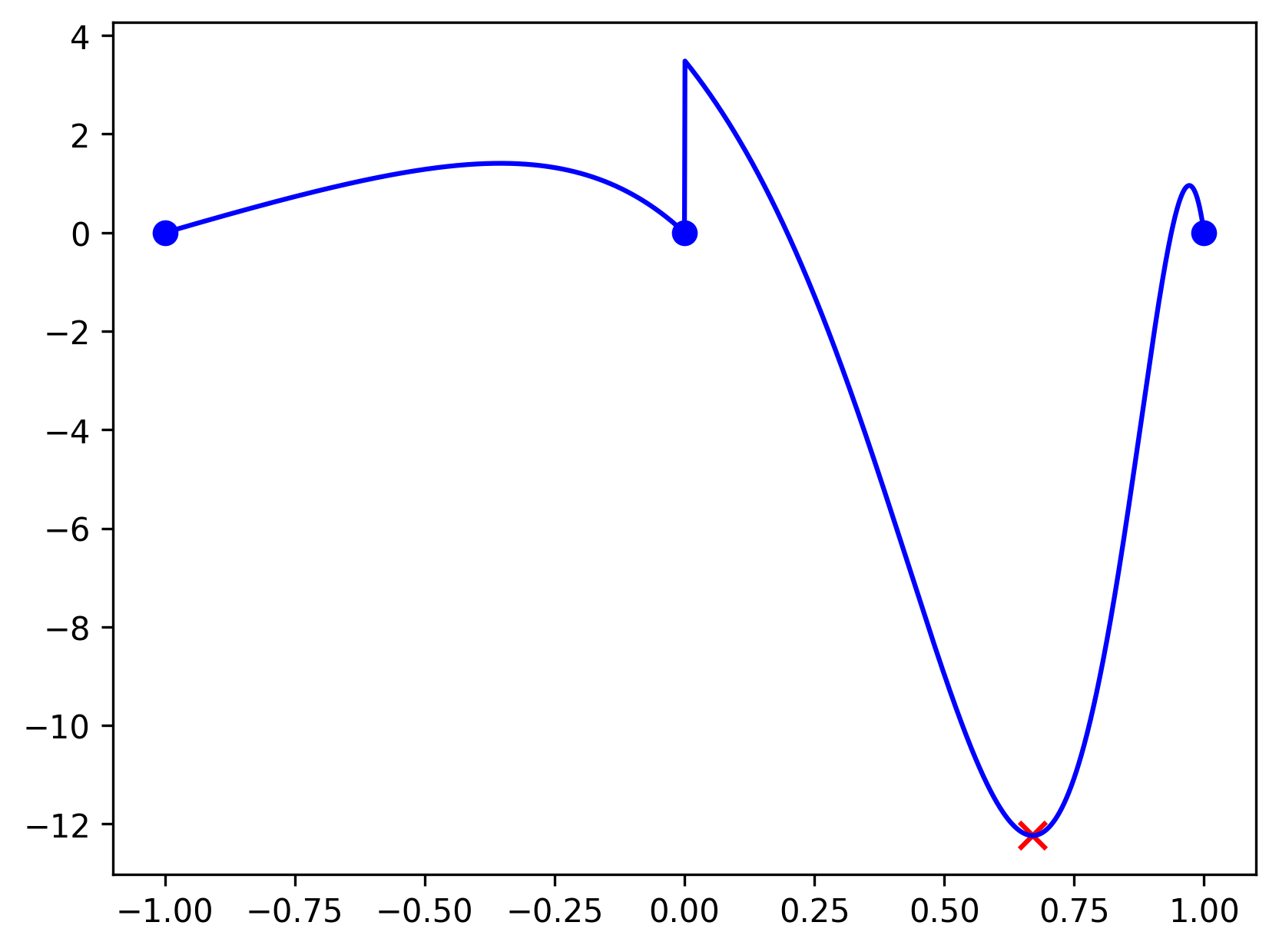}
		\subcaption{Iteration 2.}
		\label{fig:constrained-doe-iter-2}
	\end{subfigure}
	
	\begin{subfigure}[t]{0.45\textwidth}
		\centering
		\includegraphics[width=\textwidth]{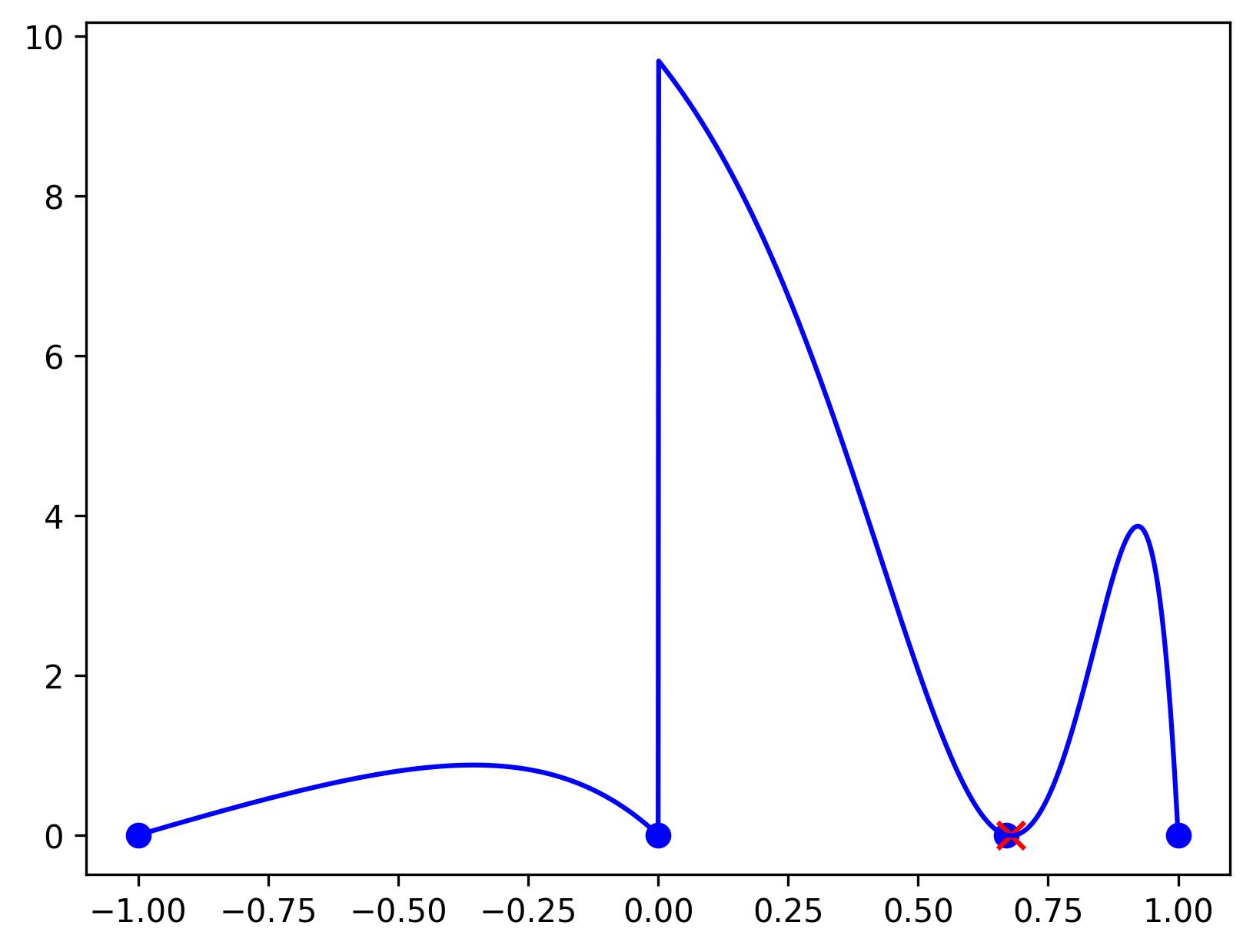}
		\subcaption{Iteration 3.}
		\label{fig:constrained-doe-iter-3}
	\end{subfigure}
	\hfill
	\begin{subfigure}[t]{0.45\textwidth}
		\centering
		\includegraphics[width=\textwidth]{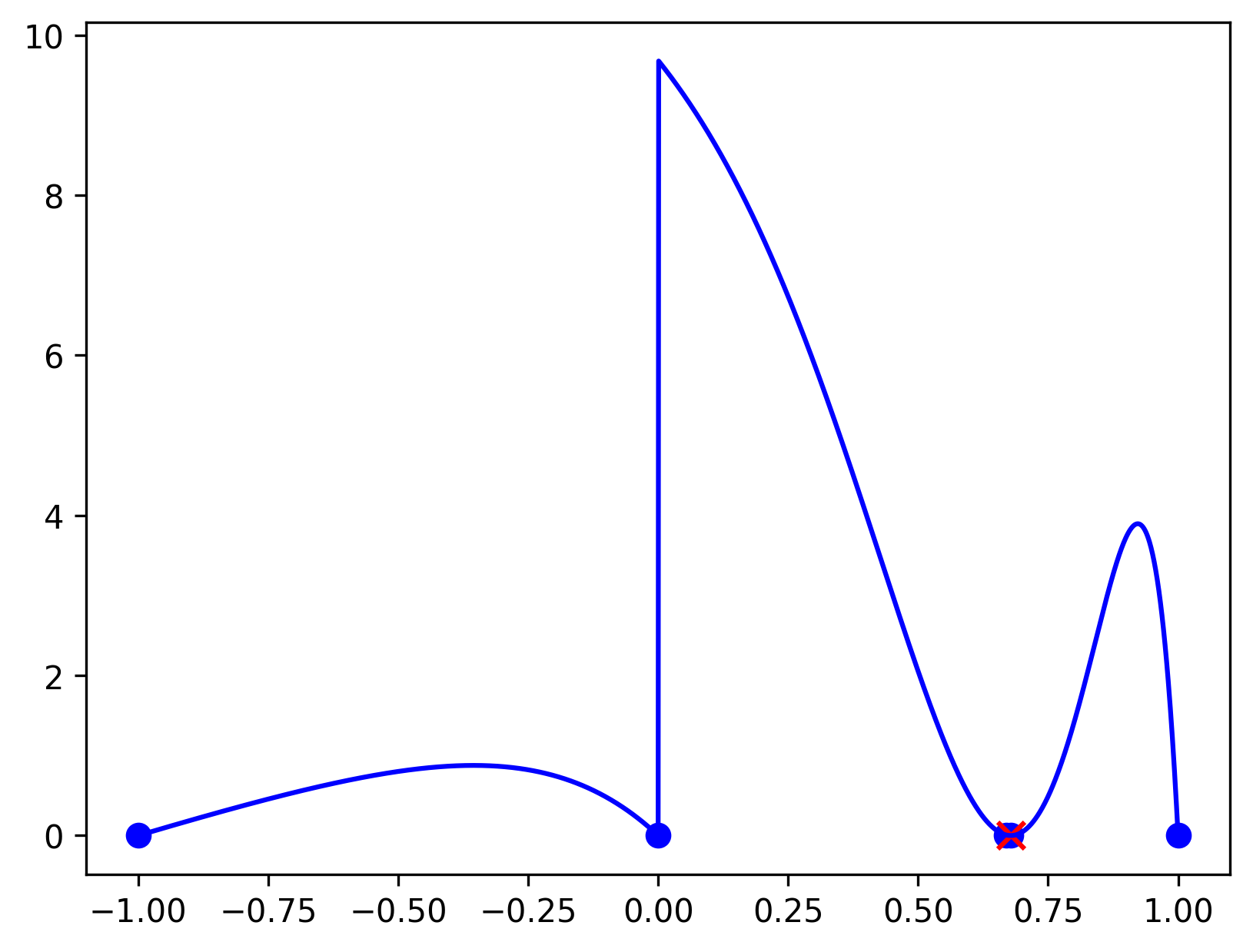}
		\subcaption{Iteration 4.}
		\label{fig:constrained-doe-iter-4}
	\end{subfigure}
	\caption{Sensitivity function $x \mapsto \psi^L(\xi^k,  \lambda^k, x)$ of the Lagrange function of~\eqref{eq:COED(X)-example-1} for the designs $\xi^k$ and the Lagrange multipliers $\lambda^k$ computed in the algorithm's iterations. The discretization set $X^k$ is indicated by the blue points and the sensitivity function's minimizer $x^k$ is marked by a red cross.}
	\label{fig:sensitivity-plots-over-iterations}
\end{figure}

In the next example, 
we also consider a non-affine inequality constraint. Specifically, we require that the value of a different design criterion be below a certain upper bound. We use the A-criterion here, which is defined by
\begin{align} \label{eq:A-criterion}
	\Phi_1(M(\xi)) 
	:= 
	\begin{cases} 
		\tr\left( M(\xi)^{-1} \right), \quad M(\xi)\ \text{is non-singular} \\
		\infty, \quad \text{else}
	\end{cases}.
\end{align}

\begin{ex}\label{ex:constrained-example-2}
In this example, we consider the constrained design problem
\begin{align}
	\label{eq:COED(X)-example-2}
	\min_{\xi \in \Xi(X)} \Psi_0(\xi) \quad 
	\text{s.t.} \quad \Psi_1(\xi) \le 0 \text{ and }
	\Psi_2(\xi) = 0
\end{align}
with the design criterion~\eqref{eq:D-criterion} and with the nonlinear inequality constraint and the linear equality constraint defined by
\begin{align}
	\Psi_1(\xi) := \Phi_1(M(\xi)) - 5
	\qquad \text{and} \qquad
	\Psi_2(\xi) := \int_X g_2(x) \d\xi(x),
\end{align}
where $\Phi_1$ is the A-criterion~\eqref{eq:A-criterion} and $g_2(x) := x + 0.5$ for $x \in \X$ as in the previous example.
It is trivial to verify Conditions~\ref{cond:objective-and-constraint-fcts} and~\ref{cond:sensi-fcts} using Lemma~\ref{lm:sufficient-conditions-for-standard-assumptions}. In order to verify Condition~\ref{cond:slater-and-non-one-sidedness}, in turn, we have to specify a suitable initial discretization $X^0$ and of a suitable design $\eta^0 \in \Xi(X^0)$ on $X^0$. We first try the discretization $X^0 := \{-1, 0\}$ we already used in the previous examples. Since all designs on this discretization are of the form~\eqref{eq:eta_alpha}, it follows from~\eqref{eq:tr(M(eta_alpha)-inverse)} that $\Psi_1(\xi) \ge 2 + \e^6 > 5$ for every $\xi \in \Xi(X^0)$ and therefore $\CD(X^0)$ is infeasible. In particular, $\CD(X^0)$ does not have any strictly feasible point and thus $X^0 := \{-1, 0\}$ cannot be used to verify Condition~\ref{cond:slater-and-non-one-sidedness}. We therefore try another initial discretization now, namely
\begin{align}
	\label{eq:X^0-constrained-example-2}
	X^0 := \{-1,0,1\}
	\qquad \text{and} \qquad
	\eta^0 := \eta_{\alpha,\beta} := 21/40 \cdot \delta_{-1} + 1/40 \cdot \delta_1 + 18/40 \cdot \delta_0
\end{align}
It then immediately follows that $X^0 \subset X$ and that 
\begin{align} \label{eq:eta^0-strictly-feasible-constrained-example-2}
	\eta^0 \in \dom \Psi_0 
	\qquad \text{and} \qquad
	\Psi_2(\eta^0) = 0
\end{align}
by virtue of~\eqref{eq:det(M(eta_alpha,beta))}.  Additionally, it follows by~\eqref{eq:tr(M(eta_alpha,beta)-inverse)} that
\begin{align}
	\tr(M(\eta^0)^{-1}) 
	\le \frac{2\alpha \e^{-6}}{\beta \gamma \e^6} + \frac{2}{\gamma} + \frac{1}{\beta \e^6}
	= \frac{280/3}{\e^{12}} + \frac{40}{9} + \frac{40}{\e^6} < 4.5 + \frac{100}{\e^{12}} + \frac{100}{\e^6} < 5
\end{align}
because $\e^6 > 400$. And finally, the range
\begin{align} \label{eq:not-one-sided-constrained-example-2}
	g_2(X^0) = \{-1/2, 1/2, 3/2\}
\end{align}
of $X^0$ under the equality-constraint-defining function $g_2$ is obviously not one-sided w.r.t.~$0$. In view of~\eqref{eq:eta^0-strictly-feasible-constrained-example-2}-\eqref{eq:not-one-sided-constrained-example-2},  Condition~\ref{cond:slater-and-non-one-sidedness} is thus satisfied.
As a consequence, 
Algorithm~\ref{algo:COED-general} with the initial discretization $X^0$ from~\eqref{eq:X^0-constrained-example-2} and with arbitrary tolerances $\ul{\delta}_k, \eps > 0$ satisfying~\eqref{eq:condition-on-tolerances-for-termination} terminates after finitely many iterations at an $\eps$-optimal design for~\eqref{eq:COED(X)-example-1} (Theorem~\ref{thm:termination-thm-for-general-algo-with-eps>0}).  Specifically, with $\eps := 10^{-3}$ and $\ul{\delta}_k := 10^{-4}$, our algorithm terminates after $3$ iterations with the $\eps^*$-optimal design $\xi^*$ from Table~\ref{tab:exponential-growth-examples} 
with $\eps^* := 4.8579 \cdot 10^{-4} < \eps$. It consists of $3$ design points, while 
\begin{align}
	\frac{d_\theta(d_\theta + 1)}{2} + |\Ieq| + 1 = 5
\end{align}
is the general upper bound on the support cardinality of $\eps$-optimal designs  for~\eqref{eq:COED(X)-example-2} from Corollary~\ref{cor:opt-design-with-finite-support}.  Figure~\ref{fig:sensitivity-plots-at-termination} (b) shows the sensitivity function $X \ni x \mapsto \psi^L(\xi^*,\lambda^*,x)$ for the saddle point $(\xi^*, \lambda^*)$ from the terminal iteration of the algorithm. In particular, this figure shows that the terminal sensitivity function's minimum $\min_{x \in X} \psi^L(\xi^*,\lambda^*,x) = -\eps^*$ is basically $0$. 
$\blacktriangleleft$ 
\end{ex}

\begin{table}[H]
	\setlength{\tabcolsep}{10pt} 
	\renewcommand{\arraystretch}{1.5}
	\caption{Computed $\eps$-optimal designs $\xi^* = \sum_{x\in\supp\xi^*} w_x \delta_x$ and algorithm performance for the experimental design problems with the exponential growth  model~\eqref{eq:exponential-model-function} (Examples~\ref{ex:unconstrained-example-1}, \ref{ex:constrained-example-1} and \ref{ex:constrained-example-2}).}
	\label{tab:exponential-growth-examples}
	\centering
	\begin{tabular}{p{4.5cm}p{3cm}p{3cm}p{3cm}}
		\hline
		& Example~\ref{ex:unconstrained-example-1} & Example~\ref{ex:constrained-example-1} & Example~\ref{ex:constrained-example-2} \\ \hline
		
		\raggedright Computed design $\xi^*$ (pairs $(x_i,w_i)$ of support points $x_i$ and weights $w_i$) & $(0.672, 0.5)$,\newline $(1.0, 0.5)$ & $(-1.0, 0.5846)$,\newline $(0.0, 0.3154)$ \newline $(0.679, 0.0285)$ \newline $ (1.0, 0.0715) $ & $(-1.0, 0.7214)$,\newline $(0.626, 0.1532)$, \newline $ (1.0, 0.1255) $\\
		\raggedright Criterion value $\Psi_0(\xi^*)$ & $-6.4162$ & $-2.6738$ & $-3.8456$ \\
		\raggedright Constraint values \newline $(\Psi_1(\xi^*), \Psi_2(\xi^*) )$ & $-$ & $(0.0, 0.0)$ & $(-2.6412, 0.0)$  \\
		\raggedright Approximation error bound $\eps^*$ & $6.0104 \cdot 10^{-4}$  & $3.4068\cdot 10^{-5} $ &  $4.8579\cdot 10^{-4} $ \\
		\raggedright Iterations & 2 & 4 & 3\\
		\raggedright Support points & 2 & 4 & 3 \\
		\raggedright Computation time (in $\SI{}{\second}$) & 0.1244 & 0.2399 & 0.7182 \\ \hline
	\end{tabular}
\end{table}

\subsection{Applications with a reaction kinetics  model}

In this section, we consider experimental design problems for a reaction kinetics model~\cite{UcBo04, UcBo05} from chemical engineering. It describes a chain of chemical reactions 
\begin{equation}\label{eq:ucinski-chem-representation}
	A \rightleftharpoons B \rightarrow C
\end{equation}
involving three components $A, B, C$, where $B$ is the target product and $C$ is a waste product. Accordingly, we are interested in experimental designs leading to a sufficiently high target product return after a sufficiently short time span. 
As in \cite{UcBo04}, we assume that the forward reactions in~\eqref{eq:ucinski-chem-representation} 
have reaction order $2$, that the backward reaction in~\eqref{eq:ucinski-chem-representation} has reaction order $1$, and that the rate constants $k_i$ of all three kinetic reactions  depend on the temperature $T$ according to the Arrhenius relation
\begin{align}
	\label{eq:arrhenius-relation}
	k_i(T) := \alpha_i \exp\left(-E_i/(RT)\right),
\end{align}
where $R := \SI{ 1.986}{\calorie \cdot \mole^{-1} \kelvin^{-1}}$ is the universal gas constant and $\alpha_i$ and $E_i$ denote the pre-exponential factor and the activation energy of the $i$th reaction, respectively. 
We denote the molar composition of the reactive mixture by 
\begin{align}
	s := (s_1,s_2,s_3) := (a,b,c),
\end{align}
that is, $a, b$ and $c$ represent the molar fractions of the components $A, B$ and $C$. 
As a consequence of the above model assumptions, the evolution of the reactive system can be described by a system of ordinary differential equations of the form
\begin{align}
	\label{eq:ucinski-ivp}
	\dot{s}(t) = g(s(t),T,\theta)
	\qquad \text{and} \qquad
	s(0) = s_0
\end{align}
where the right-hand side $g: \St \times \T \times \Theta \to \R^3$ is defined by
\begin{align}
	g_1(s,T,\theta) &:= -k_1(T,\theta) s_1^2 + k_3(T,\theta) s_2\\
	g_2(s,T,\theta) &:= k_1(T,\theta) s_1^2 - k_2(T,\theta) s_2^2 - k_3(T,\theta) s_2\\
	g_3(s,T,\theta) &:= k_2(T,\theta) s_2^2
\end{align}
and where $k_i(T,\theta) := k_i(T)$ is the model from~\eqref{eq:arrhenius-relation} and the model parameters $\theta$ are given by
\begin{align}
	\theta := (\alpha_1,\alpha_2,\alpha_3,E_1,E_2,E_3)
\end{align}
the pre-exponential factors and the activation energies of the three reactions. 
As the input quantities defining individual experiments we take 
\begin{align}
	\label{eq:ucinski-state-variables}
	x := (\tm, s_0, T) \in \X := [0,\tfin] \times \St_0 \times \T \subset \R^5,
\end{align}
that is, the measurement time $\tm$ (measured in $\SI{}{\hour}$), the initial molar composition $s_0 = (a_0,b_0,c_0)$ (measured in $\SI{}{\mole / \mole}$), and the temperature $T$ (measured in $\SI{}{\kelvin}$). We take the final time to be $\tfin := 10$ and the sets of possible initial compositions and of possible temperatures to be
\begin{gather}
	\St_0 := \big\{ (a_0,b_0,c_0) \in [0.5, 1.0] \times [0.1, 0.7]^2: a_0 + b_0 + c_0 = 1 \big\},\\
	\T := [300, 700].
\end{gather} 
What we are interested in here is the prediction of the composition $s$ of the mixture at the measurement time $\tm$ as a function of the aforementioned input quantities $x$ and model paremeters $\theta$. In other words, the output quantity of interest is $y := s$ and the prediction model of interest is the function $f: \X \times \Theta \to \R^3$ given by
\begin{align}
	\label{eq:reaction-kinetics-model}
	f(x,\theta) := s(\tm,s_0,T,\theta) = s(x,\theta),
\end{align}
where $s(\cdot,s_0,T,\theta): [0,\tfin] \to \R^3$ denotes the solution of the initial-value problem~\eqref{eq:ucinski-ivp} for given initial composition $s_0$, temperature $T$, and model parameters $\theta$. 
As already pointed above, we take the design space $X$ to be a large finite subset of $\X$, namely the uniform grid
\begin{align} \label{eq:chemical-engineering-design-space}
	X := \big\{(\tm,s_0,T) 
	&\in \{1, 2, 3, \ldots, 10\} 
	\times \{0.5, 0.51, 0.52, \ldots, 1.0 \} 
	\times \{0.1, 0.11, 0.12, \ldots, 0.7 \}^2  \notag \\
	&\quad \times \{300, 301, 302, \ldots, 700\}: a_0 + b_0 + c_0 = 1 \big\} \subset \X.
\end{align}
consisting of $1\,988\,960$ grid points in $\X$. Also, the reference parameter value $\ol{\theta}$ and the covariance matrix $\varsigma(x)$ of the measurement errors we define in the same way as~\cite{UcBo04}, namely
\begin{gather}
	\ol{\theta} := \left(\ol{\alpha}_1, \ol{\alpha}_2, \ol{\alpha}_3, \ol{E}_1, \ol{E}_2, \ol{E}_3\right) := (0.7, 0.2, 0.1, 1000, 1000, 1000),\\
	\varsigma(x) := \frac{1}{100} \begin{pmatrix}
		s_1(x,\ol{\theta}) & 0 & 0 \\
		0 & s_2(x,\ol{\theta}) & 0 \\
		0 & 0 & s_3(x,\ol{\theta})
	\end{pmatrix}.
\end{gather}
In contrast to the exponential growth model~\eqref{eq:exponential-model-function}, the reaction kinetics model~\eqref{eq:reaction-kinetics-model} is not explicitly defined by a closed-form expression but implicitly defined by the solution $s(\cdot,s_0,T,\theta)$ of the initial-value problem~\eqref{eq:ucinski-ivp}. Accordingly, the Jacobians $D_\theta f(x,\ol{\theta})$ and the one-point information matrices $m_f(x,\ol{\theta})$ from~\eqref{eq:one-point-information-matrix} are not explicitly defined either. In order to compute them, we exploit the well-known fact that the differentiated solution $t \mapsto D_\theta s(t,s_0,T,\ol{\theta}) \in \R^{3 \times 6}$ is the solution of the linearized initial-value problem
\begin{align}
	\label{eq:ucinski-differentiated-ivp}
	\dot{s}_\theta(t) = D_s g\big(s(t,s_0,T,\ol{\theta}), T, \ol{\theta}\big) s_\theta(t) + D_\theta g\big(s(t,s_0,T,\ol{\theta}), T, \ol{\theta}\big)
	\qquad \text{and} \qquad
	s_\theta(0) = 0
\end{align}
(Corollary 9.3 of~\cite{Amann}), whose right-hand side depends on the solution of the initial-value problem~\eqref{eq:ucinski-ivp}. Instead of solving~\eqref{eq:ucinski-ivp} and~\eqref{eq:ucinski-differentiated-ivp} sequentially, we solve them simultaneously, that is, we solve the initial-value problem
\begin{gather}
	\label{eq:ucinski-overall-ivp}
	\begin{pmatrix}
		\dot{s}(t)\\
		\dot{s}_\theta(t)
	\end{pmatrix}
	=
	\begin{pmatrix}
		g(s(t),T,\ol{\theta})\\
		D_s g\big(s(t), T, \ol{\theta}\big) s_\theta(t) + D_\theta g\big(s(t), T, \ol{\theta}\big)
	\end{pmatrix}
	\quad \text{and} \quad
	\begin{pmatrix}
		s(0)\\
		s_\theta(0)
	\end{pmatrix}
	=
	\begin{pmatrix}
		s_0\\
		0
	\end{pmatrix}.
\end{gather} 
In order to do so, we use the solution method $\mathrm{solve\_ivp}$ 
from the scipy.integrate package~\cite{SciPy}. With the solution $t \mapsto (s(t),s_\theta(t))$ of~\eqref{eq:ucinski-overall-ivp} at hand, we then obtain the required Jacobians and one-point information matrices as $D_\theta f(x,\ol{\theta}) = s_\theta(\tm)$ and 
\begin{align}
	m_f(x,\ol{\theta}) = 100 \cdot s_\theta(\tm)^\top \diag(s_1(\tm), s_2(\tm), s_3(\tm))^{-1} s_\theta(\tm)
\end{align}
for all design points $x \in X$. 
As in the examples with the exponential growth model above, the verification of Conditions~\ref{cond:objective-and-constraint-fcts} and~\ref{cond:sensi-fcts} will again be achieved with the help of Lemma~\ref{lm:sufficient-conditions-for-standard-assumptions}. Condition~\ref{cond:slater-and-non-one-sidedness}, in turn, will be verified using an initial discretization 
\begin{align}
	\label{eq:ucinski-initial-discretization}
	X^0 \subset X_\constr := \big\{x \in X: x_1 < 5 \text{ and } \ret(x) > 4 \big\},
\end{align}
where $x_1 = \tm$ is the measurement time as in~\eqref{eq:ucinski-state-variables} and where 
\begin{align}
	\ret(x) := f_2(x,\ol{\theta})/x_3 = s_2(x,\ol{\theta})/s_{02} = b(\tm,s_0,T,\ol{\theta})/b_0
\end{align}
is the return-on-invest of the target product $B$ at time $\tm$. So, in words, $X_\constr$ is the set of those candidate experiments from $X$ that are taken before $5$ hours have elapsed and result in a target-product return of more than $4$. As it turns out, $X_\constr$ consists of $852$ points and we choose $X^0$ as a $20$-element subset of $X_\constr$ for which the information matrix $\frac{1}{20} \sum_{x \in X^0} m_f(x,\ol{\theta})$ is invertible. We further choose $\eta^0$ as the design with uniformly distributed weights on $X^0$ so that
\begin{align}
	\label{eq:ucinski-strictly-feasible-design}
	\eta^0 := \frac{1}{20} \sum_{x \in X^0} \delta_x
	\qquad \text{and} \qquad
	\Psi_0(\eta^0) < \infty.
\end{align}
With our specific choice of $X^0$, we have $\Psi_0(\eta^0) \approx 43.31$. 

\begin{table}
	\begin{center}
		\renewcommand{\arraystretch}{1.2}
		\caption{Support points $x$ and weights $w_x$ of the computed $\eps$-optimal design $\xi^* = \sum_{x \in \supp\xi^*} w_x \delta_x$ for the unconstrained design  problem~\eqref{eq:unconstrained-oed}, along with the corresponding mole fraction predictions $f_i(x,\ol{\theta}) = s_i(x,\ol{\theta})$ and product return predictions $\ret(x)$.}
		\label{tab:unconstrained-optimal-experimental-design}
		\begin{tabular}{cccccccc} 
			\hline
			$x = (\tm, a_0, b_0, c_0, T)$ & $w_x$ & & $f_1(x,\ol{\theta})$ & $f_2(x,\ol{\theta})$ & $f_3(x,\ol{\theta})$ & & $\ret(x)$ \\
			\hline
			$(5, 0.8, 0.1, 0.1, 300)$ & $0.1290$ & & 0.542 & 0.346 & 0.112 & & 3.4563\\ 
			$(10, 0.8, 0.1, 0.1, 300)$ & $0.0581$ & & 0.429 & 0.430 & 0.141 & & 4.2998\\ 
			$(10, 0.5, 0.4, 0.1, 300)$ & $0.3129$ & & 0.357 & 0.468 & 0.175 & & 1.1691\\ 
			$(2, 0.8, 0.1, 0.1, 700)$ & $0.0217$ & & 0.535 & 0.352 & 0.113 & & 3.5151\\ 
			$(10, 0.8, 0.1, 0.1, 700)$ & $0.2722$ & & 0.302 & 0.436 & 0.262 & & 4.3586\\ 
			$(10, 0.5, 0.4, 0.1, 700)$ & $0.2061$ & & 0.284 & 0.420 & 0.296 & & 1.0500\\ 
			\hline
		\end{tabular}
	\end{center}
\end{table}

\begin{ex}\label{ex:unconstrained-example-2}	
In this example, we consider the unconstrained design problem 
\begin{equation}
	\label{eq:unconstrained-oed}
	\min_{\xi \in \Xi(X)} \Psi_0(\xi)
\end{equation}
with the design criterion~\eqref{eq:D-criterion} for reference. 
It is trivial to verify Conditions~\ref{cond:objective-and-constraint-fcts} and~\ref{cond:sensi-fcts} using Lemma~\ref{lm:sufficient-conditions-for-standard-assumptions}. In view of~\eqref{eq:ucinski-strictly-feasible-design}, Condition~\ref{cond:slater-and-non-one-sidedness} is satisfied as well.
As a consequence, 
Algorithm~\ref{algo:COED-general} with the initial discretization $X^0$ from~\eqref{eq:ucinski-initial-discretization} and with arbitrary tolerances $\ul{\delta}_k, \eps > 0$ satisfying~\eqref{eq:condition-on-tolerances-for-termination} terminates after finitely many iterations at an $\eps$-optimal design for~\eqref{eq:unconstrained-oed} (Theorem~\ref{thm:termination-thm-for-general-algo-with-eps>0}).  Specifically, with $\eps := 10^{-3}$ and $\ul{\delta}_k := 10^{-4}$, our algorithm terminates after $8$ iterations with the $\eps^*$-optimal design $\xi^*$ from Table~\ref{tab:unconstrained-optimal-experimental-design} 
with $\eps^* := 4.8579 \cdot 10^{-4} < \eps$. It consists of $6$ design points, which is well below the general upper bound 
\begin{align}
	\frac{d_\theta(d_\theta + 1)}{2} +1 = 22
\end{align}
on the support cardinality of $\eps$-optimal designs for~\eqref{eq:unconstrained-oed} from Corollary~\ref{cor:opt-design-with-finite-support}. See  Table~\ref{tab:performance-constrained-oed-2}. 
	$\blacktriangleleft$
\end{ex}

In the next example, we add performance and cost constraints on the designs stipulating that the product return when averaged over all design points and the elapsed time averaged over all design points be above $4$ or below $5$, respectively. 

\begin{ex}\label{ex:constrained-example-3}
In this example, we consider the constrained design problem
\begin{align}
	\label{eq:constrained-example}
	\min_{\xi \in \Xi(X)} \Psi_0(\xi) \quad 
	\text{s.t.} \quad \Psi_1(\xi) \le 0 \text{ and }
	\Psi_2(\xi) \le 0
\end{align}
with the design criterion~\eqref{eq:D-criterion} and with the affine inequality constraints  defined by
\begin{align}
	\Psi_1(\xi) & := \int_X \ 4 - \ret(x) \d\xi( x), \\
	\Psi_2(\xi) & := \int_X x_1 - 5 \d\xi(x).
\end{align}
Clearly, the constraints in~\eqref{eq:constrained-example} represent the aforementioned conditions on the average product return and the average measurement time. 
It is trivial to verify Conditions~\ref{cond:objective-and-constraint-fcts} and~\ref{cond:sensi-fcts} using Lemma~\ref{lm:sufficient-conditions-for-standard-assumptions}. In view of~\eqref{eq:ucinski-initial-discretization} and~\eqref{eq:ucinski-strictly-feasible-design}, Condition~\ref{cond:slater-and-non-one-sidedness} is also satisfied. 
As a consequence, 
Algorithm~\ref{algo:COED-general} with the initial discretization $X^0$ from~\eqref{eq:ucinski-initial-discretization} and with arbitrary tolerances $\ul{\delta}_k, \eps > 0$ satisfying~\eqref{eq:condition-on-tolerances-for-termination} terminates after finitely many iterations at an $\eps$-optimal design for~\eqref{eq:unconstrained-oed} (Theorem~\ref{thm:termination-thm-for-general-algo-with-eps>0}).  Specifically, with $\eps := 10^{-3}$ and $\ul{\delta}_k := 10^{-4}$, our algorithm terminates after $9$ iterations with the $\eps^*$-optimal design $\xi^*$ from Table~\ref{tab:constrained-optimal-experimental-design} 
with $\eps^* := 4.8579 \cdot 10^{-4} < \eps$. It consists of $6$ design points, which is well below the general upper bound 
\begin{align}
	\frac{d_\theta(d_\theta + 1)}{2} + |I_{\mathrm{ineq},2}| + 1 = 24
\end{align}
on the support cardinality of $\eps$-optimal designs for~\eqref{eq:constrained-example} from Corollary~\ref{cor:opt-design-with-finite-support}. See  Table~\ref{tab:performance-constrained-oed-2}.
$\blacktriangleleft$
\end{ex}

\begin{table}
\begin{center}
	\renewcommand{\arraystretch}{1.2}
	\caption{Support points $x$ and weights $w_x$ of the computed $\eps$-optimal design $\xi^* = \sum_{x \in \supp\xi^*} w_x \delta_x$ for the constrained design  problem~\eqref{eq:constrained-example}, along with the corresponding mole fraction predictions $f_i(x,\ol{\theta}) = s_i(x,\ol{\theta})$ and product return predictions $\ret(x)$.}
	\label{tab:constrained-optimal-experimental-design}
	\begin{tabular}{cccccccc} 
		\hline
		$x = (\tm, a_0, b_0, c_0, T)$ & $w_x$ & & $f_1(x,\ol{\theta})$ & $f_2(x,\ol{\theta})$ & $f_3(x,\ol{\theta})$ & & $\ret(x)$ \\
		\hline
		$(4, 0.8, 0.1, 0.1, 300)$ & $0.0807$ & & 0.577 & 0.315 & 0.108 & & 3.1503 \\ 
		$(10, 0.8, 0.1, 0.1, 300)$ & $0.0606$ & & 0.429 & 0.430 & 0.141 & & 4.2998 \\ 
		$(10, 0.5, 0.4, 0.1, 300)$ & $0.0458$ & & 0.357 & 0.468 & 0.175 & & 1.1691 \\ 
		$(3, 0.8, 0.1, 0.1, 700)$ & $0.3281$ & & 0.469 & 0.404 & 0.127 & & 4.0421 \\ 
		$(4, 0.8, 0.1, 0.1, 700)$ & $0.3699$ & & 0.422 & 0.434 & 0.144 & & 4.3374 \\ 
		$(10, 0.8, 0.1, 0.1, 700)$ & $0.1150$ & & 0.302 & 0.436 & 0.262 & & 4.3586 \\ 
		\hline
	\end{tabular}
\end{center}
\end{table}

\begin{table}[H]
\begin{center}
	\renewcommand{\arraystretch}{1.4}
	\caption{Algorithm performance for the experimental design problems for the reaction kinetics model~\eqref{eq:reaction-kinetics-model}.}
	\label{tab:performance-constrained-oed-2}
		\begin{tabular}{p{5.5cm}p{0.1cm}p{3cm}p{3cm}} 
			\hline
			& & Example~\ref{ex:unconstrained-example-2} & Example~\ref{ex:constrained-example-3}   \\
			\hline
			Criterion value $\Psi_0(\xi^*)$& & $22.8570$& $\phantom{-}27.7649$   \\
			Constraint value $\Psi_1(\xi^*)$ & & $1.4595$ & $-6.9389\cdot 10^{-18}$ \\
			Constraint value $\Psi_2(\xi^*)$ && $4.1813$ & $\phantom{-}0.0$  \\
			\raggedright Approximation error bound &  & $2.6459\cdot 10^{-6}$ & $\phantom{-}3.3856\cdot 10^{-5}$  \\
			\raggedright Iterations & & $8$ & $\phantom{-}9$ \\
			Support points & & $6$ & $\phantom{-}6$  \\
			Computation time (in $\SI{}{\second}$) &  & $301.01$ s & $\phantom{-}302.72$ s  \\
			\hline
		\end{tabular}
\end{center}
\end{table}


\section*{Acknowledgments}

We gratefully acknowledge funding from the Deutsche Forschungsgemeinschaft (DFG, German Research Foundation) through the research unit ``FOR 5359: Deep Learning on Sparse Chemical Process Data''.

\begin{small}

\end{small}


\begin{thebibliography}{}

\bibitem{Amann} H. Amann: Ordinary differential equations. An introduction to nonlinear analysis. DeGruyter (1990)

\bibitem{AmannEscher} H. Amann, J. Escher: Analysis I, II, III. Birkh\"auser (2005, 2008, 2009)

\bibitem{AtFe88} A.C. Atkinson, V.V. Fedorov: \emph{The optimum design of experiments in the presence of uncontrolled variability and prior information.} In: Y. Dodge, V.V. Fedorov, H. Wynn (eds.): Optimal design and analysis of experiments. North-Holland (1988), 327-344

\bibitem{AtDoTo} A.C. Atkinson, A.N. Donev, R.D. Tobias: Optimum experimental design, with SAS. Oxford University Press (2007)

\bibitem{At73} C. Atwood: \emph{Sequences converging to D-optimal designs of experiments.} Ann. Statist.~\textbf{1} (1973), 342-352

\bibitem{Bhatia} R. Bhatia: Matrix analysis. Springer (1997)

\bibitem{BlFa76} J.W. Blankenship, J.E. Falk: \emph{Infinitely constrained optimization problems.} J. Optim. Th. Appl.~\textbf{19} (1976), 261-281

\bibitem{Bo86} D. B\"ohning: \emph{A vertex-exchange method in D-optimal design theory.} Metrika~\textbf{33} (1986), 337-347

\bibitem{BuSc24} M. Bubel, J. Schmid, V. Kozachynskyi, E. Esche, M. Bortz: \emph{Sequential optimal experimental design for vapor-liquid equilibrium modeling.} Chem. Eng. Sci.~300 (2024), 120566

\bibitem{CaCu68} M.D. Canon, C.D. Cullum: \emph{A tight upper bound on the rate of convergence of the Frank-Wolfe algorithm.} SIAM J. Contr.~\textbf{6} (1968), 509-516

\bibitem{CoTh80} D. Cook, L.A. Thibodeau: \emph{Marginally restricted D-optimal designs.} Amer. Statist. Assoc.~\textbf{75} (1980), 366-371

\bibitem{CoFe95} D. Cook, V.V. Fedorov: \emph{Constrained optimization of experimental design.} Statistics \textbf{26} (1995), 129-178

\bibitem{DaGiMa09} P. Daniele, S. Giuffr\'{e}, A. Maugeri: \emph{Remarks on general infinite dimensional duality with cone and equality constraints.} Comm. Appl. Math.~\textbf{13} (2009), 567-578

\bibitem{Do11} M.B. Donato: \emph{The infinite dimensional Lagrange multiplier rule for convex optimization problems.} J. Funct. Anal.~\textbf{261} (2011), 2083-2093

\bibitem{DuAt25} B.P.M. Duarte, A.C. Atkinson:  \emph{Algorithms for optimal design of experiments.} In: M. Lovric (ed.), International Encyclopedia of Statistical Science. Springer (2025) 

\bibitem{DupuisEllis} P. Dupuis, R.S. Ellis: A weak convergence approach to the theory of large deviations. Wiley (1997)

\bibitem{Fe} V.V. Fedorov: Theory of optimal experiments. Academic Press (1972)

\bibitem{Fe89} V.V. Fedorov: \emph{Optimal design with bounded density: optimization algorithms of the exchange type. J. Statist. Plann. Inference~\textbf{22} (1989), 1-13}

\bibitem{FeLe} V.V. Fedorov, S.L. Leonov: Optimal design for nonlinear response models. Chapman \& Hall (2014)

\bibitem{Ga86} A. Gaivoronski: \emph{Linearization methods for optimization of functionals which depend on probability measures.} In: A. Pr\'{e}kopa, R.J.B. Wets (eds.): Stochastic Programming 84 Part II. Mathematical Programming Studies~\textbf{28} (1986), 157-181

\bibitem{HaFiRi20} R. Harman, L. Filov\'{a}, P. Richt\'{a}rik: \emph{A ranomized exchange algorithm for computing optimal approximate designs of experiments.} J. Am. Statist. Assoc.~\textbf{115} (2020), 348-361

\bibitem{HeLe20} R. Herzog, E. Legler: \emph{First-order methods for optimal experimental design problems with bound constraints.} arXiv:2004.08084 (2020)

\bibitem{HuHs93} M.L. Huang, M.C. Hsu: \emph{Marginally restricted linear-optimal designs.} J. Statist. Plann. Inference~\textbf{35} (1993), 251-266

\bibitem{Ja13} M. Jaggi: \emph{Revisiting Frank–Wolfe: projection-free sparse convex optimization.} 
Proceedings of the 30th International Conference on Machine Learning (ICML), PMLR (2013), 427-435

\bibitem{Jahn} J. Jahn: Introduction to the theory of nonlinear optimization. 3rd edition, Springer (2007)

\bibitem{Ki74} J. Kiefer: \emph{General equivalence theory for optimum designs (approximate theory).} Ann. Stat.~\textbf{2} (1974), 849-879

\bibitem{LaJa15} S. Lacoste-Julien, M. Jaggi: \emph{On the global linear convergence of Frank-Wolfe optimization variants.} Advances in Neural Information Processing Systems~\textbf{28} (2015), 496-504

\bibitem{Mangasarian} O.L. Mangasarian: Nonlinear programming. McGraw-Hill (1994)

\bibitem{MagnusNeudecker} J.R. Magnus, H. Neudecker: Matrix differential calculus with applications in statistics and econometrics. 3rd edition, Wiley (2007)

\bibitem{MoZu02} I. Molchanov, S. Zuyev: \emph{Steepest descent algorithms in a space of measuers.} Stat. Comp.~\textbf{12} (2002), 115-123 

\bibitem{MoZu04} I. Molchanov, S. Zuyev: \emph{Optimisation in space of measures and optimal design.} ESAIM Prob. Stat.~\textbf{8} (2004), 12-24

\bibitem{Parthasarathy} K.R. Parthasarathy: Probability measures on metric spaces. Academic Press (1967)

\bibitem{Pr04} L. Pronzato: \emph{A minimax equivalence theorem for optimum bounded design measures.} Stat. Prob. Lett.~\textbf{68} (2004), 325-331

\bibitem{PrPa} L. Pronzato, A. P\'{a}zman: Design of experiments in nonlinear models. Asymptotic normality, optimality criteria and small-sample properties. Springer (2013)

\bibitem{Pu} F. Pukelsheim: Optimal design of experiments. SIAM Classics of Applied Mathematics (2006)

\bibitem{Ro76} S.M. Robinson: \emph{Stability theory for systems of inequalities in nonlinear programming, part II: differentiable nonlinear systems.} SIAM J. Numer. Anal.~\textbf{13} (1976), 497-513

\bibitem{Rockafellar} R.T. Rockafellar: Convex analysis. Princeton University Press (1970)

\bibitem{Sc11} K. Schittkowski: \emph{A robust implementation of a sequential quadratic programming algorithm with successive error restoration.} Opt. Lett.~\textbf{5} (2011), 283–296

\bibitem{Sc14} K. Schittkowski: \emph{NLPQLP: A Fortran implementation of a sequential quadratic programming algorithm with distributed non-monotone line search.} User’s guide, Version 4.2 (2014)

\bibitem{ScSeBo} J. Schmid, P. Seufert, M. Bortz: \emph{Adaptive discretization algorithms for locally optimal experimental design.} arXiv:2406.01541 (2024)

\bibitem{SiTiTo78} S.D. Silvey, D.H. Titterington, B. Torsney: \emph{An algorithm for optimal designs on a design space.} Comm. Statist. Theory Methods~\textbf{14} (1978), 1379-1389

\bibitem{Silvey} S.D. Silvey: Optimal design. An introduction to the theory for parameter estimation. Chapman \& Hall (1980)

\bibitem{UcBo04} D. Uci\'{n}ski, B. Bogacka: \emph{T-optimum designs for multiresponse dynamic heteroscedastic models.} In: A. Di Bucchianico, H. L\"auter, H.P. Wynn (eds.): Contributions to statistics. Springer (2004), 191-199

\bibitem{UcBo05} D. Uci\'{n}ski, B. Bogacka: \emph{T-optimum designs for discrimination between two multiresponse dynamic models.} J. R. Statist. Soc. B~\textbf{67} (2005), 3-18

\bibitem{VaSe21} C. Vanaret, P. Seufert, J. Schwientek, G. Karpov, G. Ryzhakov, I. Oseledets, N. Asprion, M. Bortz:  \emph{Two-phase approaches to optimal model-based design of experiments: how many experiments and which ones?} Comp. Chem. Eng.~\textbf{146} (2021), 107218

\bibitem{SciPy} P. Virtanen, R. Gommers, T.E. Oliphant, M. Haberland, T. Reddy, D. Cournapeau, E. Burovski, P. Peterson, W. Weckesser, J. Bright, S.J. van der Walt, M. Brett, J. Wilson, K.J. Millman, N. Mayorov, A.R.J. Nelson, E. Jones, R. Kern, E. Larson, C.J. Carey, {\.I}. Polat, Y. Feng, E.W. Moore, J. VanderPlas, D. Laxalde, J. Perktold, R. Cimrman, I. Henriksen, A.E. Quintero, C.R. Harris, A.M. Archibald, A. Ribeiro, F. Pedregosa, P. van Mulbregt,  SciPy 1.0 Contributors: \emph{SciPy 1.0: Fundamental algorithms for scientific computing in Python.} Nature Methods~\textbf{17} (2020), 261-272

\bibitem{Wy70} H.P. Wynn: \emph{The sequential generation of D-Optimal experimental designs.} Ann. Math.  Statist.~\textbf{41} (1970), 1655-1664

\bibitem{YaBiTa13} M. Yang, S. Biedermann, E. Tang: \emph{On optimal designs for nonlinear models: a general and efficient algorithm.} J. Am. Statist. Assoc.~\textbf{108} (2013), 1411-1420

\bibitem{Yu10} Y. Yu: \emph{Monotonic convergence of a general algorithm for computing optimal designs.} The Annals of Statistics~\textbf{38} (2010), 1593

\bibitem{Yu11} Y. Yu: \emph{D-optimal designs via a cocktail algorithm.} Stat. Comp.~\textbf{21} (2011), 475-481

\bibitem{ZoKu79} J. Zowe, S. Kurcyusz: \emph{Regularity and stability for the mathematical programming problem in Banach spaces.} Appl. Math. Optim.~\textbf{5} (1979), 49-62

\end{thebibliography}
\end{document}